\documentclass[11pt, reqno]{amsart}

\usepackage[utf8]{inputenc}
\usepackage[margin = 1in]{geometry}
\usepackage{amsmath}
\usepackage{amsthm}
\usepackage{amssymb}
\usepackage{natbib}
\usepackage{xcolor}
\usepackage{algorithm}
\usepackage{algpseudocode}
\usepackage{comment}
\usepackage{cleveref}
\usepackage{appendix}
\usepackage{graphicx}
\usepackage{url}
\usepackage{mathrsfs}
\newtheorem{theorem}{Theorem}[section]
\newtheorem{lemma}{Lemma}[section]
\newtheorem{remark}{Remark}[section]
\newtheorem{proposition}{Proposition}[section] 
\newtheorem{corollary}{Corollary}[section]
\newtheorem{assumption}{Assumption}[section]
\newtheorem{example}{Example}[section]

\newtheorem{definition}{Definition}[section]

\newtheorem{result}{Result}[section]
\begin{document}
\title[Unadjusted Barker]{Skew-symmetric schemes for stochastic differential equations with non-Lipschitz drift: an unadjusted Barker algorithm}
\author[Yuga Iguchi, Samuel Livingstone, Nikolas N{\"u}sken, Giorgos Vasdekis and Rui-Yang Zhang]{}

\date{}

\maketitle
\vspace{-10pt}
\begin{center}
    \textsc{Yuga Iguchi} \\
    School of Mathematical Sciences,Lancaster University, Lancaster, U.K..\\
    \texttt{y.iguchi@lancaster.ac.uk} \\
    \bigskip
    \textsc{Samuel Livingstone} \\
    Department of Statistical Science, University College London,  U.K.\\
    \texttt{samuel.livingstone@ucl.ac.uk} \\
    \bigskip
    \textsc{Nikolas N{\"u}sken} \\
	Department of Mathematics, King's College London, U.K. \\ \texttt{nikolas.nusken@kcl.ac.uk}\\
    \bigskip
    \textsc{Giorgos Vasdekis} \\
	School of Mathematics, Statistics and Physics, Newcastle University, U.K. \\ \texttt{giorgos.vasdekis@newcastle.ac.uk}\\
    \bigskip
    \textsc{Rui-Yang Zhang} \\
	School of Mathematical Sciences, Lancaster University, U.K. \\
    \texttt{r.zhang26@lancaster.ac.uk}
\end{center}

\maketitle

\begin{abstract}
    We propose a new simple and explicit numerical scheme for time-homogeneous stochastic differential equations. The scheme is based on sampling increments at each time step from a skew-symmetric probability distribution, with the level of skewness determined by the drift and volatility of the underlying process.  We show that as the step-size decreases the scheme converges weakly to  the diffusion of interest. We then consider the problem of simulating from the limiting distribution of an ergodic diffusion process using the numerical scheme with a fixed step-size.  We establish conditions under which the numerical scheme converges to equilibrium at a geometric rate, and quantify the bias between the equilibrium distributions of the scheme and of the true diffusion process.  Notably, our results do not require a global Lipschitz assumption on the drift, in contrast to those required for the Euler--Maruyama scheme for long-time simulation at fixed step-sizes.  Our weak convergence result relies on an extension of the theory of Milstein \& Tretyakov to stochastic differential equations with non-Lipschitz drift, which could also be of independent interest.  We support our theoretical results with numerical simulations.
    \end{abstract}
\textbf{MSC2020 subject classifications:} 60H35, 65C05, 65C30 ,  65C40.

\section{Introduction}

We consider the problem of sampling from the equilibrium distribution of an ergodic diffusion process, which is ubiquitous in areas such as molecular simulation \citep{stoltz2010free,faulkner2024sampling}, Bayesian statistical inference and machine learning \citep{green2015bayesian,andrieu2003introduction}, as well as many others (e.g. \citep{bortoli.thornton.heng.doucet:21, sant.jenkins.koskela.spano:23}).  The classical Euler--Maruyama numerical scheme does not always approximate the dynamics of the process to a desired degree of accuracy \citep{roberts1996exponential}.  Lack of numerical stability commonly arises when the dynamics are in some sense stiff, in particular when the drift is not globally Lipschitz. The problem is discussed in \cite{higham.mao.stuart:03} and can cause issues in various application domains, which has motivated the search for alternative explicit schemes that are equally straightforward to implement \citep{hutzenthaler2012strong,sabanis2013note,mao2015truncated}. 

In this work we study explicit numerical integration schemes based on the simple idea of sampling an increment at each time step from a distribution that is skewed in the direction of the drift vector.  The scheme is inspired by a recently introduced sampling algorithm in Bayesian statistics known as the Barker proposal \citep{livingstone2022barker,hird2020fresh,vogrinc2022optimal}, which is a Metropolis--Hastings scheme in which proposals are generated at each step by approximating the dynamics of a particular Markov jump process. A Metropolis--Hastings filter is then applied to ensure that the resulting Markov chain has the correct equilibrium distribution.  The scheme is known to be extremely stable in cases when alternative approaches (often based on approximating Langevin diffusions using the Euler--Maruyama numerical scheme) can perform poorly. The starting point for this work was the realisation that the same proposal mechanism can be viewed as approximating the dynamics of an overdamped Langevin diffusion, a well-studied object about which much is known regarding long-time behaviour and stability properties (e.g. \citep{roberts1996exponential,bolley2012convergence}, see also \cite{bourabee.vandeneijnden:18} for an earlier work in this direction).

To establish this connection more formally, in Section \ref{sec:scheme} we define a general class of skew-symmetric numerical integrators for uniformly elliptic stochastic differential equations. Choosing a particular element of the class and targeting the overdamped Langevin diffusion recovers the Barker algorithm (minus the Metropolis--Hastings filter). We later refer to this approach as the \textit{unadjusted Barker algorithm}  or the \textit{Barker proposal} when applied to Langevin diffusions in Section \ref{sec:experiments}.  In Section \ref{sec:general.convergence:00} we prove that, under suitable conditions, the schemes defined in Section \ref{sec:scheme} converge weakly to the true diffusion as the step-size decreases with weak order 1.  To do this we establish a regularity result, Theorem \ref{Generalisation.theorem:1}, which allows the general convergence theorem of Milstein \& Tretyakov \cite[Theorem 2.2.1]{milstein.tretyakov:21} to be extended to a class of stochastic differential equations with non-Lipschitz drift.  In principle this theorem can also be applied to establish weak convergence of other numerical schemes, and so we believe it to be of independent interest. In Section \ref{sec:ergodicity} we consider simulation over long time scales. We prove that, under suitable assumptions, the skew-symmetric scheme has an equilibrium distribution to which it converges at a geometric rate, and additionally we characterise the bias incurred when computing expectations under the equilibrium distribution of the numerical scheme versus the true diffusion process as a function of the integrator step-size. Notably, these long-time behaviour results do not require a global Lipschitz condition on the drift, whereas the Euler--Maruyama scheme is known to be transient when the drift is not globally Lipschitz \citep{roberts1996exponential}, rendering it unsuitable for long-time simulation in this setting. This justifies to some extent the stable performance of the skew-symmetric approach, which has the property of naturally preventing large values of the drift from disrupting the stability of the numerical scheme, akin to approaches such as the tamed and truncated Euler--Maruyama methods introduced in \cite{hutzenthaler2012strong} and \cite{mao2016convergence}. In Section \ref{sec:experiments} we illustrate the performance of the scheme on a range of benchmark examples, via numerical simulations, before concluding in Section \ref{conclusion:00} with a discussion. Proofs of the main results are presented in the appendices. The supplementary material contains a conceptual comparison to the tamed and adaptive Euler approaches, and some more simulation studies.

\subsection{Notation} 
For any function $f: \mathbb{R}^{2d} \rightarrow \mathbb{R}$ and any  $(x,\nu) \in \mathbb{R}^{2d}$, with $x=(x_1,...,x_d) \in \mathbb{R}^d$ and $\nu=(\nu_1,...,\nu_d) \in \mathbb{R}^d$, we write $\partial_if(x,\nu)$ to denote the partial derivative with respect to $x_i$ and $\partial^{
\nu}_if(x,\nu)$ for that with respect to $\nu_i$. The notation $f^{(k)}$ will indicate any derivative of order $k$ with respect to components of potentially both $x$ and $\nu$. The gradient operator applied to functions $g: \mathbb{R}^d \rightarrow \mathbb{R}$ will be denoted by $\nabla$, and the Hessian by $\nabla^2$.

For a matrix-valued function $\sigma:\mathbb{R}^d \rightarrow \mathbb{R}^{d \times d}$ evaluated at $x \in \mathbb{R}^d$ we write $\sigma_{i,j}(x)$ to denote the element in the $i$th row and $j$th column. Similarly the $i$th element of a vector-valued function $\mu$ evaluated at $x$ is denoted $\mu_i(x)$. The $d \times d$ identity matrix is written $I_d$. 

Let $\| \cdot \|$ denote the Euclidean norm on $\mathbb{R}^d$ when applied to vectors or the Frobenius norm when applied on matrices on $\mathbb{R}^{d \times d}$. More generally, we write 
\begin{equation}
\Vert A \Vert_F := \sqrt{\sum_{i_1,\ldots,i_n = 1}^d (A^2_{i_1,\ldots,i_n})}, \qquad A \in \mathbb{R}^{(d^n)},
\end{equation}
to denote the Frobenius norm of a multidimensional array,
which reduces to the Euclidean norm if $n=1$ and to the standard matrix Frobenius norm if $n=2$.

The supremum norm on $\mathbb{R}^d$ is denoted $\| \cdot \|_{\infty}$. For two vectors $c=(c_1,...,c_d)$ and $c'=(c_1',...,c_d')$ in $\mathbb{R}^d$, we write $c \cdot c'$ to denote their scalar product, and $c \ast c' := (c_1c'_1,c_2c'_2,...,c_dc'_d)$ to denote their element-wise product. For a matrix $A \in \mathbb{R}^{d \times d}$ and a vector $c \in \mathbb{R}^d$, $A c$ denotes the usual multiplication between matrix and vector, and $A^T$ denotes the transpose of $A$. For a probability measure $\pi$ on $\mathbb{R}^d$, $\langle \cdot , \cdot \rangle_{\pi}$ denotes the usual inner product in $L^2(\pi)$.

Let $C^l(\mathbb{R}^d)$ denote the set of $l$-times differentiable functions $f: \mathbb{R}^d \rightarrow \mathbb{R}$, with continuous derivatives up to order $l$.
We also denote
\begin{align*}
C_P^{l,m}(\mathbb{R}^d) = \left\{ f \in C^l(\mathbb{R}^d): \,\, \exists C>0 \text{ s.t. } |f^{(k)}(x)| \le C(1 + \|x\|^m), \forall x \in \mathbb{R}^d, k \leq l  \right\},
\end{align*}
as well as
\begin{equation}\label{CPl.definition:1}
C^l_P(\mathbb{R}^d)=\bigcup_{m \in \mathbb{N}}C^{l,m}_P \qquad \text{and} \qquad     C^{\infty}_P(\mathbb{R}^d)=\bigcap_{l \in \mathbb{N}}C^{l}_P.
\end{equation}

For $a,b \in \mathbb{R}$, we write $a \wedge b := \min\{ a ,b \}$ and $a \vee b := \max\{ a , b \}$.

Finally, $\mathbb{E}_x$ denotes expectation with respect to the law of a stochastic process conditioned on the initial value $x$. $\mathbb{I}_A$ denotes the indicator function of the set $A$.

\section{A skew-symmetric numerical scheme}
\label{sec:scheme}

\subsection{Intuition in one dimension}\label{one.dimension:00}
We consider the one-dimensional stochastic differential equation (SDE)
\begin{equation}\label{eqn:SDE}
\mathrm{d}Y_t = \mu(Y_t) \, \mathrm{d}t + \sigma(Y_t) \, \mathrm{d}W_t,\qquad \qquad Y_0 = x_0,
\end{equation}
where $\mu:\mathbb{R} \rightarrow \mathbb{R}$ and $\sigma: \mathbb{R} \rightarrow \mathbb{R}_{>0}$ are drift and volatility coefficients of sufficient regularity (see Section \ref{sec:results} below), $x_0 \in \mathbb{R}$ is a deterministic initial condition, and $(W_t)_{t \ge 0}$ is a standard Brownian motion. Numerical approximations of $(Y_t)_{t \ge 0}$ often rely on the \emph{Euler--Maruyama} scheme \citep{kloeden1992stochastic}
\begin{equation}
\label{eq:EM}
 X^{n+1} = X^n + \mu(X^n) \Delta t   + \sqrt{\Delta t }  \sigma(X^n) \nu^n, \qquad \qquad X^0 = x_0, 
\end{equation}
where $\Delta t > 0$ is a step size, and $(\nu^n)_{n \in \mathbb{N}}$ is an i.i.d. sequence of standard Gaussian random variables.

In this paper, we suggest the alternative iterative scheme
\begin{equation}
\label{eq:skew scheme}
X^{n+1} = X^{n} +  b^{n} \sqrt{\Delta t} \sigma(X^{n})   \nu^{n}, \qquad \qquad X^0 = x_0,
\end{equation}
with $\nu^n$ referring to Gaussian increments as before. The auxiliary random variables $b^n$ are binary,
\begin{equation}
b^{n} = \begin{cases}
	+1 &\quad \text{with probability } p(X^{n}, \nu^n), \\
	-1 &\quad \text{otherwise,}
\end{cases}
\end{equation}
introducing random sign flips of $\nu^n$ in \eqref{eq:skew scheme}, moderated by an appropriate (differentiable) function $p:\mathbb{R} \times \mathbb{R} \rightarrow [0,1]$ yet to be specified. Notice that $b^n$ and $\nu^n$ are not independent, and hence in general $\mathbb{E}[b^n \nu^n] \neq 0$. It is therefore plausible that by a judicious choice of $p$, the additional flips in \eqref{eq:skew scheme} may reproduce the drift $\mu(X^n) \Delta t $ in \eqref{eq:EM}, up to $\mathcal{O}(\Delta t^2)$. Indeed, appropriate Taylor series expansions (see Appendix \ref{proof.of.semigroup.expansion:00}) show that \eqref{eq:skew scheme} captures the dynamics \eqref{eqn:SDE} in the weak sense as $\Delta t  \rightarrow 0$, provided that $p$ satisfies 
\begin{equation}
\label{eq:p condition 1d}
\partial_{\nu} \Big\vert_{\nu = 0} p(x,\nu) = \frac{\mu(x)}{2 \sigma(x)}\sqrt{\Delta t},    
\end{equation}
for all $x \in \mathbb{R}$.
We provide precise statements in Section \ref{sec:results}.

\begin{remark}[Robustness]
As alluded to above, Taylor series expansions show that 
\begin{equation}
\mathbb{E}[X^{n+1} - X^n | X^n] = \mu(X^n) \Delta t + \mathcal{O}(\Delta t^2),
\end{equation}
both for the Euler-Mayurama scheme \eqref{eq:EM} and the skew-symmetric scheme \eqref{eq:skew scheme}, assuming \eqref{eq:p condition 1d}. This observation makes it plausible that some properties of \eqref{eq:EM} might hold as well for \eqref{eq:skew scheme}, and we confirm this intuition for the weak order (Theorem \ref{theorem.weak.convergence:0}) and approximation accuracy of invariant measures (Theorem \ref{thm.bias.equilibrium:1}). To expose qualitative differences between \eqref{eq:EM} and \eqref{eq:skew scheme}, it is instructive to consider the expected squared jump sizes
\begin{equation}
\label{eq:jump size EM}
\mathbb{E}[(X^{n+1} - X^n)^2 | X_n] = \mu^2(X^n) \Delta t^2 + \sigma^2(X^n) \Delta t
\end{equation}
for the Euler-Mayurama scheme, and
\begin{equation}
\label{eq:jump size skew}
\mathbb{E}[(X^{n+1} - X^n)^2 | X^n] =  \sigma^2(X^n) \Delta t
\end{equation}
for the skew-symmetric scheme. Indeed, \eqref{eq:jump size EM} is reflective of the fact that \eqref{eq:EM} is generally unstable for non-Lipschitz drifts \citep{hutzenthaler2012strong,sabanis2013note}. In contrast, \eqref{eq:jump size skew} shows that the expected squared jump size (but not the direction) for \eqref{eq:skew scheme} does not depend on $\mu(X^n)$, hence suppressing potential instabilities due to unbounded drifts. Both our theoretical results (allowing for polynomially growing drifts that are not globally Lipschitz) and our numerical experiments in Section \ref{sec:experiments} support the intuition that the skew-symmetric scheme \eqref{eq:skew scheme} has favourable robustness properties in the presence of strong unbounded drifts. 
\end{remark}
\subsection{Relationships to skew-symmetric distributions and sampling algorithms}\label{subsection:skew-symmetric.mcmc:00}

Whilst the relationship \eqref{eq:p condition 1d} %is -- at least formally -- sufficient to
will ensure weak convergence of \eqref{eq:skew scheme} towards \eqref{eqn:SDE}, it clearly does not uniquely determine $p$ given $\mu$ and $\sigma$. To explore and restrict the design space of $p:\mathbb{R} \times \mathbb{R} \rightarrow [0,1]$, we make the following two observations:
\begin{enumerate}
 \item The density of $\xi:=X^{n+1}-X^{n}$ given $X^n$ in (\ref{eq:skew scheme}) can be written as 
\begin{align*}
   \hspace{-1cm} \gamma(X^n, \xi)=& \left( p\left(X^n,\frac{\xi}{\sqrt{\Delta t} \sigma(X^n)}\right) + 1 - p\left(X^n,\frac{-\xi}{\sqrt{\Delta t} \sigma(X^n)} \right) \right)  \frac{1}{\sqrt{\Delta t} \sigma(X^n)} \phi\left( \frac{\xi}{\sqrt{\Delta t} \sigma(X^n)} \right),
\end{align*}
where $\phi$ is the density of a standard Normal distribution. Therefore, $p$ influences the law of the process only through the quantity $p(X^n,\nu)+1-p(X^n,-\nu)$. We observe that, if one replaces $p$ by the function $\tilde{p}(x,\nu):=\left[p(x,\nu)+1-p(x,-\nu)\right]/2$, we would induce the same law for $\xi$, while $\tilde{p}(x,0)=1/2$. Therefore, we can without loss of generality impose the further constraint \begin{equation}\label{p.0.1d:1}
p(x,0) = \frac{1}{2}, \qquad \text{for all} \,\,\, x \in \mathbb{R}.
\end{equation}

\item 
Under the assumption that $p(X_n,\cdot)$ is a cumulative distribution function (CDF)\footnote{that is, $\nu \mapsto p(X_n,\nu)$ is right-continuous and monotone increasing, with limits $\lim_{\nu \rightarrow - \infty} p(x,\nu) = 0$ and $\lim_{\nu \rightarrow - \infty} p(x,\nu) = 1$.} of a centred and symmetric real-valued random variable, the product $b_n \nu_n$ is known as a skew-symmetric random variable (see e.g. \citep{azzalini2013skew}), to be thought of as a biased modification of a standard Gaussian.
\end{enumerate}

Based on these observations, we could consider functions $p$ of the form given by the following proposition. 

\begin{proposition}\label{cdf.p:0}
Let $F$ be the CDF of a one dimensional distribution with continuous density $f$, which is symmetric under the reflection $x \mapsto -x$, and which satisfies $f(0) \neq 0$. Then \begin{equation}\label{cdf.p:1}
p(x,\nu)=F\left(\nu \Delta t^{1/2} \frac{\mu(x)}{2f(0)\sigma(x)} \right)
\end{equation} satisfies (\ref{eq:p condition 1d}) and \eqref{p.0.1d:1}. 
\end{proposition}

\begin{proof}[Proof of Proposition \ref{cdf.p:0}]
The symmetry of $f$ implies $p(x,0)=F(0)=1/2$. Using the chain rule and noting that $F'(y)=f(y)$ gives
    \begin{equation*}
        \partial_{\nu}p(x,\nu)=f\left(\nu \Delta t^{1/2} \frac{\mu(x)}{2f(0)\sigma(x)} \right) \Delta t^{1/2} \frac{\mu(x)}{2f(0)\sigma(x)},
    \end{equation*}
    and setting $\nu=0$ gives (\ref{eq:p condition 1d}).
\end{proof}

\begin{example}\label{example.cdf.p:0}
\begin{enumerate}
\item \hspace{-0.18cm} \ The \emph{Logistic} distribution has density {${f(x)\!=e\!^{-x}(1\!+\!e^{-x})^{-2}}$} and  CDF $F(x)=\left( 1+ e^{-x} \right)^{-1}$, so that (\ref{cdf.p:1}) results in
\begin{equation*}
p(x,\nu)=\frac{1}{1+\exp\left\{ -2\Delta t^{1/2}\nu \frac{\mu(x)}{\sigma(x)} \right\} }.
\end{equation*}
\item \
The \emph{standard Normal} distribution
with density $f(x)= (2\pi)^{-1/2}e^{ -x^2/2 }$ and CDF $\Phi$ induces   \begin{equation*}
p(x,\nu)=\Phi\left( \sqrt{\frac{\pi}{2}}\Delta t^{1/2}  \nu \frac{\mu(x)}{\sigma(x)} \right).
  \end{equation*}
   \end{enumerate}
\end{example}

Although restricting to CDFs is not strictly necessary in the context of \eqref{eq:skew scheme}, we believe that relaxing this assumption is unlikely to lead to numerical benefits. We emphasise, however, that our results hold beyond this class of choices.

The work of this article can also find applications in the area of sampling algorithms, for which the goal is to sample from or compute expectations with respect to a probability distribution $\pi$ of interest. A common approach is to design a diffusion with equilibrium distribution $\pi$ and then simulate the process numerically over sufficiently long time scales. This can be done in such a way that the equilibrium distribution is preserved by applying a Metropolis--Hastings filter at each time-step, leading to a Metropolis-adjusted sampling algorithm (e.g. \cite{roberts1996exponential}). It can also be done approximately without applying this filter, leading to an unadjusted sampling algorithm (e.g. \cite{durmus2017nonasymptotic}).  In both cases it is important to understand how quickly the numerical process approaches equilibrium, and in the unadjusted case it is also important to understand the level of bias introduced by the numerical discretisation at equilibrium. 

The unadjusted Langevin algorithm (ULA) is a state-of-the art example in which the overdamped Langevin diffusion is discretised using the Euler--Maruyama scheme \citep{durmus2017nonasymptotic}. It is well-known that when the drift is not globally Lipschitz, as will be the case when $\pi$ has sufficiently light tails, then ULA does not preserve the equilibrium properties of the diffusion, and instead produces a transient Markov chain.  By contrast, the results of Section \ref{sec:ergodicity} show that using the skew-symmetric numerical scheme produces an algorithm that converges to equilibrium at a geometric rate, and for which the equilibrium distribution is close to $\pi$ for sufficiently small choice of numerical step-size. We call the scheme the \emph{unadjusted Barker algorithm} in reference to both ULA and the Barker proposal of \cite{livingstone2022barker}. An alternative approach using the tamed Euler scheme is discussed in \cite{brosse2019tamed}.  We note that a version of the algorithm using stochastic gradients has recently been proposed in \cite{mauri2024robust}.

\subsection{Skew-symmetric schemes in general dimension}
Next we turn to the problem of simulating a multi-dimensional SDE of the form 
\begin{equation}\label{multi.dimensional.sde:0}
\mathrm{d}Y_t=\mu(Y_t)\,\mathrm{d}t+\sigma(Y_t)\,\mathrm{d}W_t,\quad  Y_0=x \in \mathbb{R}^d,
\end{equation}
where $\mu=(\mu_1,...,\mu_d):\mathbb{R}^d \rightarrow \mathbb{R}^d$ and $\sigma=(\sigma_{i,j})_{i,j \in \{ 1,...,d \}}: \mathbb{R}^d \rightarrow \mathbb{R}^{d \times d}$ specify the drift and volatility, respectively, and $(W_t)_{t \geq 0}=(W^1_t,...,W^d_t)_{t \geq 0}$ is a $d$-dimensional standard Brownian motion. Let us also assume that we have chosen a family of functions $p_i: \mathbb{R}^d \times \mathbb{R}^d \rightarrow (0,1)$ for all $i \in \{ 1,...,d\}$; those will be required  to satisfy appropriate modifications of \eqref{eq:p condition 1d} and \eqref{p.0.1d:1}. For details, see Section \ref{sec:results} below, where we also provide precise assumptions on the growth and regularity of $\mu$ and $\sigma$.

The idea behind the multi-dimensional skew-symmetric scheme is to view the columns of the volatility matrix $\sigma$ as a basis of $\mathbb{R}^d$ and update using a skew-symmetric jump (as introduced in Section \ref{one.dimension:00}) in every direction given by the basis. In particular, for a  fixed step size $\Delta t \in (0,1)$, we iteratively construct $X^0, X^1 ,...,$ as follows. First we set $X^0=x$ according to the initial condition in \eqref{multi.dimensional.sde:0}. Then, for any $n \in \mathbb{N}$, we generate $\nu^n=(\nu^n_1,...,\nu^n_d) \sim N(0,I_d)$, and for all $i \in \{ 1,...,d \}$ a random variable
\begin{equation}
b^{n}_i = \begin{cases}
	+1 &\quad \text{with probability } p_i(X^{n}, \nu^n), \\
	-1 &\quad \text{otherwise.}
\end{cases}
\end{equation}
There is freedom in how to choose the probabilities $p_i$. As will be seen later, one practical way to do so is to set
\begin{equation}\label{first.cdf.p.high.dimensions:1}
    p_i(x,\nu)=F\left(\nu_i \Delta t^{1/2} \frac{1}{2f(0)} \Psi_i(x) \right),
\end{equation} 
where $\Psi(x)=\sigma^{-1}(x) \mu(x)$ and $f,F$ are respectively the probability density and cumulative distribution function of a real-valued random variable such that $f$ is symmetric around zero with $f(0) \neq 0$.
We then set $b^n:=(b^n_1,...,b^n_d)$, and 
\begin{equation}\label{general.update.scheme:1}
X^{n+1}=X^n+ {\Delta t}^{1/2} \ \sigma(X^n) (b^n \ast \nu^n),
\end{equation}
recalling that $\ast$ denotes element-wise multiplication. A more intuitive way to understand the update (\ref{general.update.scheme:1}) is to write it as
\begin{equation}\label{general.update.scheme:2}
  X^{n+1}=X^n+ \sum_{i=1}^d b^n_{i} \ \sigma_{i}(X^n) \nu_i^n{\Delta t}^{1/2},
\end{equation}
where $\sigma_i$ is the $i$th column of the volatility matrix $\sigma$. That is, we see the columns of $\sigma$ as a basis of $\mathbb{R}^d$, and for every element $i$ of the basis we suggest a normal jump $N(0,\Delta t)$ in the direction of the element. We then accept that jump with probability $p_i$ or reflect it in the opposite direction with probability $1-p_i$.

The values of $X^1,X^2,...$ can then be taken as approximate samples from the law of $Y_{\Delta t}, Y_{2 \Delta t},...$. The scheme is presented in algorithmic form as Algorithm \ref{skew.symmetric.algorithm:001}. Assumptions under which the scheme is asymptotically exact in the weak sense as the step-size decreases are presented in the next sub-section.

\begin{algorithm}
{\bf Input}: Initial point $x \in \mathbb{R}^d$, number of iterations $N \in \mathbb{N}$, step size $\Delta t >0$, probability functions $p_i:\mathbb{R}^d \times \mathbb{R}^d \rightarrow (0,1)$.

{\bf Goal}: Approximate samples from $Y_{\Delta t}, Y_{2 \Delta t}, ..., Y_{N \Delta t}$.

    {\bf Output}: $X^1, X^2, ..., X^N$.

    \begin{itemize}
        \item \ Set $X^0=x$.
        \item \ For $n \in \{0,...,N-1\}$, repeat
        \begin{enumerate}
            \item \ Sample $\nu=(\nu_1,...,\nu_d) \sim N(0,I_d)$.
            \item \ For all $i \in \{ 1,...,d \}$, set \begin{equation*}
                b_i = \begin{cases}
	                   +1 &\quad \text{with probability } p_i(X^{n}, \nu), \\
	                    -1 &\quad \text{otherwise,}
                    \end{cases}
                   \end{equation*}
                and set $b=(b_1,...,b_d)$. 
            \item \ Set  $X^{n+1}=X^n+\Delta t^{1/2}  \sigma(X^n) (b \ast \nu)$.
        \end{enumerate}
    \end{itemize}
    \caption{Skew-symmetric scheme}\label{skew.symmetric.algorithm:001}
\end{algorithm}

\begin{remark}\label{remark.q.notation:1}
Let 
$\xi=(\xi_1,...,\xi_d)=\Delta t^{1/2}\nu$, with $\nu \sim N(0,I_d)$. In some cases it will be more convenient to view each $p_i(x,\nu)$ as a function of the current position $x$ and the proposed Brownian increment $\xi$. We therefore introduce the functions $q_i$, where for all $i \in \{ 1,...,d  \}$, $x \in \mathbb{R}^d$ and $\xi \in \mathbb{R}^d$
\begin{equation}\label{notation.q:001}
    q_i(x,\xi):=p_i(x,\nu),
\end{equation}
with the convention that $\xi=\Delta t^{1/2} \nu$ and $\nu = \Delta t^{-1/2} \xi$.
\end{remark}

\subsection{Assumptions and main results}\label{sec:results}

We begin by stating our main assumptions on the drift and volatility coefficients in the SDE \eqref{multi.dimensional.sde:0} that will be imposed throughout the document. 

\begin{assumption}[Drift]
\label{ass.drift.general:1}
For all $i \in \left\{ 1,\dots , d \right\}$, we have $\mu_i \in C^{\infty}_P(\mathbb{R}^d)$. Furthermore, the drift satisfies a \emph{one-sided Lipschitz condition:} There exists $C>0$ such that
\begin{equation}\label{one.sided.lipschitz.assump:2}
\left( \mu(y) -  \mu(x) \right) \cdot \left( y-x \right) \leq C^2\left( 1+ \| y-x \|^2 \right),
\end{equation}
for all $x,y \in \mathbb{R}^d$.
\end{assumption}

\begin{assumption}[Volatility]\label{ass.volatility.general:1}
\begin{enumerate}
    \item \ For all $i,j \in \left\{ 1,\dots, d \right\}$, we have $\sigma_{i,j} \in C^{\infty}_P(\mathbb{R}^d)$ 
   Furthermore, there exists $C>0$ such that for all $x, y \in \mathbb{R}^d$
  \begin{equation}\label{volatility.lipschitz:1}
      \| \sigma(y)-\sigma(x)\|_F^2 \leq C^2\left(\| y-x \|^2 \right).
  \end{equation}

\item \ For all $x \in \mathbb{R}^d$, $\sigma(x)$ is invertible.

    \end{enumerate}
\end{assumption}

%We note here that Assumption \ref{ass.volatility.general:1}.2 is not used for Theorem \ref{Generalisation.theorem:1}, stated in Section \ref{sec:weak-convergence}. Aside from this theorem, Assumption \ref{ass.volatility.general:1}.2 will be made throughout this document.

\begin{remark}
\begin{enumerate}
    \item  \ The one-sided Lipschitz condition (\ref{one.sided.lipschitz.assump:2}) along with (\ref{volatility.lipschitz:1}) and the local Lipschitz property of $\mu$ and $\sigma$ (which is implied by the assumption that they are $C^{\infty}$) imply that the stochastic differential equation (\ref{multi.dimensional.sde:0}) admits a solution for all $t \in [0,+\infty)$ (e.g. \cite{krylov1991simple}, or Lemma 3.2 of \cite{higham.mao.stuart:03}). 
    \item \ The one-sided Lipschitz condition on the drift $\mu$ is weaker than the global Lipschitz condition requiring that there exists $C>0$ such that for all $x,y \in \mathbb{R}^d$ 
    \begin{equation}\label{ass.fully.lipschitz:1}
        \| \mu(x)-\mu(y) \| \leq C \| x-y \|.
\end{equation}
Indeed, \eqref{one.sided.lipschitz.assump:2} allows for fast growth of $\Vert \mu(x) \Vert$ as $\Vert x\Vert \rightarrow \infty$, as long as $\mu$ is ``inward-pointing'' or ``confining''.
   \end{enumerate}
\end{remark}

\begin{example}[Langevin Dynamics]\label{langevin.running:1}
    Our running example will be the overdamped Langevin dynamics
    \begin{equation}\label{langevin.running:2}
        \mathrm{d}Y_t = \nabla \log \pi(Y_t)\,\mathrm{d}t + \sqrt{2}\, \mathrm{d}W_t,
    \end{equation}
    for some probability density $\pi$ on $\mathbb{R}^d$. Under mild regularity conditions (see Theorem 2.1 of \cite{roberts1996exponential}) this diffusion is known to converge as $t \rightarrow \infty$ to the law induced by $\pi$, in total variation distance. Numerically discretising the diffusion has traditionally given rise to well-known sampling algorithms (e.g. MALA  \citep{roberts1996exponential}). 

   When the probability density $\pi$ has very light tails, the one-sided Lipschitz condition might still hold, but the (global) Lipschitz condition (\ref{ass.fully.lipschitz:1}) does not. A natural example is when $d=1$ and for some $\epsilon,c >0$
\begin{equation*}
\pi(x)=\frac{1}{Z}\exp\{ -c|x|^{2+\epsilon} \},
\end{equation*}
where $Z=\int_{\mathbb{R}^d} \exp\{ -c|y|^{2+\epsilon} \}\,\mathrm{d}y$, i.e. when $\pi$ has lighter tails than any Gaussian distribution.
\end{example}

Furthermore, throughout we make the following assumption on each $p_i$. 

\begin{assumption}\label{assumption.p.derivative:001}
    For all $i \in \{ 1,...,d \}$, we have that $p_i \in C^4(\mathbb{R}^d \times \mathbb{R}^d)$ as well as $0 < p_i(x,\nu) <1$ for all $x,\nu \in \mathbb{R}^d$. Moreover, writing
\begin{equation}\label{def:Psi.notation:1}
        \Psi(x)=\sigma^{-1}(x) \mu(x)
    \end{equation}
    we have
\begin{equation}\label{condition.p.derivative:0}
p_i(x,0)=\frac{1}{2} \qquad \text{and} \qquad 
\partial^{\nu}_ip_i(x,0) = \frac{1}{2} \Psi_i(x)  \Delta t^{1/2},
\end{equation}
for all $x \in \mathbb{R}^d$.
\end{assumption}

Assumption \ref{assumption.p.derivative:001} naturally generalises the one-dimensional conditions (\ref{p.0.1d:1}) and (\ref{eq:p condition 1d}). In particular, the second equation in (\ref{condition.p.derivative:0}) is essential for the skew-symmetric scheme to approximate the diffusion.

 \begin{remark}
 \begin{enumerate}
 \item \ Using the notation introduced in Remark \ref{remark.q.notation:1}, and using the chain rule,  \eqref{condition.p.derivative:0} becomes 
\begin{equation}\label{condition.q.derivative:0}
q_i(x,0)=\frac{1}{2}, \qquad \text{and} \qquad 
\partial^{\xi}_iq_i(x,0) = 
\frac{1}{2} \Psi_i (x).
\end{equation}
 
 \item \ As in Proposition \ref{cdf.p:0},  let $F$ be the cumulative distribution function of a one-dimensional distribution with continuous density $f$ which is symmetric with respect to zero and with $f(0) \neq 0$. Then the choice \begin{equation}\label{cdf.p.high.dimensions:1}
    p_i(x,\nu)=F\left(\nu_i \Delta t^{1/2} \frac{1}{2f(0)} \Psi_i(x) \right)
\end{equation} 
 satisfies \eqref{condition.p.derivative:0}.
\item \ It is possible to adjust our scheme to accommodate the case where the dimensions of the state $Y$ and the Brownian motion are different. For instance, consider a $d$-dimensional elliptic stochastic differential equation driven by $m$-dimensional standard Brownian motion with its volatility matrix $\sigma: \mathbb{R}^d \to \mathbb{R}^{d \times m}$, where $a (x) \equiv \sigma (x) \sigma (x)^\top$ is assumed to be positive definite uniformly in $x \in \mathbb{R}^d$. Then, the skew-symmetric scheme can be designed in the same manner as above, with the sign flipping for each coordinate of the Gaussian variable $v = (v_1, \ldots, v_m)$ dictated by the probability 
\begin{align*}
p_i (x, v) = F \Bigl( v_i \Delta t^{1/2} \tfrac{1}{2 f(0)} \bigl( \mu (x)^\top a (x)^{-1} \sigma _i (x) \bigr) \Bigr), \quad i = 1, \ldots, m.  
\end{align*}
Here $F$ is as in the second part of this Remark and $\sigma_i$ is the $i$th column of $\sigma$.
\end{enumerate}
\end{remark}

\begin{example}[The Barker Proposal]\label{example:Barker}
    We return to our running example, the overdamped Langevin dynamics given by (\ref{langevin.running:2}). If the $p_i$'s are of the form (\ref{cdf.p.high.dimensions:1}), with $f$ and $F$ being respectively the density and cumulative distribution functions of the Logistic distribution, introduced in Example \ref{example.cdf.p:0}, then
    $$
    X^{n+1}_i=X^n_i+ \sqrt{2\Delta t}\nu_i b_i
    $$
    for all $i \in \{ 1 , \dots , d\}$, where $\nu_i \sim N(0,1)$, and $b_i=1$ with probability
    \begin{equation*}
p_i(x,\nu)=\frac{1}{1+\exp\left\{ -\sqrt{2\Delta t} \nu_i \partial_i \log \pi(x) \right\} }
\end{equation*}
    and $b=-1$ otherwise. This update was originally considered in \cite{livingstone2022barker} as a Metropolis--Hastings proposal.  A version without a Metropolis-Hastings correction step is presented as Algorithm 1 in \cite{mauri2024robust}, and called the \emph{unadjusted} Barker proposal. We refer to these schemes throughout as the \textit{Barker proposal} and the \textit{Unadjusted Barker algorithm}.
    \end{example}

Finally, we make the following regularity assumption throughout.

\begin{assumption}\label{regularity.coefficients:1}
There exists  $C>0$ and $m \in \mathbb{N}$ such that for all $k \leq 4$, $x,\xi \in \mathbb{R}^d$,
    \begin{equation*}
      \left|q_i^{(k)}(x,\xi)\right|  \leq C \left(1 + \| x \|^m + \| \xi \|^m \right), 
    \end{equation*} 
where $q_i^{(k)}$ denotes any partial derivative of $q_i$ (with respect to any coordinates in $\mathbb{R}^{d} \times \mathbb{R}^d$) of order $k$.
\end{assumption}

\begin{remark}
Straightforward calculations show that, under Assumptions \ref{ass.drift.general:1} and \ref{ass.volatility.general:1}, various choices of $q_i$'s such as 
$q_i(x,\xi)= 1/\!\left[ 1 \!+\! \exp\left\{ -2 \xi_i \Psi_i (x)  \right\} \right]$,
induced by the Logistic CDF applied on (\ref{cdf.p.high.dimensions:1}) satisfy Assumption \ref{regularity.coefficients:1}.
\end{remark}

Given the above, combined with additional result-specific assumptions outlined in Sections \ref{sec:general.convergence:00}-\ref{sec:ergodicity}, the main contributions of this article can be summarised as follows.

\begin{result}\label{result:1}
    Under appropriate assumptions, as the step-size $\Delta t \rightarrow 0$, the skew-symmetric scheme converges weakly to the true diffusion in $O(\Delta t)$, meaning it is weak order 1 (Theorem \ref{theorem.weak.convergence:0}).
\end{result}

\begin{result}\label{result:2}
    Under appropriate assumptions, for any fixed step-size $\Delta t \in (0,1)$, the skew-symmetric scheme has an invariant probability measure $\pi_{\Delta t}$ and converges to this measure at a geometric rate in total variation distance (Theorem \ref{thm:geometric_ergodicity}).
\end{result}

\begin{result}\label{result:3}
    Under appropriate assumptions, as the step-size $\Delta t \rightarrow 0$, the bias between the limiting measures of the skew-symmetric scheme, $\pi_{\Delta t}$, and the true diffusion, $\pi$, vanishes at a rate $O(\Delta t)$ (Theorem \ref{thm.bias.equilibrium:1}).
\end{result}

Finally, in order to prove Result \ref{result:1}, we first establish Theorem \ref{Generalisation.theorem:1}, which is of independent interest, since it generalises the theory of Milstein \& Tretyakov \cite[Theorem 2.2.1]{milstein.tretyakov:21} to a setting where there is no global Lipschitz condition for the drift of the diffusion. There are previous works in this direction by \cite{cerrai2001second,wang2023weak}, for which one needed assumption is that the drift satisfies a one-sided Lipschitz condition with negative Lipschitz constant, see e.g. Assumption 3.6.4 of \cite{wang2023weak}.  We establish the result under a general one-sided Lipschitz assumption.  Another result under similar assumptions to that presented here was independently and concurently proved in \cite{zhao2024weak}, which was originally published online around the same time as the present work.

%To the best of our knowledge, this is the first time this theorem is generalised beyond the Lipschitz drift regime and it adds a new tool that can be used in order to prove weak convergence results for a number of algorithms approximating stochastic differential equations.

\begin{result}\label{result:4}
    Under appropriate assumptions, Theorem 2.2.1 of \cite{milstein.tretyakov:21} generalises beyond the globally Lipschitz case to the setting in which the drift satisfies a one-sided Lipschitz condition, a locally Lipschitz condition and exhibits at most polynomial growth (Theorem \ref{Generalisation.theorem:1}).
\end{result}

\section{Weak convergence to the diffusion}\label{sec:general.convergence:00}

%\subsection{Weak convergence} \label{sec:weak-convergence}

In this section we will present our first result (Theorem \ref{theorem.weak.convergence:0}), showing that the skew-symmetric numerical scheme approximates the stochastic differential equation in the limit as the step-size $\Delta t \rightarrow 0$ at finite time scales. We begin by defining operators that will govern the leading-order terms in the relevant Taylor series expansions.

\begin{definition}
Let $L$ denote the generator of the diffusion process corresponding to the solution of (\ref{multi.dimensional.sde:0}), i.e. for $f \in C^2(\mathbb{R}^d)$
\begin{equation}\label{generator.diffusion:1}
  Lf(x)=  \mu(x) \cdot \nabla f(x)  + \frac{1}{2}\textup{Tr}\left( \nabla^2 f(x) \sigma(x)\sigma^{\top}(x)\right). 
\end{equation}
We also introduce the operator $\mathcal{A}_2:C^4(\mathbb{R}^d) \rightarrow C^0(\mathbb{R}^d)$, defined as
\begin{align}\label{second.derivative.numerical.scheme:1}
\mathcal{A}_2f(x)&=\sum_{i,k=1}^d\partial_if(x)\partial^{\xi}_{i}\partial^{\xi}_{k}\partial^{\xi}_{k}q_i(x,0)  \sigma_{i,i}(x)  \\
     &+ \sum_{i \neq k} \partial_i\partial_kf(x) \mu_i(x) \mu_k(x)
     +4 \sum_{i \neq k} \partial_{i}\partial_{k}f(x) \partial^{\xi}_{k}q_i(x,0) \partial^{\xi}_{i} q_k (x,0) \sigma_{i,i}(x) \sigma_{k,k}(x) \nonumber \\
        &+ \frac{1}{2} \sum_{i,k=1}^d\partial_i\partial_k^2f(x)  \mu_i(x)  \sigma_{k,k}(x)^2 
        +\frac{1}{8} \sum_{i,k=1}^d \partial_i^2\partial_k^2f(x) \sigma_{i,i}(x)^2   \sigma_{k,k} (x)^2, \nonumber
\end{align}
for $f \in C^4(\mathbb{R}^d)$. 
\end{definition}
Given the above, the following result characterises the local error associated to the skew-symmetric scheme $(X_n)_{n \in \mathbb{N}}$.

\begin{proposition}\label{semigroup.expansion:1}
Under Assumptions  \ref{ass.drift.general:1}, \ref{ass.volatility.general:1}, \ref{assumption.p.derivative:001} and \ref{regularity.coefficients:1}, the following hold:
\begin{enumerate}
\item \ For any $f \in C^4_P(\mathbb{R}^d)$, there exists $M_1:\mathbb{R}^d \times (0,1) \rightarrow \mathbb{R}$ and  $K_1 \in C^{\infty}_P(\mathbb{R}^d)$ such that for all $\Delta t \in (0,1)$ and $x \in \mathbb{R}^d$
\begin{equation}\label{semigroup.expansion.basic:2}
    \mathbb{E}_x\left[ f(X^{1}) \right] = f(x)+ \Delta t \cdot Lf(x) + M_1(x,\Delta t),
\end{equation}
with $|M_1(x,\Delta t)| \leq \Delta t^2 K_1(x)$.
\item \  Assume further that the volatility matrix $\sigma$ is diagonal: for all $i \neq j$, and all $x \in \mathbb{R}^d$, $\sigma_{i,j}(x)=0$. For any $f \in C^6_P(\mathbb{R}^d)$, there exists $M_2:\mathbb{R}^d \times (0,1) \rightarrow \mathbb{R}$ and  $K_2 \in C^{\infty}_P(\mathbb{R}^d)$ such that for all $\Delta t \in (0,1)$ and $x \in \mathbb{R}^d$
\begin{equation}\label{semigroup.expansion:2}
    \mathbb{E}_x\left[ f(X^{1}) \right] = f(x)+ \Delta t \cdot Lf(x) + \Delta t^2 \cdot \mathcal{A}_2f(x) + M_2(x,\Delta t),
\end{equation}
with $|M_2(x,\Delta t)| \leq \Delta t^3 K_2(x).$
\end{enumerate}
\end{proposition}
\begin{proof}[Proof of Proposition \ref{semigroup.expansion:1}]
 The proof relies on Taylor series expansions and is presented in Appendix \ref{proof.of.semigroup.expansion:00}.
\end{proof}

The second part of Proposition \ref{semigroup.expansion:1} will be useful when we study the long-time behaviour of the numerical scheme in Section \ref{subsection:bias.equilibrium:00}. The first part of Proposition \ref{semigroup.expansion:1} will be recognisable as the main ingredient needed for standard arguments establishing weak convergence. We show this below after establishing the following result, which is of independent interest, since it allows the theory of Milstein \& Tretyakov \cite[Theorem 2.2.1]{milstein.tretyakov:21} to be used for non-globally Lipschitz drifts. For this we need to make the following assumption.

\begin{assumption}\label{ass:der assumption}
The volatility coefficient $\sigma: \mathbb{R}^d \rightarrow \mathbb{R}^{d \times d}$ has bounded derivatives of all orders. Furthermore, there exists a constant $C > 0$ such that
 \begin{equation}
 \label{eq:drift contracting}
  \xi^\top \nabla \mu(x) \xi \le C \| \xi \|^2,   
 \end{equation}
 for all $x, \xi \in \mathbb{R}^d$.
\end{assumption}

\begin{remark}
\begin{enumerate}
    \item \ The assumption in \eqref{eq:drift contracting} is inspired by \cite{imkeller2019differentiability}.
Intuitively, it  means that the degree with  which $\mu$ ``points inwards towards the origin'' is bounded from below as $\Vert x\Vert \rightarrow \infty$. %For example, $\mu(x) = -f(\|x\|^2)x$, with $f$ positive, differentiable, and monotonically increasing satisfies this assumption.
\item \ In our running example of the Langevin dynamics, the drift satisfies $\mu=-\nabla U$ for some $U \in C^2(\mathbb{R}^d)$. In that case, \eqref{eq:drift contracting} is equivalent to asking that there exists $C>0$ such that for all $x \in \mathbb{R}^d$
\begin{equation*}
    \nabla^2U(x) \geq -C I_d,
\end{equation*}
meaning that all the eigenvalues of the Hessian of $U$ are bounded below by some (potentially negative) constant. This is a fairly mild assumption, typically satisfied in the context of sampling literature. 
\item \ Notice that it is equivalent to impose this condition outside of a ball of some radius $R$: The drift satisfies \eqref{eq:drift contracting} for all $x \in \mathbb{R}^d$ if and only if there exists $R > 0$ and $C>0$ such that \eqref{eq:drift contracting} is satisfied for all $x \in \mathbb{R}^d$ with $\Vert x\Vert \ge R $. 
\end{enumerate}
\end{remark}

\begin{theorem}[Regularity and growth of diffusion semi-group with non-globally Lipschitz drift]\label{Generalisation.theorem:1}
Assume that Assumptions \ref{ass.drift.general:1}, \ref{ass:der assumption} and the first part of Assumption \ref{ass.volatility.general:1} hold. Fix $T > 0$ and $l \in \mathbb{N}$, and consider the SDE
\begin{equation}
\label{eq:SDE_appendix}
\mathrm{d}Y^x_t = \mu(Y^x_t) \, \mathrm{d}t + \sigma(Y^x_t) \, \mathrm{d}W_t, \qquad Y_0 = x, \qquad t \in [0,T].
\end{equation}
For $f \in C^l_P(\mathbb{R^d})$, 
\begin{equation}\label{def.u:01}
u_t(x) := \mathbb{E}[f(Y^x_t)],\qquad  x \in \mathbb{R}^d, t \in [0,T],
\end{equation}
is well defined (that is, $\mathbb{E}[|f(Y_t^x)|] < \infty$, for all $x \in \mathbb{R}^d$ and $t \in [0,T]$). Moreover, $u_t \in C^{l}_P(\mathbb{R}^d)$, and there exist constants $C>0$ and $m \in \mathbb{N}$, independent of $t \in [0,T]$ and $x \in \mathbb{R}^d$, such that for all $x \in \mathbb{R}^d$
\begin{equation}
\label{eq:moment bound u}
|u_t(x)| \le C(1+ \Vert x \Vert^m)
\end{equation}
and for all $n=1,2, \dots l$
\begin{equation}
\label{eq:der u bound}
\Vert \nabla^n u_t(x) \Vert_F \le C(1 + \Vert x \Vert^m).
\end{equation}
\end{theorem}

\begin{proof}[Proof of Theorem \ref{Generalisation.theorem:1}]
    The proof is postponed until Appendix \ref{proof.of.regularity:00}.
\end{proof}
Combining Proposition \ref{semigroup.expansion:1} and Theorem \ref{Generalisation.theorem:1} we get the following result. This guarantees that for small step-size $\Delta t$, the skew symmetric process approximates the underlying stochastic differential equation.

\begin{theorem}[Weak convergence]\label{theorem.weak.convergence:0}
Assume that Assumptions 
\ref{ass.drift.general:1}, \ref{ass.volatility.general:1},
\ref{assumption.p.derivative:001}, \ref{regularity.coefficients:1} and \ref{ass:der assumption} hold. Fix $T>0$. Then there exists $K \in C^{\infty}_P(\mathbb{R}^d)$ such that for all $\Delta t>0$, for all $k \leq N:=\lfloor \frac{T}{\Delta t} \rfloor$, all $f \in C^4_P(\mathbb{R}^d)$ and all $x \in \mathbb{R}^d$ we have
\begin{equation}\label{order.1.convergence:1}
        \left| \mathbb{E}_x\left[ f(X^k) -f(Y_{k\Delta t})\right] \right| \leq K(x) \Delta t.
    \end{equation}
\end{theorem}

\begin{proof}[Proof of Theorem \ref{theorem.weak.convergence:0}]
The proof is postponed until Appendix  \ref{proof.weak.convergence:000}. 
\end{proof}

\section{Long-time behaviour}\label{sec:ergodicity}
 In the sampling literature, where the goal is to sample from a distribution $\pi$, many algorithms have been motivated by numerically discretising a stochastic differential equation such as (\ref{multi.dimensional.sde:0}), with $\mu$ and $\sigma$ carefully chosen so that the diffusion has equilibrium distribution $\pi$. Samples from the numerical scheme $X^1, X^2, \dots, X^n$ are then obtained, with the idea that the law of $X^n$ approximates $\pi$ as $n \rightarrow \infty$. The skew-symmetric scheme can be used for this purpose, and so in this section we study its behaviour in the limit as $n \rightarrow \infty$. Note that in this section we assume that the volatility $\sigma(x)$ is diagonal (see Assumption \ref{ass:long_time}.1).  This assumption will hold in many examples motivated by the need to sample from the equilibrium distribution $\pi$, such as our running example of Overdamped Langevin dynamics.

\subsection{Equilibrium distribution and geometric convergence}\label{sec:geom.ergo:}

We begin by investigating the existence of an equilibrium distribution for the skew-symmetric scheme. The main result of this subsection (stated as Theorem \ref{thm:geometric_ergodicity}) shows that under suitable assumptions the Markov chain produced by the skew-symmetric numerical scheme with fixed step-size $\Delta t$ has an equilibrium distribution $\pi_{\Delta t}$, with finite moments of all orders, and converges to this equilibrium at a geometric rate.
To this end we make the following additional assumption.

\begin{assumption} \label{ass:long_time}
    Assume the following:
\begin{enumerate}  
\item \ The volatility matrix is diagonal: for all $i \neq j$, $x \in \mathbb{R}^d$, 
\begin{equation*}
   \sigma_{i,j}(x)=0.
\end{equation*}
\item \ There exist $M>N>0$, such that for all $i \in \{ 1,...,d \}$ and $x \in \mathbb{R}^d$
\begin{equation*}
N \leq \sigma_{i,i}(x) \leq M.
\end{equation*} 
\item \ For $i \in\{1,...,d\}$, the drift component  $\mu_i$ satisfies
\begin{equation} \label{eq:contract}
\lim_{\substack{\|x\|\to\infty \\ x_i < -\|x\|_\infty/2}} \mu_i(x) 
= \infty, 
~~
\lim_{\substack{\|x\|\to\infty \\ x_i > \|x\|_\infty/2}} \mu_i(x)
= -\infty .
\end{equation}

\item \ For all $i \in\{1,...,d\}$, $p_i:\mathbb{R}^d \times \mathbb{R}^d \rightarrow (0,1)$ depends on $\nu$ only via the $i$th coordinate $\nu_i$, i.e. $p_i(x,\nu) := g_i(x,\nu_i)$ for some function $g_i$. Furthermore, for all $i \in \{ 1,...,d \}$, $g_i$ is bounded away from zero on compact sets. In addition %$g_i(x,\nu_i) \to 1$ as $\nu_i \cdot \mu_i(x) \to \infty$ and $\to 0$ as $\nu_i\cdot\mu_i(x) \to -\infty$.
 \begin{equation} \label{gi_eqation:1}
    \lim_{\nu_i \cdot \mu_i(x) \rightarrow +\infty} g_i(x,\nu) 
     = 1, 
     ~~
     \lim_{\nu_i \cdot \mu_i(x) \rightarrow -\infty} g_i(x,\nu)
     = 0 .
\end{equation}
\end{enumerate}
\end{assumption}

We observe that, under Assumptions \ref{ass:long_time}.1 and \ref{ass:long_time}.2, Assumption \ref{ass:long_time}.4 will be satisfied if we choose $p_i(x,\nu)$ according to the CDF form of (\ref{cdf.p.high.dimensions:1}). Below we provide two simple sufficient conditions under which equations \eqref{eq:contract} hold, although we stress that they are also likely to hold in many other cases of interest.

\begin{proposition} \label{prop:sufficient_for_contract}
    Equations \eqref{eq:contract} hold if for some $R<\infty$ and all $\|x\|>R$ either
    \begin{enumerate}
        \item (Independence outside a ball) For all $i \in \{ 1,...,d \}$, $\mu_i$ depends on $x$ only via its $i$th coordinate, i.e. $\mu_i(x) := \tilde{\mu_i}(x_i)$, and $\lim_{x_i \to \pm \infty}\tilde{\mu_i}(x_i) =\mp \infty$.
        \item (Spherical symmetry outside of a ball) $\mu_i(x) = x_i\cdot \rho'(\|x\|)/\|x\|$ for some $\rho \in C^1(\mathbb{R})$ such that $\lim_{s\to \infty}\rho'(s) = -\infty$.
    \end{enumerate}
\end{proposition}

\begin{proof}[Proof of Proposition \ref{prop:sufficient_for_contract}]
     The independent case follows directly from the definition. For the second case note that by equivalence of norms on $\mathbb{R}^d$ there exists $c>0$ such that if $|x_i| > \|x\|_{\infty}/2$ and $x \neq 0$ then
     $
     c \leq |x_i| / \|x\| \leq 1.
     $
     As $x_i \to \infty$ under the restriction $|x_i|>\|x\|_\infty/2$, it therefore holds that $c \rho'(\|x\|) \leq \mu_i(x) \leq \rho'(\|x\|)$, which implies that $\mu_i(x) \to -\infty$. In the case $x_i \to -\infty$ under the same restriction, a similar argument shows that $\mu_i(x) \to\infty$.
\end{proof}

\begin{remark}
    Both cases can be considered when $\mu(x) := -\nabla U(x)$ for some potential $U(x)$. Outside a ball of radius $R$, the first corresponds to $U(x) := \sum_i U_i(x_i)$ for suitable $U_i$ and the second to $U(x) := \rho(\|x\|)$ for suitable $\rho$.
\end{remark}

We now state the main result of this section.

\begin{theorem}[Geometric Ergodicity]\label{thm:geometric_ergodicity}
Under Assumption \ref{ass:long_time} the Markov chain generated by the skew-symmetric numerical scheme has an invariant measure $\pi_{\Delta t}$ with finite moments of all orders. Furthermore, denoting by $P$ the associated Markov kernel, there exists $C>0$, $\lambda > 0$ and $V:\mathbb{R}^d \to [1,\infty)$ such that for all $n \in \mathbb{N}$ and $x \in \mathbb{R}^d$ 
\begin{equation} \label{eq:geo_erg}
\|P^n(x,\cdot) - \pi_{\Delta t}\|_{TV} \leq CV(x) e^{-\lambda n},
\end{equation}
where $\|\cdot\|_{TV}$ denotes the total variation norm.
\end{theorem}

\begin{proof}[Proof of Theorem \ref{thm:geometric_ergodicity}]
The proof is postponed until Appendix \ref{proof.geom.ergo:000}. 
\end{proof}

\begin{remark}
Assumption \ref{ass:long_time}.3 is sufficient for Theorem \ref{thm:geometric_ergodicity}, but it may not be necessary: geometric ergodicity can still hold when each $|\mu_i(x)|$ remains bounded as $|x| \to \infty$. Such cases require a different approach and likely stronger assumptions on each $p_i$, as they no longer tend to $0$ or $1$ in the appropriate limit for typical values of $\nu$.  We exclude these cases in favour of having a cleaner statement under milder assumptions on each $p_i$.  % and note that they are often omitted in analyses of sampling algorithms where is known that a bespoke approach will be needed for these cases (see e.g. the discussion of \cite{jarner2000geometric}).
\end{remark}

\subsection{Controlling the bias at equilibrium}\label{subsection:bias.equilibrium:00}

Let us return to the problem of sampling from a distribution of interest $\pi$, using samples from a stochastic differential equation that has $\pi$ as its unique invariant distribution. In Section \ref{sec:geom.ergo:} we have established that in the long-time limit the skew-symmetric scheme approximating the diffusion converges exponentially fast to an equilibrium measure $\pi_{\Delta t}$. The discretisation error induced by the numerical scheme, however, means that $\pi_{\Delta t} \neq \pi$ in general, making the skew-symmetric scheme a non-exact/unadjusted algorithm to sample from $\pi$. On the other hand, since the scheme approximates the diffusion as $\Delta t \rightarrow 0$, there is hope that $\pi_{\Delta t}$ will converge to $\pi$ as $\Delta t \to 0$ in an appropriate sense. Given that the skew-symmetric scheme will eventually sample from $\pi_{\Delta t}$, rather than $\pi$, it is important to understand the discrepancy between $\pi$ and $\pi_{\Delta t}$. This will be studied in this subsection (Theorem \ref{thm.bias.equilibrium:1}). We begin by imposing the following assumption.

\begin{assumption}\label{stoltz.ass.pi:1}
    \begin{enumerate}
        \item \ The measure $\pi$ has moments of all orders. Furthermore, the associated probability density satisfies $\pi \in C^4(\mathbb{R}^d)$, $\pi(x)>0$ for all $x \in \mathbb{R}^d$, and $\log \pi \in C^4_P(\mathbb{R}^d)$. 
        \item \ There exists $a>0$ and $R>0$ such that for all $\| x \| \geq R$,
    \begin{equation}
        \sum_{i=1}^dx_i\mu_i(x) \leq -C \| x \|^a.
    \end{equation}
    \item \ The diffusion $(Y_t)_{t \geq 0}$ in (\ref{multi.dimensional.sde:0}) is time-reversible, has $\pi$ as its unique invariant distribution and converges to $\pi$ in total variation distance as $t \rightarrow \infty$.
    \end{enumerate}
\end{assumption}

\begin{remark}
 Assumption \ref{stoltz.ass.pi:1}.1 holds for example when the density is $\pi(x)\propto \exp\{ -p(\|x\|) \}$ for some polynomial $p$. Returning to our running example of the overdamped Langevin diffusion of \eqref{langevin.running:2}, Assumption \ref{stoltz.ass.pi:1}.2 will typically be satisfied for suitably regular forms of $\pi$ with exponential or lighter tails. In the one-dimensional case, for example, $\mu(x)=(\log \pi)^{'}(x)$, meaning Assumption \ref{stoltz.ass.pi:1}.2 becomes $(\log\pi)^{'}(x) \leq -C x^{a-1}$
 when $x>R$ and $(\log\pi)^{'}(x) \geq C |x|^{a-1}$ when $x < -R$, which is typically satisfied by any density $\pi$ that decays at a faster than polynomial rate.
\end{remark}

Under this assumption we have the following error estimate.

\begin{theorem}[Bias at equilibrium]\label{thm.bias.equilibrium:1}
Under Assumptions  \ref{ass.drift.general:1}, \ref{ass.volatility.general:1}, \ref{assumption.p.derivative:001}, \ref{regularity.coefficients:1}, \ref{ass:long_time} and \ref{stoltz.ass.pi:1}, there exists $L>0$ and $g \in C^{\infty}_P(\mathbb{R}^d)$ with $\int g \,\mathrm{d}\pi=0$, such that for all $\Delta t \in (0,1)$  and $f \in C^{\infty}_P(\mathbb{R}^d)$ there exists $|R_{\Delta t, f}| \leq L$ such that
    \begin{equation}\label{equilibrium.bias.bound:0}
        \int f d\pi_{\Delta t}=\int f d\pi+ \Delta t \int f g d\pi + \Delta t^2 R_{\Delta t, f}. 
    \end{equation}
\end{theorem}

\begin{proof}[Proof of Theorem \ref{thm.bias.equilibrium:1}]
The proof is presented in Appendix \ref{proof.bias.equilibrium:00}.
\end{proof}

A natural corollary of Theorems \ref{thm:geometric_ergodicity} and \ref{thm.bias.equilibrium:1} is the following result on the long-time behaviour of the unadjusted Barker algorithm (Example \ref{example:Barker}, see also Algorithm 1 of \cite{mauri2024robust}).

\begin{corollary}[Long-time behaviour of the unadjusted Barker algorithm]\label{barker.corollary}
    Assume that the measure $\pi$ has moments of all orders. Furthermore, assume that the associated probability density satisfies $\pi \in C^{\infty}(\mathbb{R}^d)$, $\pi(x)>0$ for all $x \in \mathbb{R}^d$, and $\log \pi \in C^{\infty}_P(\mathbb{R}^d)$. Assume further that there exists $a>0$ and $R>0$ such that for all $\| x \| \geq R$
    \begin{equation*}
         x \cdot \nabla \log \pi(x) \leq -C \| x \|^a,
    \end{equation*}
     and that for all $i \in\{1,...,d\}$
\begin{equation*}
\lim_{\substack{\|x\|\to\infty \\ x_i < -\|x\|_\infty/2}} \partial_i \log \pi(x) 
= \infty, 
~~
\lim_{\substack{\|x\|\to\infty \\ x_i > \|x\|_\infty/2}} \partial_i \log \pi(x)
= -\infty .
\end{equation*}
Let $X^1, X^2, \dots $ be the unadjusted Barker algorithm targeting $\pi$ as introduced in Example \ref{example:Barker} with transition kernel denoted by $P$. Then there exists a measure $\pi_{\Delta}$ with finite moments of all orders, constants $C, \lambda >0$ and $V: \mathbb{R}^d \rightarrow [1,\infty)$ such that for all $n \in \mathbb{N}$ and $x \in \mathbb{R}^d$
\begin{equation*}
    \| P^n(x,\cdot) - \pi_{\Delta t} \|_{TV} \leq C V(x) e^{-\lambda n}.
\end{equation*}
Furthermore, there exists $L>0$ and $g \in C^{\infty}_P(\mathbb{R}^d)$ with $\int g d\pi=0$ such that for all $\Delta t \in (0,1)$ and $f \in C^{\infty}_P(\mathbb{R}^d)$, there exists $|R_{\Delta t , f}| \leq L$ such that 
\begin{equation*}
    \int f d\pi_{\Delta t} = \int f d \pi + \Delta t \int f g d\pi + \Delta t^2 R_{\Delta t , f}.
\end{equation*}
\end{corollary}

\section{Experiments}\label{sec:experiments}

In this section we conduct various numerical experiments to illustrate the weak order, long-time behaviour and stability properties of the skew-symmetric numerical scheme.

\subsection{Multiplicative volatility}

We illustrate the behaviour of the skew-symmetric scheme when simulating a diffusion process with multiplicative volatility. In particular, we consider sampling from the one-dimensional stochastic differential equation
$$
dY_t = -Y_t dt + a Y_t dW_t, \ \ \ Y_0=x,
$$
for various values of $a>0$ and $x>0$.

We compare the performance of the skew-symmetric scheme (with flipping probability set as in (\ref{cdf.p.high.dimensions:1}) using $F(x) := (1+e^{-x})^{-1}$), against the Euler–Maruyama scheme and the tamed Euler approach of \cite{hutzenthaler2012strong}. Performance is assessed by considering the absolute difference between the exact solution of $\mathbb{E}[Y_T]$ and a Monte Carlo average generated using the respective numerical scheme. Here we use the terminal time $T=5$. We report results for three different starting points ($x=0.1 , 1 , 10$) and for two values of $a$ ($a= 0.5 , 2$). In each case the Monte Carlo estimator was generated using $10^5$ i.i.d. samples from the numerical scheme in question. In Figure \ref{fig:experiment4} we report our results on the weak error ($y$-axis) over various step-sizes for the numerical schemes ($x$-axis). 

\begin{figure}[h!]
    \centering
    \includegraphics[width=0.9\linewidth]{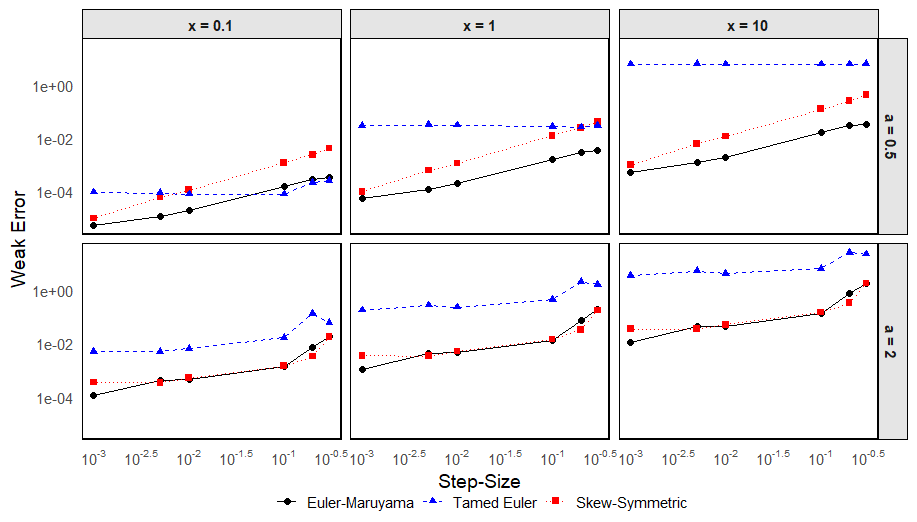}
    \caption{Plots for weak error of the mean for the diffusion process $dY_t = -Y_t dt + a Y_t dW_t$ initialized at $x = 0.1, 1, 10$ and for $a = 0.5, 2$. The weak error is approximated by the absolute difference between the true expectation and the Monte Carlo average over $10^5$ i.i.d. sample trajectories simulated using the Euler--Maruyama, tamed Euler, and skew-symmetric schemes.}
    \label{fig:experiment4}
\end{figure}

The results indicate that, while the volatility takes arbitrarily small values in the area where the process spends most of its time, this fact alone does not influence the algorithmic performance of the skew-symmetric scheme. This may be surprising, as the size of the jump is essentially dictated by only the volatility, and therefore it might be expected that the skew-symmetric scheme performs poorly when this is small. As shown by the results, however, the key driver of algorithm performance is instead the ratio between drift and volatility. In the model considered here, this ratio is given by $a^{-1}$. As can be seen in Figure \ref{fig:experiment4}, in the case $a=2$ the skew-symmetric scheme performs similarly to Euler--Maruyama, while Euler--Maruyama performs better when $a = 0.5$.  It is of interest to note that in the considered cases, the skew-symmetric scheme generally performs better than tamed Euler, particularly when $x = 10$.  It can be shown that in this model the skew-symmetric scheme will track the mean more accurately than Tamed Euler when the initial value $x$ is large.  We illustrate this with a more detailed mathematical comparison between the two approaches in Section 1 of the supplementary material.

\subsection{Poisson random effects model}

Next we compare the performance of the skew-symmetric scheme against Euler--Maruyama and semi-implicit Euler for long-time simulation of an overdamped Langevin diffusion with equilibrium distribution corresponding to the posterior of a Poisson random effects model, which is popular in Bayesian Statistics (e.g. \cite{dey2000generalized}). Using the Euler--Maruyama scheme equates to implementing the unadjusted Langevin algorithm (ULA) to sample from the associated posterior distribution. We compare the performance of ULA with the skew-symmetric scheme, with probability function $p$ chosen as in Example \ref{example.cdf.p:0}.1. In this setting we call this algorithm the {\it unadjusted Barker algorithm (UBA)}. We also compare to the semi-implicit Euler method (e.g. \cite{hu1996semi}), in which the increments are $X^{n+1} = X^n + \Delta t((1-\theta)\mu(X^n) + \theta \mu(X^{n+1})) + \sigma(X^n)\cdot \sqrt{\Delta t}\cdot \nu^n$, with $\nu^n \sim N(0,I)$. We set $\theta = 0.2$ to balance numerical stability and implicitness, and use a tolerance of $10^{-3}$ allowing a maximum number of $500$ fixed point iterations per step to implement the semi-implicit Euler method.

The Poisson random effects model is of the following form: 
\begin{equation*}
\begin{split}
y_{ij} | \eta_i &\sim \text{Poi}(e^{\eta_i}), \quad \quad j = 1, \ldots, J, \\  
\eta_i | \mu &\sim N(\mu,1), \quad \quad\,\, i = 1, \ldots, I, \\
\mu &\sim N(0, \sigma_\mu^2).
\end{split}
\end{equation*}
Defining the state vector $x = (\mu, \eta_1, ..., \eta_{I})$, this results in a posterior distribution with density $\pi(x) \propto \exp \{ -U(x) \}$ and corresponding potential
\begin{equation*}
U(x) = J\sum_i e^{\eta_i} - \sum_{i,j}y_{ij}\eta_i + \frac{1}{2}\sum_i (\eta_i - \mu)^2 + \frac{\mu^2}{2\sigma^2_\mu}.
\end{equation*}
We approximate the resulting overdamped Langevin diffusion $\mathrm{d}X_t = - \nabla U(X_t)\mathrm{d}t + \sqrt{2}\mathrm{d}W_t$, with a key focus on simulating from the distribution with density $\pi(x)\propto e^{-U(x)}$. Note that the drift vector is not globally Lipschitz.

The data $y_{ij}$ for $i=1,...,I$ and $j = 1,..., J$ were simulated using true parameter values $\mu^* = 5$, with each $\eta_i \sim N(\mu^*,1)$. To generate the data we set $I = 50$, $J = 5$ and $\sigma_{\mu} = 10$.  With these choices there are enough observations that the marginal posterior distribution for $\mu$ is essentially centred at $\mu^*$.  We compared the two numerical schemes by assessing the quality of ergodic averages of $\mu$ using mean squared Monte Carlo error for a fixed number of simulation steps.  For each step-size $\Delta t$ we repeated the simulation 100 times and in each case evaluated the ergodic average of $\mu$. We then computed the empirical mean squared error of these estimates as compared to $\mu^*$.  Larger step-sizes decrease the level of correlation between neighbouring iterates, reducing the Monte Carlo variance associated with ergodic averages, but increase the bias.  %(Could a better comparison be done here using some distance measure against 'ground truth samples'?)

We consider two ways to initialise the scheme. The first is to initialise at the true value $\mu^* = 5$. The second is to draw the starting point from a $\text{N}(5, 10^2)$ distribution, often called a \emph{warm start}. In the latter case, we expect the resulting Markov chain to have an initial transient phase before reaching equilibrium, so $10,000$ initial steps were discarded before another $50,000$ were used to evaluate mean squared error.  The choice of $10,000$ iterations to discard was made using pilot runs and visual inspection of trajectories. We note in passing that ULA tended to take a longer time to reach equilibrium than UB for the same choice of step-size.  In both settings each $\eta_i|\mu$ was initialised from the distribution $N(\mu,1)$.

\begin{figure}[ht] 
\centering
\includegraphics[width=10cm]{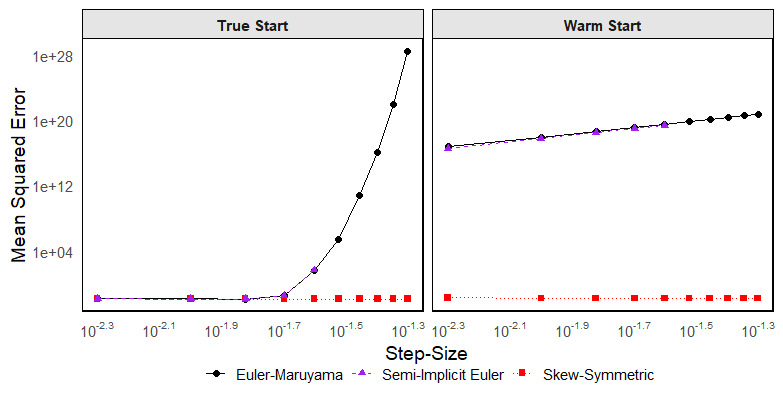} 
\caption{Mean squared error comparisons for the Poisson random effects example. Figure 2(a) (left-hand side) shows mean squared error when initialised at the true value $\mu^*$. Figure 2(b) (right-hand side) is initialised from a $N(5, 10^2)$ sample.}
\label{fig:poisson}
\end{figure}

As shown in Figure \ref{fig:poisson}, when initialised at $\mu^*$ both Euler--Maruyama and semi-implicit Euler perform poorly for $\Delta t \geq 0.03$, while the skew-symmetric scheme performs reasonably for all step-sizes used in the simulation. At small $\Delta t$, the performance of all three schemes is comparable, as expected. The warm start results are more pronounced, as shown in Figure \ref{fig:poisson}, as the skew-symmetric mean squared error remains low whereas the Euler--Maruyama error increases exponentially quickly as $\Delta t$ increases. Furthermore, the semi-implicit Euler scheme exhibits severe numerical instability for larger step-sizes, and is hundreds of times more computationally costly to run.

\subsection{Soft spheres in an anharmonic trap}

In this experiment we simulate particles in two spatial dimensions with soft sphere interactions evolving in an anharmonic trap (see e.g. Section 4.2 of \cite{boffi2023probability}). The particle dynamics are governed by the stochastic differential equation
\begin{align*}
\mathrm{d}Y_t^{(i)} = \,\,
& 4B(\beta - Y_t^{(i)}) \|Y_t^{(i)} - \beta \|^2 \, \mathrm{d}t 
+ \frac{A}{Nr^2} \sum_{j=1}^N (Y_t^{(i)} - Y_t^{(j)})e^{-\| Y_t^{(i)} - Y_t^{(j)} \|^2/2r^2} \, \mathrm{d}t + \sqrt{2D}\,\mathrm{d}W_t,
\end{align*}
for $i = 1, 2, \ldots, N$ where $\beta$ is the position of the trap, $A > 0$ is the strength of the repulsion between spheres with radius $r$, and $B > 0$ is the strength of the trap. We set $\beta = (0,0)^T$, $D = 0.25$, $A = 30$, $r = 0.15$ and $N = 50$. We vary the strength of the trap through parameter $B$ to vary the degree of numerical difficulty associated with the problem. 

We compared the skew-symmetric, Euler--Maruyama and semi-implicit Euler  schemes with step-sizes $\Delta t \in \{0.1, 0.2,0.3, \ldots, 1\}$, $B \in \{0.1, 0.2,0.3, \ldots, 1\}$ and hand-tuned $\theta = 0.2$ for semi-implicit. For each $(\Delta t,B)$-pair the dynamics were simulated for 10 steps, and this simulation was then repeated 100 times. The initial positions of the particles were independently drawn from a uniform distribution on $[-1,1]^2$. We then measured the probability of numerical explosion, meaning the position of one or more spheres reaching $\pm \infty$ (a number too large/small for the computer to record) during the simulation. The proportion of times such an explosion occurred for each chosen $(\Delta t,B)$-pair is shown in Figure \ref{fig:softsphere}.  

\begin{figure}[h] 
\centering
\includegraphics[width=13.5cm]{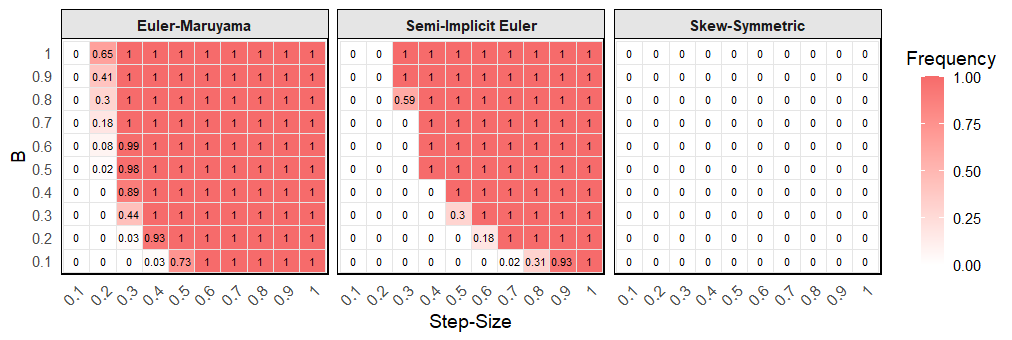} %{heatmap.with.semi.implicit}
\caption{Heat map of numerical explosion frequency for the soft spheres model. The $x$-axis shows chosen $B$ values, with larger $B$ increasing the level of difficulty of the problem, and the $y$ axis shows step-size.}
\label{fig:softsphere}
\centering
\end{figure}

The left-hand side of Figure \ref{fig:softsphere} illustrates that explosion occurred for both the Euler--Maruyama and semi-implicit schemes quite frequently as either $B$ or $\Delta t$ increases. This is because the drift is not globally Lipschitz, meaning the explicit Euler--Maruyama becomes unstable and in the semi-implicit scheme the fixed point iterations fail to converge often.  The increase in $B$ implies a stronger attraction of the trap, which leads to more numerical difficulties and the need for a smaller step-size to ensure non-explosion on the simulation timescales. The skew-symmetric scheme is, however, extremely stable and does not exhibit any explosion events, as shown on the right-hand side of Figure \ref{fig:softsphere}.

\section{Conclusion}\label{conclusion:00}
We have introduced a new simple and explicit numerical scheme for time-homogeneous and uniformly elliptic stochastic differential equations. The skew-symmetric scheme updates its current position using a two-step process: for each coordinate, first the size of the jump is simulated without the influence of the drift, and then the direction of the jump is decided, based on a probability that depends on both the drift and volatility of the current position. This mechanism makes the scheme highly robust when the drift is not globally Lipschitz.  We have established basic results on the weak convergence of the numerical scheme, and, while proving these, we generalised the theory of Milstein and Tretyakov \cite[Theorem 2.2.1]{milstein.tretyakov:21} to non-globally Lipschitz drifts. %Path-wise accuracy results are obtained in Section \ref{strong.convergence:000} for diffusions with constant drift and volatility. %The more general case is work in progress, and outside the scope of this article.

A key concern of the article is long-time simulation of ergodic diffusions. To this end, we have established that the skew-symmetric scheme can be used to generate approximate samples from the invariant distribution of such a diffusion. Under suitable conditions we have established that the numerical scheme converges at a geometric rate to its equilibrium distribution. We have also provided quantitative bounds on the distance between the equilibria of the numerical and exact processes with respect to the numerical step-size. In Section \ref{sec:experiments} we present numerical experiments, providing empirical evidence to support our theoretical results. 

There are many avenues for future work on skew-symmetric schemes. For example, the current scheme is not suitable for sampling hypo-elliptic diffusions, where some of the state coordinates are not directly driven by Brownian motion; thus, the diffusion matrix $\sigma$ is not invertible. Extensions in this direction are currently being considered. The Taming approach can also be applied to non-Lipschitz volatilities, whereas the skew-symmetric scheme as currently defined may perform poorly in this setting, and so extensions in this direction are possible. Finally, it would be interesting to consider a multilevel extension of the scheme.

\section*{Supplementary Material}
Supplementary Material for ``Skew-symmetric schemes for stochastic
differential equations with non-Lipschitz drift: an unadjusted Barker
algorithm”. Doi: To be confirmed.
Includes a discussion on the differences between the skew-symmetric scheme and other schemes, such as Taming. It also includes more simulations, comparing the performance of the skew-symmetric schemes against other algorithms.
The simulation codes can also be found at \url{https://github.com/ShuSheng3927/skew_symmetric_SDE}.

\section*{Acknowledgments}
The authors thank Camilo Garcia--Trillos, Alex Beskos, Jure Vogrinc \& Michael Tretyakov for useful discussions. The research was conducted while GV was a postdoctoral fellow at UCL, under SL. 

\section*{Funding}
SL and GV were supported by an EPSRC New Investigator Award (EP/V055380/1).  RZ was supported by an LMS undergraduate bursary (URB-2023-71). YI was supported by an EPSRC grant Prob AI (EP/Y028783/1).

\bibliographystyle{plainnat}
\bibliography{main}

\begin{thebibliography}{44}
\providecommand{\natexlab}[1]{#1}
\providecommand{\url}[1]{\texttt{#1}}
\expandafter\ifx\csname urlstyle\endcsname\relax
  \providecommand{\doi}[1]{doi: #1}\else
  \providecommand{\doi}{doi: \begingroup \urlstyle{rm}\Url}\fi

\bibitem[Abdulle et~al.(2015)Abdulle, Vilmart, and Zygalakis]{abdulle2015long}
Assyr Abdulle, Gilles Vilmart, and Konstantinos~C Zygalakis.
\newblock {Long time accuracy of Lie--Trotter splitting methods for Langevin dynamics}.
\newblock \emph{SIAM Journal on Numerical Analysis}, 53\penalty0 (1):\penalty0 1--16, 2015.

\bibitem[Abramowitz and Stegun(1948)]{abramowitz1948handbook}
Milton Abramowitz and Irene~A Stegun.
\newblock \emph{{Handbook of Mathematical Functions with Formulas, Graphs, and Mathematical Tables}}, volume~55.
\newblock US Government printing office, 1948.

\bibitem[Andrieu et~al.(2003)Andrieu, De~Freitas, Doucet, and Jordan]{andrieu2003introduction}
Christophe Andrieu, Nando De~Freitas, Arnaud Doucet, and Michael~I Jordan.
\newblock {An introduction to MCMC for machine learning}.
\newblock \emph{Machine learning}, 50:\penalty0 5--43, 2003.

\bibitem[Azzalini(2013)]{azzalini2013skew}
Adelchi Azzalini.
\newblock \emph{{The Skew-Normal and Related Families}}, volume~3.
\newblock Cambridge University Press, 2013.

\bibitem[Boffi and Vanden-Eijnden(2023)]{boffi2023probability}
Nicholas~M Boffi and Eric Vanden-Eijnden.
\newblock {Probability flow solution of the Fokker--Planck equation}.
\newblock \emph{Machine Learning: Science and Technology}, 4\penalty0 (3):\penalty0 035012, 2023.

\bibitem[Bogachev and Ruas(2007)]{bogachev2007measure}
Vladimir~Igorevich Bogachev and Maria Aparecida~Soares Ruas.
\newblock \emph{{Measure Theory}}, volume~1.
\newblock Springer, 2007.

\bibitem[Bolley et~al.(2012)Bolley, Gentil, and Guillin]{bolley2012convergence}
Fran{\c{c}}ois Bolley, Ivan Gentil, and Arnaud Guillin.
\newblock {Convergence to equilibrium in Wasserstein distance for Fokker--Planck equations}.
\newblock \emph{Journal of Functional Analysis}, 263\penalty0 (8):\penalty0 2430--2457, 2012.

\bibitem[Bou-Rabee and Vanden{-}Eijnden(2018)]{bourabee.vandeneijnden:18}
Nawaf Bou-Rabee and Eric Vanden{-}Eijnden.
\newblock \emph{Continuous--Time Random Walks for the Numerical Solution of Stochastic Differential Equations}, volume 256 of \emph{Memoirs of the American Mathematical Society}.
\newblock American Mathematical Society, Providence, RI, 2018.
\newblock ISBN 978-1-4704-3181-5.
\newblock \doi{10.1090/memo/1228}.
\newblock URL \url{https://doi.org/10.1090/memo/1228}.

\bibitem[Brosse et~al.(2019)Brosse, Durmus, Moulines, and Sabanis]{brosse2019tamed}
Nicolas Brosse, Alain Durmus, {\'E}ric Moulines, and Sotirios Sabanis.
\newblock {The tamed unadjusted Langevin algorithm}.
\newblock \emph{Stochastic Processes and their Applications}, 129\penalty0 (10):\penalty0 3638--3663, 2019.

\bibitem[Cerrai(2001)]{cerrai2001second}
Sandra Cerrai.
\newblock \emph{Second order PDE’s in finite and infinite dimension: a probabilistic approach}.
\newblock Springer, 2001.

\bibitem[De~Bortoli et~al.(2021)De~Bortoli, Thornton, Heng, and Doucet]{bortoli.thornton.heng.doucet:21}
Valentin De~Bortoli, James Thornton, Jeremy Heng, and Arnaud Doucet.
\newblock {Diffusion Schr\"{o}dinger bridge with applications to score-based generative modeling}.
\newblock In \emph{Advances in Neural Information Processing Systems}, volume~34, pages 17695--17709. Curran Associates, Inc., 2021.

\bibitem[Dey et~al.(2000)Dey, Ghosh, and Mallick]{dey2000generalized}
Dipak~K Dey, Sujit~K Ghosh, and Bani~K Mallick.
\newblock \emph{{Generalized Linear Models: A Bayesian Perspective}}.
\newblock CRC Press, 2000.

\bibitem[Durmus and Moulines(2017)]{durmus2017nonasymptotic}
Alain Durmus and Eric Moulines.
\newblock {Nonasymptotic convergence analysis for the unadjusted Langevin algorithm.}
\newblock \emph{Annals of applied probability}, 27\penalty0 (3):\penalty0 1551--1587, 2017.

\bibitem[Fang and Giles(2020)]{fang2020adaptive}
Wei Fang and Michael~B Giles.
\newblock {Adaptive Euler--Maruyama method for SDEs with nonglobally Lipschitz drift}.
\newblock \emph{The Annals of Applied Probability}, 30\penalty0 (2):\penalty0 526--560, 2020.

\bibitem[Faulkner and Livingstone(2024)]{faulkner2024sampling}
Michael~F Faulkner and Samuel Livingstone.
\newblock {Sampling algorithms in statistical physics: a guide for statistics and machine learning}.
\newblock \emph{Statistical Science}, 39\penalty0 (1):\penalty0 137--164, 2024.

\bibitem[Friedli and Velenik(2017)]{friedli.velenik:17}
Sacha Friedli and Yvan Velenik.
\newblock \emph{{Statistical Mechanics of Lattice Systems: A Concrete Mathematical Introduction}}.
\newblock Cambridge University Press, 2017.

\bibitem[Gilbarg et~al.(1977)Gilbarg, Trudinger, Gilbarg, and Trudinger]{gilbarg1977elliptic}
David Gilbarg, Neil~S Trudinger, David Gilbarg, and NS~Trudinger.
\newblock \emph{{Elliptic Partial Differential Equations of Second Order}}, volume 224.
\newblock Springer, 1977.

\bibitem[Green et~al.(2015)Green, {\L}atuszy{\'n}ski, Pereyra, and Robert]{green2015bayesian}
Peter~J Green, Krzysztof {\L}atuszy{\'n}ski, Marcelo Pereyra, and Christian~P Robert.
\newblock {Bayesian computation: a summary of the current state, and samples backwards and forwards}.
\newblock \emph{Statistics and Computing}, 25:\penalty0 835--862, 2015.

\bibitem[Hairer and Mattingly(2011)]{hairer2011yet}
Martin Hairer and Jonathan~C Mattingly.
\newblock {Yet another look at Harris’ ergodic theorem for Markov chains}.
\newblock In \emph{Seminar on Stochastic Analysis, Random Fields and Applications VI: Centro Stefano Franscini, Ascona, May 2008}, pages 109--117. Springer, 2011.

\bibitem[Higham et~al.(2003)Higham, Mao, and Stuart]{higham.mao.stuart:03}
Desmond~J. Higham, Xuerong Mao, and Andrew~M. Stuart.
\newblock {Strong Convergence of Euler-Type Methods for Nonlinear Stochastic Differential Equations}.
\newblock \emph{SIAM Journal on Numerical Analysis}, 40\penalty0 (3):\penalty0 1041--1063, 2003.
\newblock ISSN 00361429.
\newblock URL \url{http://www.jstor.org/stable/4100914}.

\bibitem[Hird et~al.(2020)Hird, Livingstone, and Zanella]{hird2020fresh}
Max Hird, Samuel Livingstone, and Giacomo Zanella.
\newblock {A fresh take on ‘Barker dynamics’ for MCMC}.
\newblock In \emph{International Conference on Monte Carlo and Quasi-Monte Carlo Methods in Scientific Computing}, pages 169--184. Springer, 2020.

\bibitem[Hu(1996)]{hu1996semi}
Yaozhong Hu.
\newblock Semi-implicit euler-maruyama scheme for stiff stochastic equations.
\newblock In \emph{Stochastic Analysis and Related Topics V: The Silivri Workshop, 1994}, pages 183--202. Springer, 1996.

\bibitem[Hutzenthaler et~al.(2012)Hutzenthaler, Jentzen, and Kloeden]{hutzenthaler2012strong}
Martin Hutzenthaler, Arnulf Jentzen, and Peter~E. Kloeden.
\newblock {Strong convergence of an explicit numerical method for SDEs with nonglobally Lipschitz continuous coefficients}.
\newblock \emph{The Annals of Applied Probability}, 22\penalty0 (4), 2012.

\bibitem[Imkeller et~al.(2019)Imkeller, dos Reis, and Salkeld]{imkeller2019differentiability}
Peter Imkeller, Gon{\c{c}}alo dos Reis, and William Salkeld.
\newblock {Differentiability of SDEs with drifts of super-linear growth}.
\newblock \emph{Electronic Journal of Probability}, 24\penalty0 (none), 2019.

\bibitem[Kloeden and Platen(1992)]{kloeden1992stochastic}
Peter~E Kloeden and Eckhard Platen.
\newblock \emph{{Stochastic Differential Equations}}.
\newblock Springer, 1992.

\bibitem[Krylov(1991)]{krylov1991simple}
Nikolai~Vladimirovich Krylov.
\newblock {A simple proof of the existence of a solution of It\^o’s equation with monotone coefficients}.
\newblock \emph{Theory of Probability \& Its Applications}, 35\penalty0 (3):\penalty0 583--587, 1991.

\bibitem[Kunita(2019)]{kunita2019stochastic}
Hiroshi Kunita.
\newblock \emph{{Stochastic Flows and Jump-Diffusions}}.
\newblock Springer, 2019.

\bibitem[Leimkuhler et~al.(2016)Leimkuhler, Matthews, and Stoltz]{leimkuhler2016computation}
Benedict Leimkuhler, Charles Matthews, and Gabriel Stoltz.
\newblock {The computation of averages from equilibrium and nonequilibrium Langevin molecular dynamics}.
\newblock \emph{IMA Journal of Numerical Analysis}, 36\penalty0 (1):\penalty0 13--79, 2016.

\bibitem[Lelievre and Stoltz(2016)]{lelievre2016partial}
Tony Lelievre and Gabriel Stoltz.
\newblock {Partial differential equations and stochastic methods in molecular dynamics}.
\newblock \emph{Acta Numerica}, 25:\penalty0 681--880, 2016.

\bibitem[Livingstone and Zanella(2022)]{livingstone2022barker}
Samuel Livingstone and Giacomo Zanella.
\newblock {The Barker proposal: combining robustness and efficiency in gradient-based MCMC}.
\newblock \emph{Journal of the Royal Statistical Society Series B: Statistical Methodology}, 84\penalty0 (2):\penalty0 496--523, 2022.

\bibitem[Mao(2015)]{mao2015truncated}
Xuerong Mao.
\newblock {The truncated Euler--Maruyama method for stochastic differential equations}.
\newblock \emph{Journal of Computational and Applied Mathematics}, 290:\penalty0 370--384, 2015.

\bibitem[Mao(2016)]{mao2016convergence}
Xuerong Mao.
\newblock {Convergence rates of the truncated Euler--Maruyama method for stochastic differential equations}.
\newblock \emph{Journal of Computational and Applied Mathematics}, 296:\penalty0 362--375, 2016.

\bibitem[Mauri and Zanella(2024)]{mauri2024robust}
Lorenzo Mauri and Giacomo Zanella.
\newblock {Robust Approximate Sampling via Stochastic Gradient Barker Dynamics}.
\newblock In \emph{International Conference on Artificial Intelligence and Statistics}, pages 2107--2115. PMLR, 2024.

\bibitem[Meyn and Tweedie(2012)]{meyn2012markov}
Sean~P Meyn and Richard~L Tweedie.
\newblock \emph{{Markov Chains and Stochastic Stability}}.
\newblock Springer Science \& Business Media, 2012.

\bibitem[Milstein and Tretyakov(2021)]{milstein.tretyakov:21}
Grigori~N Milstein and Michael~V Tretyakov.
\newblock \emph{{Stochastic Numerics for Mathematical Physics}}.
\newblock Springer, 2021.

\bibitem[Pardoux and Veretennikov(2001)]{pardoux.veretennikov:01}
E.~Pardoux and Yu. Veretennikov.
\newblock {On the Poisson Equation and Diffusion Approximation. I}.
\newblock \emph{The Annals of Probability}, 29\penalty0 (3):\penalty0 1061 -- 1085, 2001.

\bibitem[Roberts and Tweedie(1996)]{roberts1996exponential}
Gareth~O Roberts and Richard~L Tweedie.
\newblock {Exponential convergence of Langevin distributions and their discrete approximations}.
\newblock \emph{Bernoulli}, pages 341--363, 1996.

\bibitem[Rousset et~al.(2010)Rousset, Stoltz, and Lelievre]{stoltz2010free}
Mathias Rousset, Gabriel Stoltz, and Tony Lelievre.
\newblock \emph{{Free Energy Computations: A Mathematical Perspective}}.
\newblock World Scientific, 2010.

\bibitem[Sabanis(2013)]{sabanis2013note}
Sotirios Sabanis.
\newblock {A note on tamed Euler approximations}.
\newblock \emph{Electronic Communications in Probability}, 18\penalty0 (a):\penalty0 47, 2013.

\bibitem[Sant et~al.(2023)Sant, Jenkins, Koskela, and Spanò]{sant.jenkins.koskela.spano:23}
Jaromir Sant, Paul~A Jenkins, Jere Koskela, and Dario Spanò.
\newblock {EWF: simulating exact paths of the Wright–Fisher diffusion}.
\newblock \emph{Bioinformatics}, 39\penalty0 (1), 01 2023.
\newblock ISSN 1367-4811.

\bibitem[Skiadas(2010)]{skiadas2010exact}
Christos~H Skiadas.
\newblock {Exact solutions of stochastic differential equations: Gompertz, generalized logistic and revised exponential}.
\newblock \emph{Methodology and Computing in Applied Probability}, 12:\penalty0 261--270, 2010.

\bibitem[Vogrinc et~al.(2023)Vogrinc, Livingstone, and Zanella]{vogrinc2022optimal}
Jure Vogrinc, Samuel Livingstone, and Giacomo Zanella.
\newblock {Optimal design of the Barker proposal and other locally balanced Metropolis--Hastings algorithms}.
\newblock \emph{Biometrika}, 110\penalty0 (3):\penalty0 579--595, 2023.

\bibitem[Wang et~al.(2023)Wang, Zhao, and Zhang]{wang2023weak}
Xiaojie Wang, Yuying Zhao, and Zhongqiang Zhang.
\newblock Weak error analysis for strong approximation schemes of sdes with super-linear coefficients.
\newblock \emph{IMA Journal of Numerical Analysis}, 44\penalty0 (5):\penalty0 3153--3185, 11 2023.
\newblock ISSN 0272-4979.
\newblock \doi{10.1093/imanum/drad083}.
\newblock URL \url{https://doi.org/10.1093/imanum/drad083}.

\bibitem[Zhao et~al.(2024)Zhao, Wang, and Zhang]{zhao2024weak}
Yuying Zhao, Xiaojie Wang, and Zhongqiang Zhang.
\newblock Weak error analysis for strong approximation schemes of sdes with super-linear coefficients ii: finite moments and higher-order schemes.
\newblock \emph{arXiv preprint arXiv:2406.14065}, 2024.

\end{thebibliography}

\newpage

\appendix

\section{Proof of Proposition \ref{semigroup.expansion:1}}\label{proof.of.semigroup.expansion:00}

We begin with the following simple observation.
\begin{lemma}\label{lemma.betas:1}
    If we write $b=(b_1,...,b_d) \in \{ \pm 1 \}^d$, then for any $n \in \mathbb{N}$ and $i_1,...,i_n \in \{ 1,...,d \}$, the following statements hold. If at least one of the $i_j$'s appears an odd number of times amongst $\{ i_1,...,i_n \}$ then
    \begin{equation*}
        \sum_{b \in \{ \pm 1 \}^d}b_{i_1}b_{i_2}...b_{i_n}=0.
    \end{equation*}
   Otherwise 
   \begin{equation*}
        \sum_{b \in \{ \pm 1 \}^d}b_{i_1}b_{i_2}...b_{i_n}=2^d.
   \end{equation*}
\end{lemma}

\begin{proof}[Proof of Lemma \ref{lemma.betas:1}]
    Assume first that at least one of the coordinates (assume without loss of generality that this is $i_n$) appears an odd number of times amongst $i_1,...,i_n$. Then, there exist $k \in \mathbb{N}$ and $j_1,j_2,...,j_m \neq i_n$ such that
    \begin{align*}
        &\sum_{b \in \{ \pm 1 \}^d}b_{i_1}...b_{i_n}=\sum_{b \in \{ \pm 1 \}^d}b_{j_1}...b_{j_m}b^{2k+1}_{i_n} =2^{d-m-1}\!\!\!\!\!\!\sum_{b_{j_1},b_{j_2},... b_{j_m} \in \{ \pm 1 \}} \sum_{b_{i_n} \in \{ \pm 1 \}}b_{j_1}...b_{j_m}b_{i_n}\\
        &=2^{d-m-1}\sum_{b_{j_1},b_{j_2},... b_{j_m} \in \{ \pm 1 \}}b_{j_1}...b_{j_m}\sum_{b_{i_n} \in \{ \pm 1 \}} b_{i_n}=0.
    \end{align*}
    If, on the other hand, every $i_j$ appears an even number of times amongst $i_1,...,i_d$, then, since $b_{i_j} \in \{ \pm 1 \}$, the product between the even number of terms that are $b_{i_j}$ is equal to one. Therefore
    \begin{equation*}
        \sum_{b \in \{  \pm 1\}^d}b_{i_1}...b_{i_n}=\sum_{b \in \{ \pm 1 \}^d}1 = 2^d. 
    \end{equation*}
\end{proof}

We also use the following well-known result for moments of multivariate Gaussians. This is the odd-dimensional version of Isserlis' theorem, (see e.g. \cite[Chapter 8: Exercise 8.3]{friedli.velenik:17}). 

\begin{theorem}[Isserlis' Theorem]\label{isserlis.thm:0}
If $Y=(Y_1,...,Y_d)$ is a $d$-dimensional Gaussian, with $d$ odd then
\begin{equation*}
\mathbb{E}\left[ Y_1 Y_2 \ldots  Y_d \right]=0.
\end{equation*}
\end{theorem}

We can now prove Proposition  \ref{semigroup.expansion:1}. 

In the next two subsections we will prove the first and the second parts of Proposition \ref{semigroup.expansion:1} respectively. Recall the definition of $q (x, \sqrt{\Delta t} v) \equiv {p} (x, v)$ in Remark \ref{remark.q.notation:1}.

\subsection{Proof of Proposition \ref{semigroup.expansion:1}.1}

\begin{proof}[Proof of Proposition \ref{semigroup.expansion:1}.1]
Recall that after one step, the process starting from $x \in \mathbb{R}^d$ will transition to 
% $X_{1}= x + \xi \ast b$, 
$X_{1}= x + \Delta t^{1/2} \sigma (x) 
\cdot (b \ast \nu)$
where $\nu \sim N(0,I_d)$ and for all $i \in \{ 1,...,d \}$,  
% [$\xi_i=\sqrt{\Delta t}   \sigma_{i,i}(x) \nu_i$],  and 
$b_i=+1$ with probability 
$p_i (x, \nu)$
or $b_i=-1$ with probability 
$1 - p_i (x, \nu)$. 
Recall that for two vectors $c, d \in \mathbb{R}^d$, we write $c \ast d \in \mathbb{R}^d$ to denote the vector of element-wise products between $c$ and $d$.

Then we can write 
\begin{align}\label{semigroup.expansion.product:1}
\mathbb{E}_x\left[ f(X_{1})-f(x) \right]
= \sum_{b \in \{ \pm 1 \}^d}\mathbb{E}_{\nu} \Bigg[ &\left(f \bigl(x+ \Delta t^{1/2} \sigma(x) (b \ast \nu) \bigr) 
-f(x) \right)
\\
&\cdot \prod_{i=1}^d \left[ q_i(x, \Delta t^{1/2} \nu) \mathbb{I}(b_i=1)+(1- q_i(x, \Delta t^{1/2} \nu))\mathbb{I}(b_i=-1) \right] \Bigg] \nonumber.
\end{align}
Let us introduce the following notation. For any $n \in \mathbb{N}$, and any vector $w=(w_1,...,w_d) \in \mathbb{R}^d$, we write 
\begin{equation*}
    f^{(n)}(y)\left( w \right)^n= \sum_{i_1,...,i_n \in \{ 1,...,d \} }\partial^y_{i_1}\partial^y_{i_2} ... \partial^y_{i_n}f(y) w_{i_1}w_{i_2}...w_{i_n}.
\end{equation*}
Then, a Taylor expansion with respect to $\Delta t^{1/2}$ around $0$ gives
    \begin{align}\label{semigroup.expansion.f.expansion:1}
& f(x+\Delta t^{1/2} \sigma(x) (b \ast \nu))  -f(x)\\
    & =  \Delta t^{1/2} f^{(1)}(x)( \sigma(x) (b \ast \nu))^1 + \frac{1}{2}\Delta t^{2/2} f^{(2)}(x)( \sigma(x) (b \ast \nu) )^2
    \nonumber
    \\ &+ \frac{1}{3!} \Delta t^{3/2} f^{(3)}(x)(\sigma(x) (b \ast \nu))^3
    + \frac{1}{4!}\Delta t^{4/2} f^{(4)}(x)(\sigma(x) (b \ast \nu))^4 
    \nonumber
    \\
    & + \frac{1}{5!}\Delta t^{5/2} f^{(5)}(x)(\sigma(x) (b \ast \nu))^5
    + \frac{1}{6!}\Delta t^{6/2} f^{(6)}(x+\tau_1 \sigma(x) (b \ast \nu))( \sigma(x) (b \ast \nu))^6,
    \nonumber
    \end{align}
    for some $\tau_1 \in (0,\Delta t^{1/2})$, and 
    another Taylor expansion of $q_i( x, \sqrt{\Delta t} \nu)$ with respect to 
    $\Delta t^{1/2}$ around $0$ gives $q_i( x, \sqrt{\Delta t} \nu) = q_i ( x, 0) + \mathscr{P}_i (x, \sqrt{\Delta t} \nu)$ with 
    \begin{align}\label{def.Pi:1}
    \mathscr{P}_i (x, \sqrt{\Delta t} \nu) 
    & = \Delta t^{1/2} q_i^{(1)}(x,0)(\nu)^1 \!+\!\frac{1}{2}\Delta t^{2/2} q_i^{(2)}(x,0)(\nu)^2 \! + \frac{1}{3!}\Delta t^{3/2} q_i^{(3)}(x,0) (\nu)^3 \\
    & + \!\frac{1}{4!}\!\Delta t^{4/2} q_i^{(4)}(x,0)(\nu)^4 \! 
    + \frac{1}{5!}\Delta t^{5/2}q_i^{(5)}(x,0)(\nu)^5 \!+  \!\frac{1}{6!}\Delta t^{6/2} q_i^{(6)} (x,\tau_2 \nu)(\nu)^6 \nonumber
    \end{align}
    for some $\tau_2 \in (0,\Delta t^{1/2})$, where all the derivatives for 
    %$q_i(x,\xi)$
    $q_i (x, \xi)$
    are taken with respect to $\xi$. Noting that %$q_i(x,0)=\frac{1}{2}$,
    $q_i (x, 0) = \tfrac{1}{2}$, we see that
\begin{align}\label{semigroup.expansion.p.expansion:1}
q_i(x,\sqrt{\Delta t} \nu) \mathbb{I}(b_i=1)+ \bigl(1-q_i(x, \sqrt{\Delta t} \nu) \bigr)\mathbb{I}(b_i=-1) =\frac{1}{2} + b_i \times \mathscr{P}_i (x, \sqrt{\Delta t} \nu).   
\end{align}
    
    Combining \eqref{semigroup.expansion.f.expansion:1} and \eqref{semigroup.expansion.p.expansion:1} we view the quantity inside the expectation of \eqref{semigroup.expansion.product:1} as a polynomial in $\Delta t^{1/2}$. %, up to an error term. 
    We will now identify the coefficients of this polynomial.
    First of all, we consider the coefficient of the term $\Delta t^{1/2}$. This is
    \begin{align*}
        f^{(1)}(x)(\sigma(x) (b \ast \nu))^1 \left( \frac{1}{2} \right)^d= \left( \frac{1}{2} \right)^d \sum_{i, j =1}^d\partial_{i} f(x) \sigma_{i, j}(x) \nu_j b_j.
    \end{align*}
    Taking the expectation over $\nu \sim N (0, I_d)$, we see that this is equal to zero. 
    % , since $\sigma(x)  \nu \sim N(0,\sigma\sigma^T)$.
    
    Next, we consider the coefficient of $\Delta t$. Considering the different ways in which the $d+1$ factors in \eqref{semigroup.expansion.product:1} can give order $\Delta t$ though the Taylor expansions \eqref{semigroup.expansion.f.expansion:1} and \eqref{semigroup.expansion.p.expansion:1}, we see that the coefficient is
\begin{align}\label{delta.term:1}
    & f^{(1)}(x)( \sigma (x) (b \ast \nu))^1 \sum_{k=1}^d b_k q_k^{(1)}(x,0)(\nu)^1 \left( \frac{1}{2} \right)^{d-1}
      \!\! + \frac{1}{2!}f^{(2)}(x)(\sigma (x) (b \ast \nu))^2 \left( \frac{1}{2} \right)^d \nonumber \\
    & \equiv C_2^{(\mathrm{I})} (b, \nu) +  C_2^{(\mathrm{II})} (b, \nu). 
    \end{align}
    Let's consider each part of the sum separately. We have that 
\begin{align*}
    C_2^{(\mathrm{I})} (b, \nu) 
    & = \left( \sum_{i, j  = 1}^d \partial_i f(x) \sigma_{i, j} (x) b_j \nu_j  \right) 
    \left( \sum_{k,l  = 1}^d b_k \partial_l^{\xi}q_k(x,0) \nu_l \right) \left( \frac{1}{2} \right)^{d-1} \\ 
    & = \left( \frac{1}{2} \right)^{d-1} 
     \sum_{i, j, k , l = 1}^d \partial_i  f (x) \sigma_{i, j} (x) \partial_l^{\xi} q_k (x,0)  v_j v_l b_j b_k. 
\end{align*}
Summing over $b \in \{ \pm 1 \}^d$ to obtain \eqref{semigroup.expansion.product:1}, we get 
\begin{align*} %\label{delta.term:2}
\sum_{b \in \{ \pm 1 \}^d} C_2^{(\mathrm{I})} (b, \nu) 
= 2 \sum_{i, j, l = 1}^d \partial_i  f (x) \sigma_{i, j} (x) \partial_l^{\xi}q_j(x,0)  v_j v_l, \end{align*}
where the equality is due to the following result from Lemma \ref{lemma.betas:1}:  
% we have
\begin{equation}\label{delta.term:3}
\sum_{b \in \{ \pm 1 \}^d } b_j b_k=2^d\mathbb{I}(j=k). 
\end{equation}
Taking the expectation over $\nu \sim N(0,I_d)$, we see that 
\begin{align} \label{delta.term:4} 
& \sum_{b \in \{ \pm 1 \}^d } \mathbb{E}_\nu \left[ C_2^{(\mathrm{I})} (b, \nu) \right] 
= 2 \sum_{i, j = 1}^d \partial_i  f (x) \sigma_{i, j} (x) \partial_j^{\xi}q_j(x,0) \nonumber \\ 
& =  \sum_{i = 1}^d \partial_i  f (x) \left( 2 \sum_{j=1}^d \sigma_{i, j} (x) \partial_j^{\xi}q_j(x,0) \right)
= \mu(x) \cdot \nabla f(x),   
\end{align}
where in the last equality we have used Assumption \ref{assumption.p.derivative:001}:  $\partial_j^\xi q_j (x, 0) = \tfrac{1}{2} \Psi_j (x)$ for all $j \in \{1,\cdots, d\}$, leading to $2 \sum_{j = 1}^d \sigma_{i, j} (x) \partial_j^{\xi}q_j(x,0) = \mu_i(x)$, for all $i \in \{ 1, \dots , d \}$.

For the other part of the sum in \eqref{delta.term:1} we have
\begin{align*}
C_2^{(\mathrm{II})} (b, \nu) 
= \frac{1}{2^{d+1}} \sum_{i, j, k, \ell = 1}^d \partial_{i} \partial_{j} f(x) \sigma_{i, k} (x) \sigma_{j, \ell} (x) \,  v_k v_{\ell} \, b_k b_\ell. 
\end{align*}
Summing over $b \in \{ \pm 1 \}^d$ with (\ref{delta.term:3}) and taking the expectation w.r.t. $\nu$, we get that 
\begin{align*}
\sum_{b \in \{ \pm 1\}^d} \mathbb{E}_\nu \bigl[ C_2^{(\mathrm{II})} (b, \nu) \bigr]
= \frac{1}{2} \sum_{i, j, k = 1}^d 
\partial_i \partial_j f(x) \sigma_{i, k} (x)  \sigma_{j, k} (x) \mathbb{E}_\nu \left[ (\nu_k)^2 \right]
= \frac{1}{2} \mathrm{Tr} \left( \nabla^2 f(x) \sigma (x) \sigma^\top (x) \right).  
\end{align*}
Combining this with \eqref{delta.term:4}, we see that the coefficient of $\Delta t$ is $Lf(x)$, where $L$ as in \eqref{generator.diffusion:1} is the generator of the diffusion we are approximating.

We now consider the coefficient of the term $\Delta t^{3/2}$. Considering the different ways in which the $d+1$ factors in \eqref{semigroup.expansion.product:1} can give order $\Delta t^{3/2}$ through the Taylor expansions \eqref{semigroup.expansion.f.expansion:1} and \eqref{semigroup.expansion.p.expansion:1}, we see that the coefficient is
\begin{align*}
& f^{(1)}(x)((\sigma(x) (b \ast \nu) )^1\sum_{i=1}^d\frac{1}{2!}b_i q_i^{(2)}(x,0)(\nu)^2 
\left( \frac{1}{2} \right)^{d-1} \\
&+  f^{(1)}(x)((\sigma(x) (b \ast \nu))^1 \sum_{i \neq j}b_i q_i^{(1)}(x,0)(\nu)^1 b_j q_j^{(1)}(x,0)(\nu)^1 \left(  \frac{1}{2} \right)^{d-2} \\
&+ \frac{1}{2!}f^{(2)}(x)((\sigma(x) (b \ast \nu))^2 \sum_{i=1}^d b_i 
q_i^{(1)} (x,0)(\nu)^1 \left( \frac{1}{2} \right)^{d-1} \\
&+ \frac{1}{3!}f^{(3)}(x)((\sigma(x) (b \ast \nu))^3 \left( \frac{1}{2} \right)^{d}.
\end{align*} 
Taking expectations over $\nu$, all four summands involve terms of the form $\mathbb{E}_{\nu}\left[ \nu_{i_1} \nu_{i_2} \nu_{i_3} \right]$, for $i_1,i_2,i_3 \in \{ 1,2,...,d \}$. From Theorem \ref{isserlis.thm:0}, since $\nu \sim N(0,I_d)$, these expectations are zero, so the coefficient of $\Delta t^{3/2}$ is zero. 

Under our assumptions, the rest of the terms are shown to be bounded by $\Delta t^2 K (x)$ for some $K \in C_p^{\infty} (\mathbb{R}^d)$ using similar arguments to those provided below for proving the higher order error expansion (\ref{semigroup.expansion:2}). We thus conclude the proof of   (\ref{semigroup.expansion.basic:2}).
\end{proof}

\subsection{Proof of Proposition \ref{semigroup.expansion:1}.2}

\begin{proof}[Proof of Proposition \ref{semigroup.expansion:1}.1]

We now assume that the volatility matrix is diagonal. Following the proof of the first part of Proposition \ref{semigroup.expansion:1}
we recover the same coefficients for the polynomial on the terms $\Delta t^{1/2}, \Delta t , \Delta t^{3/2}$. We now use equations (\ref{semigroup.expansion.product:1}),(\ref{semigroup.expansion.f.expansion:1}) and (\ref{semigroup.expansion.p.expansion:1}) in order to identify the coefficients of the terms $\Delta t^2 , \Delta t^{5/2}$ and $\Delta t^3$.

We first focus on the coefficient of the term $\Delta t ^{5/2}$. Using a similar argument to the case of $\Delta t^{3/2}$, we see that the coefficient  involves summing terms of the form $\mathbb{E}_{\nu}\left[ \nu_{i_1} \nu_{i_2} \nu_{i_3} \nu_{i_4} \nu_{i_5} \right]$, for $i_1,i_2,i_3, i_4, i_5 \in \{ 1,2,...,d \}$. All these terms are zero due to Isserlis' Theorem, therefore the coefficient of $\Delta t^{5/2}$ is zero.

We now consider the coefficient of the term $\Delta t^2$. Considering the different ways in which the $d+1$ factors in \eqref{semigroup.expansion.product:1} can give order $\Delta t^{4/2}$ through the Taylor expansions \eqref{semigroup.expansion.f.expansion:1} and \eqref{semigroup.expansion.p.expansion:1}, we see that the coefficient is
\begin{align}\label{delta.term:6}
&f^{(1)}(x)(\sigma(x)(\nu\ast b))^1\sum_{i=1}^d\frac{1}{3!}b_i {q}_i^{(3)}(x,0)(\nu)^3 \left( \frac{1}{2} \right)^{d-1} \\
&+f^{(1)}(x)(\sigma(x)(\nu\ast b))^1\sum_{i \neq j}\frac{1}{2!}b_iq_i^{(2)}(x,0)(\nu)^2 b_jq_j^{(1)}(x,0)(\nu)^1 \left( \frac{1}{2} \right)^{d-2} \nonumber \\
&+f^{(1)}(x)(\sigma(x)(\nu\ast b))^1\sum_{i \neq j \neq k \neq i} b_iq_i^{(1)}(x,0)(\nu)^1 b_jq_j^{(1)}(x,0)(\nu)^1 b_kq_k^{(1)}(x,0)(\nu)^1  \left( \frac{1}{2} \right)^{d-3} \nonumber \\
&+\frac{1}{2!}f^{(2)}(x)(\sigma(x)(\nu\ast b))^2\sum_{i=1}^d \frac{1}{2!}b_iq_i^{(2)}(x,0)(\nu)^2 \left( \frac{1}{2} \right)^{d-1} \nonumber \\
&+\frac{1}{2!}f^{(2)}(x)(\sigma(x)(\nu\ast b))^2\sum_{i \neq j} b_iq_i^{(1)}(x,0)(\nu)^1b_jq_j^{(1)}(x,0)(\nu)^1 \left( \frac{1}{2} \right)^{d-2}  \nonumber \\
&+\frac{1}{3!}f^{(3)}(x)(\sigma(x)(\nu\ast b))^3\sum_{i=1}^d b_iq_i^{(1)}(x,0)(\nu)^1 \left( \frac{1}{2} \right)^{d-1}  \nonumber\\
&+\frac{1}{4!}f^{(4)}(x)(\sigma(x)(\nu\ast b))^4\left( \frac{1}{2} \right)^{d}.  \nonumber
\end{align}
Let us consider each of those 7 terms separately. For the first term, we write
\begin{align*}
f^{(1)}(x)(\sigma(x)(\nu\ast b))^1\sum_{i=1}^d\frac{1}{3!}b_iq_i^{(3)}(x,0)(\nu)^3 \left( \frac{1}{2} \right)^{d-1}& \\
=\frac{1}{3! 2^{d-1}}\!\!\!\!\!\!\!\!\sum_{i,k,i_1,i_2,i_3=1}^d\!\!\!\!\!\!\!\partial_kf(x)b_k\sigma_{k,k}(x) \nu_k b_i\partial^{\xi}_{i_1}\partial^{\xi}_{i_2}\partial^{\xi}_{i_3}q_i(x,0) \nu_{i_1}&\nu_{i_2}\nu_{i_3}.
\end{align*}
Summing over $b \in \{ \pm 1 \}^d$ and using \eqref{delta.term:3}, we get 
\begin{align*}
&\frac{1}{3! 2^{d-1}}\sum_{i,k,i_1,i_2,i_3=1}^d\partial_kf(x)\sigma_{k,k}(x) \nu_k\partial^{\xi}_{i_1}\partial^{\xi}_{i_2}\partial^{\xi}_{i_3}q_i(x,0) \nu_{i_1} \nu_{i_2} \nu_{i_3}\sum_{b \in\{ \pm 1 \}^d}b_ib_k\\
&=\frac{2^d}{3! 2^{d-1}}\sum_{i,i_1,i_2,i_3=1}^d\partial_if(x)\sigma_{i,i}(x) \nu_i\partial^{\xi}_{i_1}\partial^{\xi}_{i_2}\partial^{\xi}_{i_3}q_i(x,0) \nu_{i_1}\nu_{i_2}\nu_{i_3}.
\end{align*}
Taking the expectation over $\nu$, we obtain 
\begin{equation*}
\frac{1}{3} \sum_{i,i_1,i_2,i_3=1}^d\partial_if(x)\partial^{\xi}_{i_1}\partial^{\xi}_{i_2}\partial^{\xi}_{i_3}q_i(x,0)\sigma_{i,i}(x)\mathbb{E}_{\nu}\left[ \nu_i \nu_{i_1}\nu_{i_2} \nu_{i_3}\right].
\end{equation*}
Consider all cases regarding the equality or inequality of 
$i$,$i_1$,$i_2$,$i_3$. There are three distinctions: The first case is that at least one of the $i$,$i_1$,$i_2$,$i_3$ is different from the other three, in which case $\mathbb{E}\left[ \nu_i \nu_{i_1} \nu_{i_2} \nu_{i_3} \right]=0$. The second case is that all of the $i, i_1, i_2, i_3$ are equal. Then $\mathbb{E}\left[ \nu_i \nu_{i_1} \nu_{i_2} \nu_{i_3} \right]= \mathbb{E}\left[ \nu_i^4 \right]=3$.  We then get  
\begin{equation}\label{weird.counting:1}
\frac{1}{3} \sum_{i=1}^d\partial_if(x)\partial_{i}^{\xi}\partial_{i}^{\xi}\partial_{i}^{\xi}q_i(x,0) 3 \sigma_{i,i}(x)= \sum_{i=1}^d\partial_if(x)\partial_{i}^{\xi}\partial_{i}^{\xi}\partial_{i}^{\xi}q_i(x,0)  \sigma_{i,i}(x). 
\end{equation}
The last case is that two of the $\{ i, i_1, i_2, i_3 \}$ are equal and the other two are equal and different from the other two (that is, there are two pairs). In this case, by counting three times (since either of $i_1,i_2,i_3$ could be the one equal to $i$) we get
\begin{equation}\label{weird.counting:2}
\frac{1}{3}\sum_{i \neq k}^d\partial_if(x)\partial^{\xi}_{i}\partial^{\xi}_{k}\partial^{\xi}_{k}q_i(x,0) 3 \sigma_{i,i}(x)  =\sum_{i \neq k}^d\partial_if(x)\partial^{\xi}_{i}\partial^{\xi}_{k}\partial^{\xi}_{k}q_i(x,0)  \sigma_{i,i}(x).  
\end{equation}
     Adding \eqref{weird.counting:1} and \eqref{weird.counting:2}, we get that the first term of \eqref{delta.term:6} is equal to
     \begin{equation*}
         \sum_{i,k=1}^d\partial_if(x)\partial^{\xi}_{i}\partial^{\xi}_{k}\partial^{\xi}_{k}q_i(x,0)  \sigma_{i,i}(x), 
     \end{equation*}
     which is the first term of $\mathcal{A}_2$ as in \eqref{second.derivative.numerical.scheme:1}.
     
     For the second term of \eqref{delta.term:6} we write
\begin{align*}
&\frac{1}{2^{d-1}}f^{(1)}(x)(\sigma(x)(\nu\ast b))^1\sum_{i \neq j}b_iq_i^{(2)}(x,0)(\nu)^2 b_jq_j^{(1)}(x,0)(\nu)^1  \\
&=\frac{1}{2^{d-1}}\sum_{\substack{i,j,k \\ i \neq j}} \partial_kf(x)b_k \sigma_{k,k}(x) \nu_k b_iq_i^{(2)}(x,0)(\nu)^2 b_jq_j^{(1)}(x,0)(\nu)^1. 
\end{align*} 
Recall that due to Lemma \ref{lemma.betas:1} we have
\begin{equation}\label{delta.term:7}
\sum_{b \in \{ \pm 1 \}^d}b_ib_jb_k=0,
\end{equation}
therefore, on summing over $b \in \{ \pm 1 \}^d$, we see that this term is zero.

The third term in \eqref{delta.term:6} is
\begin{equation}\label{delta.term:8}
f^{(1)}(x)(\sigma(x)(\nu\ast b))^1 \sum_{i \neq j \neq k \neq i}b_iq_i^{(1)}(x,0)(\nu)^1 b_jq_j^{(1)}(x,0)(\nu)^1 b_kq_k^{(1)}(x,0)(\nu)^1  \left( \frac{1}{2} \right)^{d-3}.
\end{equation}
Expanding 
\begin{equation*}
f^{(1)}(x)(\sigma(x)(\nu\ast b))^1=\sum_{i_1=1}^d\partial_{i_1}f(x)b_{i_1}\sigma_{i_1,i_1}(x)\nu_{i_1},
\end{equation*}
we see that when we sum over all $b \in \{ \pm 1 \}^d$, inside the sum of (\ref{delta.term:8}) there will appear the quantity $\sum_{b \in \{ \pm 1 \}^d }b_{i_1}b_ib_jb_k$, where $i \neq j \neq k \neq i$. Due to  Lemma \ref{lemma.betas:1} this term will be zero, so (\ref{delta.term:8}) will also be zero.

For the fourth term of \eqref{delta.term:6} we write
\begin{align*}
&\frac{1}{2^{d+1}}f^{(2)}(x)(\sigma(x)(\nu\ast b))^2\sum_{i=1}^d b_iq_i^{(2)}(x,0)(\nu)^2\\
&=\frac{1}{2^{d+1}}\sum_{i,j,k=1}^d\partial_j\partial_kf(x)b_j \sigma_{j,j}(x)\nu_j b_k \sigma_{k,k}(x) \nu_k b_iq_i^{(2)}(x,0)(\nu)^2.
\end{align*}
When summing over $b \in \{ \pm 1 \}^d$ we get a term of the form $\sum_{b \in \{ \pm 1 \}^d}b_ib_jb_k$ which is zero due to \eqref{lemma.betas:1}, so this term is also zero.

For the fifth term of (\ref{delta.term:6}) we write
\begin{equation*}
\begin{split}
&\frac{1}{2^{d-1}}f^{(2)}(x)(\sigma(x)(\nu\ast b))^2\sum_{i \neq j}b_iq_i^{(1)}(x,0)(\nu)^1b_jq_j^{(1)}(x,0)(\nu)^1 \\
&\!\!\!=\!\!\frac{1}{2^{d-1}}\!\!\!\!\!\sum_{\substack{i,j,i_1,i_2 \\ i \neq j}} \!\!\!\!\partial_{i_1}\partial_{i_2}f(x)b_{i_1}\sigma_{i_1,i_1}\!(x) \nu_{i_1}b_{i_2}\sigma_{i_2,i_2}\!(x)\nu_{i_2}b_iq_i^{(1)}\!(x,0)(\nu)^1b_jq_j^{(1)}\!(x,0)(\nu)^1\!\!.
\end{split}
\end{equation*}
On summing over $b \in \{ \pm 1 \}^d$ we get terms of the form $\sum_{b \in \{ \pm  1 \}^d }b_{i_1}b_{i_2}b_ib_j$ with $i \neq j$. Due to Lemma \ref{lemma.betas:1}, for this not to be zero we need $i_1=i$ and $i_2=j$ or the other way around, in which cases the sum equals $2^d$. Due to symmetry, after summing over $b \in \{ \pm 1 \}^d$, we can write the last expression as 
\begin{align*}
&\frac{2\cdot 2^d}{2^{d-1}}\sum_{ i \neq j} \partial_{i}\partial_{j}f(x)\sigma_{i,i}(x) \nu_{i}\sigma_{j,j}(x) \nu_{j}q_i^{(1)}(x,0)(\nu)^1q_j^{(1)}(x,0)(\nu)^1=\\
&4\sum_{\substack{i,j,i_1,i_2 \\ i \neq j}}\partial_{i}\partial_{j}f(x)\sigma_{i,i}(x)\nu_{i}\sigma_{j,j}(x) \nu_{j} \partial^{\xi}_{i_1}q_i(x,0)  \nu_{i_1}\partial^{\xi}_{i_2}q_j(x,0) \nu_{i_2}.
\end{align*}
On taking the expectation over $\nu$ we get
\begin{align*} &4\sum_{\substack{i,j,i_1,i_2 \\ i \neq j}}\partial_{i}\partial_{j}f(x) \partial^{\xi}_{i_1}q_i(x,0)\partial^{\xi}_{i_2}q_j(x,0) \sigma_{i,i}(x) \sigma_{j,j}(x) \mathbb{E}\left[ \nu_i \nu_j \nu_{i_1} \nu_{i_2} \right].
\end{align*}
In order for the expectations not to be zero, we must have $i_1=i$ and $i_2=j$, or $i_1=j$ and $i_2=i$. The first case gives
\begin{equation*}
4 \sum_{i \neq j} \partial_{i}\partial_{j}f(x) \partial^{\xi}_{i}q_i(x,0)\partial^{\xi}_{j} q_j(x,0) \sigma_{i,i}(x) \sigma_{j,j}(x), 
\end{equation*}
whereas the second case gives 
\begin{equation*}
4 \sum_{i \neq j} \partial_{i}\partial_{j}f(x) \partial^{\xi}_{j}q_i(x,0)\partial^{\xi}_{i}q_j(x,0) \sigma_{i,i}(x) \sigma_{j,j}(x).
\end{equation*}
Adding these two terms, using \eqref{condition.p.derivative:0}, and the fact that $\sigma$ is diagonal, we get the second and third term of $\mathcal{A}_2$ as in \eqref{second.derivative.numerical.scheme:1}.

For the sixth term of \eqref{delta.term:6} we write
\begin{equation}\label{delta.term:6.5}
    \begin{split}
&\frac{1}{3 \cdot 2^d}f^{(3)}(x)(\sigma(x)(\nu\ast b))^3\sum_{i=1}^db_iq_i^{(1)}(x,0)(\nu)^1  \\
=&\frac{1}{3 \cdot 2^d}\!\!\!\!\!\!\sum_{i_1,i_2,i_3,i=1}^d\!\!\!\!\!\!\partial_{i_1}\partial_{i_2}\partial_{i_3}f(x)b_{i_1}\sigma_{i_1,i_1}\!(x) \nu_{i_1}b_{i_2}\sigma_{i_2,i_2}\!(x) \nu_{i_2} b_{i_3}\sigma_{i_3,i_3}(x) \nu_{i_3}b_iq_i^{(1)}\!(x,0)(\nu)^1. \nonumber
\end{split}
\end{equation}     
On summing over $b \in \{ \pm 1 \}^d$ we get terms of the form $\sum_{b \in \{ \pm 1 \}^d}b_{i_1}b_{i_2}b_{i_3}b_i$. Due to Lemma \ref{lemma.betas:1}, there are two ways for this to not be zero (and be equal to $2^d$). We need either $i=i_1=i_2=i_3$, or $i$ is equal to one of the $i_j$'s and the other two are also equal (but not equal to $i$). We consider the two cases separately. The first case gives
\begin{equation*}
\frac{2^d}{3 \cdot 2^d}\sum_{i=1}^d\partial_{i}^3f(x)\sigma_{i,i}^3(x) \nu_{i}^3q_i^{(1)}(x,0)(\nu)=\frac{1}{3} \sum_{i,k=1}^d\partial_{i}^3f(x)\sigma_{i,i}^3(x) \nu_{i}^3 \partial^{\xi}_kq_i(x,0) \nu_k,
\end{equation*}
and on taking the expectation over $\nu$ we get
\begin{equation*}
\frac{1}{3}\sum_{i,k=1}^d\partial_i^3f(x)\partial^{\xi}_kq_i(x,0) \sigma_{i,i}^3(x)  \mathbb{E}\left[ \nu_i^3 \nu_k \right],
\end{equation*}
and since $\mathbb{E}\left[ \nu_i^4 \right]=3$ and for $k \neq i, \mathbb{E}\left[ \nu_i^3 \nu_k \right]=0$, we get
\begin{equation}\label{delta.term:9}
\sum_{i=1}^d\partial_i^3f(x)\partial^{\xi}_iq_i(x,0) \sigma_{i,i}^3(x).
\end{equation}
When we consider the second case for (\ref{delta.term:6.5}), where $i$ is equal to one of the $\{ i_1,i_2,i_3 \}$ and the other two are also equal, due to the symmetry between $i_1,i_2,i_3$, when summing over $b \in \{ \pm 1 \}^d$ we get
\begin{align*}
&\frac{2^d\cdot 3}{ 3 \cdot 2^d}\sum_{i \neq k}\partial_{i}\partial_{k}^2f(x)\sigma_{i,i}(x) \nu_{i} \sigma_{k,k}^2(x) \nu_{k}^2q_i^{(1)}(x,0)(\nu)\\
&=\sum_{\substack{i,j,k \\ i \neq k}} \partial_{i}\partial_{k}^2f(x)\sigma_{i,i}(x) \nu_{i}\sigma_{k,k}^2(x) \nu_{k}^2\partial^{\xi}_jq_i(x,0) \nu_j.
\end{align*}
On taking expectations with respect to $\nu$ we get
\begin{align*}
&\sum_{\substack{i,j,k \\ i \neq k}} \partial_{i}\partial_{k}^2f(x)\partial^{\xi}_jq_i(x,0) \sigma_{i,i}(x)\sigma_{k,k}^2(x) \mathbb{E}\left[ \nu_i \nu_j \nu_k^2 \right].
\end{align*}
When $j \neq i$, $\mathbb{E}\left[ \nu_i \nu_j \nu_k^2 \right]=0$, whereas $\mathbb{E}\left[ \nu_i^2 \nu_k^2 \right]=1$ for $j = i$. Therefore the last expression is equal to
\begin{equation*}
\sum_{i \neq k} \partial_{i}\partial_{k}^2f(x)\partial^{\xi}_iq_i(x,0) \sigma_{i,i}(x)  \sigma_{k,k}^2(x).
\end{equation*}
Adding this to \eqref{delta.term:9}, we see that the sixth term in \eqref{delta.term:6} is 
\begin{equation*}
\sum_{i, k=1}^d \partial_{i}\partial_{k}^2f(x)\partial^{\xi}_iq_i(x,0) \sigma_{i,i}(x)   \sigma_{k,k}^2(x),
\end{equation*}
and using \eqref{condition.p.derivative:0} and the fact that $\sigma$ is diagonal, we get the fourth term of $\mathcal{A}_2$ as in \eqref{second.derivative.numerical.scheme:1}.
     
Finally, for the seventh term in \eqref{delta.term:6} we write
\begin{align*}
&\frac{1}{24 \cdot 2^d}f^{(4)}(x)(\sigma(x)(\nu\ast b))^4\\
&= \frac{1}{24 \cdot 2^d}\!\!\!\!\!\!\!\sum_{i_1,i_2,i_3,i_4=1}^d \!\!\!\!\!\!\!\partial_{i_1}\partial_{i_2}\partial_{i_3}\partial_{i_4}f(x) b_{i_1}\sigma_{i_1,i_1}(x)
\nu_{i_1}b_{i_2}\sigma_{i_2,i_2}(x)
\nu_{i_2 }b_{i_3}\sigma_{i_3,i_3}(x) \nu_{i_3} b_{i_4} \sigma_{i_4,i_4}(x) \nu_{i_4}.
\end{align*}
As in the previous case, when we sum over $b \in \{ \pm 1 \}^d$, we get a term of the form $\sum_{b \in \{ \pm 1 \}^d}b_{i_1}b_{i_2}b_{i_3}b_{i_4}$ which is non-zero (and equal to $2^d$) in only two cases: either $i_1=i_2=i_3=i_4$, or there are two pairs of equal $i_j$'s. In the first case we get
\begin{align*}
\frac{2^d}{24 \cdot 2^d}\sum_{i=1}^d \partial_{i}^4f(x)\sigma_{i,i}^4(x)
\nu_{i}^4,
\end{align*}
and when we take the expectation over $\nu$ we get 
\begin{align}\label{delta.term:10}
\frac{1}{8}\sum_{i=1}^d\partial_i^4f(x)  \sigma_{i,i}^4(x).
\end{align}
In the second case, due to symmetry between the $i_j$'s, and since there are 3 ways to pair up the 4 $i_j$'s, we get
\begin{align*}
\frac{2^d\cdot 3}{24 \cdot 2^d} \sum_{i \neq k}  \partial_i^2\partial_k^2f(x)\sigma_{i,i}^2(x) \nu_i^2 \sigma_{k,k}^2(x) \nu_k^2.
\end{align*}
Taking the expectation over $\nu$ we get
\begin{equation*}
\frac{1}{8} \sum_{i \neq k} \partial_i^2\partial_k^2f(x)  \sigma_{i,i}^2(x) 
\sigma_{k,k}^2(x).
\end{equation*}
Adding this to \eqref{delta.term:10}, we get that the seventh term in \eqref{delta.term:6} is
\begin{equation*}
\frac{1}{8} \sum_{i,k=1}^d \partial_i^2\partial_k^2f(x) \sigma_{i,i}^2(x) \sigma_{k,k}^2(x),
\end{equation*}
which is the last term of $\mathcal{A}_2$ as in \eqref{second.derivative.numerical.scheme:1}.
     
Combining all the seven cases above, we see that the coefficient of $\Delta t^2$ is $\mathcal{A}_2f(x)$ as in \eqref{second.derivative.numerical.scheme:1}.

Finally, we consider the coefficient of order $\Delta t^3$ and higher. We recall that from Assumptions \ref{ass.drift.general:1}, \ref{ass.volatility.general:1} and \ref{regularity.coefficients:1} we have polynomial bounds on all $|\mu_i|$, $\sigma_{i,i}(x)$ and all the derivatives of $q_i$ and $f$. Using also the fact that all the moments of $\sigma(x) \nu$ are finite, along with the fact the $\Delta t \in (0,1)$, we get that the higher order terms can be bounded above in absolute value by $K_2(x) \Delta t^3$, for some function $K_2 \in C^{\infty}_P(\mathbb{R}^d)$. 
\end{proof}

\section{Proof of Theorem \ref{Generalisation.theorem:1}}\label{proof.of.regularity:00}

The proof proceeds by estimating the moments associated to \eqref{eq:SDE_appendix} and its derivatives. 
We heavily rely on the theory of random diffeomorphisms induced by SDEs \citep{kunita2019stochastic}, which in general requires global Lipschitz conditions. We circumvent this problem by taming the drift outside of a ball of large enough radius (see \eqref{eq:tamed sde} below). We also note that, in order to shorten the notation, in the proof below the same letter $C$ will be used to denote potentially different upper bounding constants.

\begin{proof}[Proof of Theorem \ref{Generalisation.theorem:1}]
To ease notation, let us define $a(x) := \sigma(x) \sigma^\top(x)$, for $x \in \mathbb{R}^d$. By Lipschitz continuity of $\sigma$ (Assumption \ref{ass.volatility.general:1}), $a$ grows at most quadratically,
\begin{equation}
\label{eq:a bound}
\Vert a(x) \Vert_F \le C(1+ \Vert x \Vert^2), \qquad \qquad x \in \mathbb{R}^d,   
\end{equation}
for an appropriate constant $C>0$.
 For $p \ge 2$, we will make repeated use of the derivatives 
\begin{equation}
\label{eq:Frobenius d1}
\frac{\partial \Vert A \Vert_F^p}{\partial A_{i_1,\ldots,i_n}} = p \Vert A \Vert^{p-2}_F A_{i_1,\ldots,i_n}
\end{equation}
and
\begin{align}
\label{eq:Frobenius d2}
\frac{\partial^2 \Vert A \Vert_F^p }{\partial A_{i_1,\ldots,i_n} \partial A_{i'_1,\ldots,i'_n}} =  & p \Vert A \Vert_F^{p-2} \delta_{(i_1,\ldots,i_n),(i'_1,\ldots,i'_n)} \\ & + p (p-2) \Vert A \Vert_F^{p-4} A_{i_1,\ldots,i_n} A_{i_1',\ldots,i_n'}, 
\nonumber
\end{align}
where in \eqref{eq:Frobenius d2}, $\delta_{(i_1,\ldots,i_n),(i'_1,\ldots,i'_n)}$ is equal to one if all indices coincide ($i_1 = i'_1$, $i_2 = i'_2$, ..., $i_n = i'_n$), and zero otherwise. 

\noindent\emph{Estimates on $Y^x_t$.}
By It{\^o}'s formula and \eqref{eq:Frobenius d1}-\eqref{eq:Frobenius d2}, we have a.s.
\begin{align}\label{ito.no.derivative:1}
\Vert Y^x_t \Vert^p = & \Vert x \Vert^p + \int_0^t p \Vert Y^x_s \Vert^{p-2} Y^x_s  \, \mu\left( Y^x_s \right) ds + \int_0^t \Vert Y^x_s \Vert^{p-2} Y^x_s  \, \sigma\left( Y^x_s \right) dW_s \nonumber \\
 &+ \int_0^t \frac{1}{2} p \Vert Y^x_s \Vert^{p-2} \mathrm{Tr}\,a(Y^x_s) \, \mathrm{d}s 
 + \int_0^t \frac{1}{2}p (p-2) \Vert Y_s \Vert^{p-4} Y_s^\top a(Y^x_s) Y^x_s \, \mathrm{d}s,
\end{align}
for all $t \in [0,T]$. In particular, if we introduce for any $R \in \mathbb{N}$ the stopping time 
\begin{equation}\label{stopping.time.tau.R:0}
    \tau_R=\inf \left\{ t \geq 0 : \ \| Y^x_t \| \geq R \right\},
\end{equation}
 and denote $t \wedge \tau_R  = \min \{  t , \tau_R \}$ we get a.s. for all $R \in \mathbb{N}$ and $t \in [0,T]$,
\begin{align*}
&\Vert Y^x_{t \wedge \tau_R} \Vert^p =  \Vert x \Vert^p + \int_0^{t \wedge \tau_R} p \Vert Y^x_s \Vert^{p-2} Y^x_s  \, \mu\left( Y^x_s \right) ds + \int_0^{t \wedge \tau_R} \Vert Y^x_s \Vert^{p-2} Y^x_s  \, \sigma\left( Y^x_s \right) dW_s \\
 &+ \int_0^{t \wedge \tau_R} \frac{1}{2} p \Vert Y^x_s \Vert^{p-2} \mathrm{Tr}\,a(Y^x_s) \, \mathrm{d}s 
 + \int_0^{t \wedge \tau_R} \frac{1}{2}p (p-2) \Vert Y^x_s \Vert^{p-4} \left(Y_s^x\right)^{\top} a(Y^x_s) Y^x_s \, \mathrm{d}s. \nonumber
\end{align*}
We observe that $\left(\int_0^{t \wedge \tau_R} \Vert Y^x_s \Vert^{p-2} Y^x_s  \, \sigma\left( Y^x_s \right) dW_s \right)_{t \geq 0}$ is a true martingale, since the integrand is bounded, therefore it has zero expectation. Taking expectations, using the one-sided Lipschitz condition on $\mu$ (Assumption \ref{ass.drift.general:1}) as well as \eqref{eq:a bound}, and using Fubini's theorem to exchange integral and expectation we see that there exists a constant $C>0$ such that
\begin{equation*}
 \mathbb{E} \left[ \Vert Y^x_{t \wedge \tau_R} \Vert^p \right] \le \Vert x \Vert^p + C \int_0^{t} \mathbb{E}[\Vert Y^x_{s \wedge \tau_R} \Vert^{p-2}] \, \mathrm{d}s + C \int_0^{t} \mathbb{E}[\Vert Y^x_{s \wedge \tau_R} \Vert^p] \, \mathrm{d}s ,   
\end{equation*}
for all $t \in [0,T]$. In particular, 
\begin{align*}
    \mathbb{E} \left[ \Vert Y^x_{t \wedge \tau_R} \Vert^p \right] &\le  \Vert x \Vert^p + C \int_0^{t} \mathbb{E}[\Vert Y^x_{s \wedge \tau_R} \Vert^p] +1 \, \mathrm{d}s + C \int_0^{t} \mathbb{E}[\Vert Y^x_{s \wedge \tau_R} \Vert^p] \, \mathrm{d}s \\
    &\le  \Vert x \Vert^p + CT+ 2 C  \int_0^{t} \mathbb{E}[\Vert Y^x_{s \wedge \tau_R} \Vert^p] \, \mathrm{d}s. 
\end{align*}
Gronwall's inequality then implies
\begin{equation}
\label{eq:Y.R bound}
\mathbb{E}[\Vert Y^x_{t \wedge \tau_R} \Vert^p] \le C_p \left( \Vert x \Vert^p + 1 \right), \qquad x \in \mathbb{R}^d,
\end{equation}
for an appropriate constant $C_p>0$, independent of $t \in [0,T]$ and $R \in \mathbb{N}$. In particular, for any set $A \in \mathcal{B}(\mathbb{R}^d)$, using H{\"o}lder's inequality inequality for any $p'>p$, we have that
\begin{equation*}
\mathbb{E}[\Vert Y^x_{t \wedge \tau_R} \Vert^p 1_{A}] \leq \mathbb{E}[\Vert Y^x_{t \wedge \tau_R} \Vert^{p'}]^{1/p'} \mathbb{P}\left( A \right) ^{1-1/p'}  \leq C_{p'} \left( \Vert x \Vert^{p'} +1 \right) \mathbb{P}\left( A \right)^{1-1/p'}.
\end{equation*}
Consequently, the family of random variables $\left( \Vert Y^x_{t \wedge \tau_R} \Vert^p\right)_{R \in \mathbb{N}}$ is uniformly integrable in the sense of \cite[Section 4.5]{bogachev2007measure}. Since a.s. $\lim_{R \rightarrow \infty} Y^x_{t \wedge \tau_R}=Y^x_t$, using Vitali's theorem  \citep[Theorem 4.5.4]{bogachev2007measure} we get
\begin{equation*}
  \lim_{R \rightarrow \infty}  \mathbb{E}[\Vert Y^x_{t \wedge \tau_R} \Vert^p]  = \mathbb{E}[\Vert Y^x_{t} \Vert^p], 
\end{equation*}
and using (\ref{eq:Y.R bound}) we get
\begin{equation}
\label{eq:Y bound}
\mathbb{E}[\Vert Y^x_{t} \Vert^p] \le C_p \left( \Vert x \Vert^p + 1 \right), \qquad x \in \mathbb{R}^d.
\end{equation}
Together with $f \in C_P^l(\mathbb{R}^d)$, the bound \eqref{eq:moment bound u} follows. 

 \emph{Estimates on $\nabla_x Y^x_t$.}
 In order to deal with the non-Lipschitz character of the drift $\mu$, we first establish results on the tamed version \eqref{eq:tamed sde} below.
For any cut-off radius $r \in \mathbb{N}$, we fix a drift $\mu^{(r)} \in C_b^{\infty}(\mathbb{R}^d;\mathbb{R}^d)$ so that $\mu^{(r)}(x) = \mu(x)$ if $\Vert x\Vert \le r$, %$\mu^{(r)}$ is Lipscchitz
, and such that 
\begin{equation}
\label{eq:r contracting}
\xi^\top \nabla \mu^{(r)}(x) \xi \le C \Vert \xi \Vert^2, \qquad \text{for all}\, x, \xi \in \mathbb{R}^d, 
\end{equation}
with the same constant $C$ as in \eqref{eq:drift contracting}, independent of $r$ (the drifts $\mu^{(r)}$ can straightforwardly be obtained by modifying $\mu$ outside of $\{x \in \mathbb{R}^d : \Vert x \Vert \le r\}$, in such a way that \eqref{eq:drift contracting} is not violated). We consider the SDEs
\begin{equation}
\label{eq:tamed sde}
\mathrm{d}Y^{x,(r)}_t = \mu^{(r)}(Y^{x,(r)}_t) \, \mathrm{d}t + \sigma(Y^{x,(r)}_t) \, \mathrm{d}W_t, \qquad Y^{x,(r)}_0 = x.
\end{equation}
By \cite[Theorem 3.4.1]{kunita2019stochastic}, there exists a modification of  $(Y^{x,(r)}_t)_{t \in [0,T]}$ which is continuously differentiable with respect to the initial condition, for all $t \in [0,T]$. The derivative $\nabla_x Y_t^{x,(r)} \in \mathbb{R}^{d\times d}$ satisfies the SDEs %{\bf not sure what $\sigma_k$ is, same on following eq.}
%\begin{equation}
%\nonumber
%\mathrm{d} \nabla_x Y^{x,(r)}_t\! =\! \nabla_x Y^{x,(r)}_t \nabla \mu^{(r)}(Y^{(r)}_t) \, \mathrm{d}t + \nabla_x Y^{x,(r)}_t \nabla_x \sigma(Y^{x,(r)}_t) \, \mathrm{d}W_t, \quad \nabla_x Y^{x,(r)}_0 = I_{d}, 
%\end{equation}
%or in index form 
\begin{equation}
    \mathrm{d} \partial_j Y^{x,(r)}_{t, i} = \sum_{l=1}^d \partial_l \mu^{(r)}_{i}(Y^{x,(r)}_t) \partial_j Y^{(x,r)}_{t, l} \, \mathrm{d}t + \sum_{k=1}^d \sum_{l=1}^d \partial_l \sigma^{(r)}_{i,k}(Y^{x,(r)}_t) \partial_j Y^{x,(r)}_{t,l} \, \mathrm{d}W_t^k, %\qquad i,j = 1,\ldots,d.
\end{equation}
for all $i,j=1, \ldots , d$, with initial condition $
    \nabla_x Y^{x,(r)}_0 = I_{d}$.
%where we have temporarily dropped the $r$-superscipts on $\mu$ and $Y_t$ to simplify the notation. 
%In order to apply It{\^o}'s formula to the Frobenius norm $\Vert \nabla_x Y_t^{x,(r)} \Vert_F$, we compute the evolution of the  quadratic variation: for all $i,i',j,j'=1, \ldots , d$ we have
%\begin{equation}
%\mathrm{d} \langle \partial_{j} Y_{\cdot, i}^{x,(r)}, \partial_{j'} Y_{\cdot, i'}^{x,(r)} \rangle_t = \sum_{k=1}^d \sum_{l,l'=1}^d \partial_l \sigma_{i,k}(Y_t^{x,(r)}) \partial_{l'} \sigma_{i',k}(Y^{x,(r)}_t) \partial_j Y^{x,(r)}_{t,l} \partial_{j'}Y^{x,(r)}_{t,l'} \, \mathrm{d}t.
%\end{equation}
Based on \eqref{eq:Frobenius d1} and \eqref{eq:Frobenius d2}, using It{\^o}'s formula we get 
\begin{subequations}
\label{eq:ItoDY}
\begin{align}
\nonumber
 \mathrm{d} \Vert \nabla_x  Y^{x,(r)}_t \Vert_F^p&  = p \,  \Vert \nabla_x Y^{x,(r)}_t \Vert_F^{p-2} \sum_{i,j,l = 1}^d \partial_l \mu_i(Y^{x,(r)}_t)\partial_jY_t^l  \partial_j Y_t^i \ \mathrm{d}t
 \\
 \label{eq:DY2}
 &\hspace{-1.4cm}
  + \frac{p}{2}  \Vert \nabla_x Y^{x,(r)}_t \Vert_F^{p-2} \sum_{i,j,k=1}^d \left( \sum_{l=1}^d \partial_l\sigma_{i,k}(Y_t^{x,(r)}) \partial_jY_{t,l}^{x,(r)}\right)^2 \mathrm{d}t 
  %\sum_{i,j,l,l' = 1}^d \partial_l \sigma_i(Y^{x,(r)}_t) \partial_{l'} \sigma_i(Y_t) \partial_j Y_t^l \partial_j Y_t^{l'}\right] \mathrm{d}t
 \\
 \label{eq:DY3}
 &\hspace{-1.4cm} + \frac{p(p-2)}{2} \Vert \nabla_x Y^{x,(r)}_t \Vert_F^{p-4}\!\! \\
 &\hspace{-1.85cm}\sum_{i,i', j,j',l,l',k = 1}^d \partial_l\sigma_{i,k}(Y^{x,(r)}_t) \partial_{l'}\sigma_{i',k}(Y^{x,(r)}_t) \partial_j Y^{x,(r)}_{t,l} \partial_{j'}Y^{x,(r)}_{t,l'} \partial_j Y^{x,(r)}_{t,i} \partial_{j'} Y^{x,(r)}_{t,i'} \!\! \ \mathrm{d}t \nonumber \\
 & + p \,  \Vert \nabla_x Y^{x,(r)}_t \Vert_F^{p-2} \sum_{k=1}^d\sum_{i,j,l = 1}^d \partial_l \sigma_{i,k}(Y^{x,(r)}_t)\partial_jY_t^l  \partial_j Y_t^i \ \mathrm{d}W^k_t. \nonumber
 \end{align}
\end{subequations}
We observe that $\Vert \nabla_x Y_0^{x,(r)}\Vert_F^p = d^{p/2}$. Furthermore, for any $t \in [0,T]$, we have the bound
\begin{equation*}
\int_0^t %\mathbb{E} \left[
\Vert \nabla_x Y^{x,(r)}_s \Vert_F^{p-2} \sum_{i,j,l = 1}^d \partial_l \mu^i(Y^{x,(r)}_s)\partial_jY_s^l  \partial_j Y_s^i %\right]
\,\mathrm{d}s \le C \int_0^t %\mathbb{E} \left[
\Vert \nabla_x Y^{x,(r)}_s \Vert_F^{p} %\right] 
\,\mathrm{d}s,
\end{equation*}
using \eqref{eq:r contracting}, and where $C>0$ does not depend on $x$, $r$ or $t$. By the boundedness of $\nabla \sigma$ (Assumption \ref{ass:der assumption}), we have a similar bound on the terms in \eqref{eq:DY2} and \eqref{eq:DY3}. Taking expectations, using a similar argument as in the discussion below (\ref{ito.no.derivative:1}) (by introducing appropriate stopping times $\tau_R$) to argue that any integral term with respect to the Brownian motion will have zero expectation, and exchanging integral and expectation using Fubini, we obtain for all $R \in \mathbb{N}$
\begin{equation}
\label{eq:nabla Y bound}
\mathbb{E} \left[\Vert \nabla_x Y^{x,(r)}_{t \wedge \tau_R} \Vert_F^p\right]  \le d^{p/2} + C(T) \int_0^t \mathbb{E} \left[\Vert \nabla_x Y^{x,(r)}_{s \wedge \tau_R} \Vert_F^{p} \right] \mathrm{d}s,
\end{equation}
for all $t \in [0,T]$, where $C$ does not depend on $r, R \in \mathbb{N}$, $t \in [0,T]$ or $x \in \mathbb{R}^d$. By Gronwall's Lemma, we get that $\mathbb{E} \left[\Vert \nabla_x Y^{x,(r)}_{t \wedge \tau_R} \Vert_F^p\right]$ is bounded uniformly in $r , R \in \mathbb{N}$, $t \in [0,T]$ and $x \in \mathbb{R}^d$. Therefore, letting $R \rightarrow \infty$ and using Vitali's theorem, we conclude that  $\mathbb{E} \left[\Vert \nabla_x Y^{x,(r)}_t \Vert_F^p \right]$ is bounded uniformly in $r \in \mathbb{N}$, $t \in [0,T]$ and $x \in \mathbb{R}^d$.

\emph{Controlling the gradient of the cut-off semigroup.} Defining the cut-off version of $u_t$,
\begin{equation}
\label{eq:utr}
 u_t^{(r)}(x) := \mathbb{E} [f(Y_t^{x,(r)})], \qquad x \in \mathbb{R}^d,  
\end{equation}
we next argue that $u_t^{(r)}$ is differentiable, with 
\begin{equation}
\label{eq:nabla u}
\nabla_x u_t^{(r)}(x) = \mathbb{E}\left[\nabla_x Y_t^{x,(r)} \nabla f(Y^{x,(r)}_t) \right], 
\end{equation}
that is, we can apply the gradient operation to \eqref{eq:utr} and exchange expectation and derivatives. To that end, notice that $x \mapsto f(Y_t^{x,(r)})$ is differentiable almost surely, with derivative $\nabla_x Y_t^{x,(r)}\nabla f(Y_t^{x,(r)})$. For fixed $x \in \mathbb{R}^d$, choose a bounded and open neighbourhood $\mathcal{U}_x \subset \mathbb{R}^d$. 
 We observe that the argument establishing (\ref{eq:Y bound}) works for the process $Y^{y,(r)}$ for any $y \in U_x$ and $r>0$, and the moment bound of (\ref{eq:Y bound}) is independent of $r$. Combining this with the uniform bound on $\Vert \nabla _y Y_t^{y,(r)} \Vert_F^p$ from above, as well as the fact that $f \in C_P^l(\mathbb{R}^d)$, we see that 
\begin{subequations}
\begin{align}\label{bound.the.first.u.derivavative:12}
\mathbb{E} \left[ \Vert\nabla_y Y_t^{y,(r)} \nabla f(Y_t^{y,(r)})\Vert^2\right] & \le \left(\mathbb{E} [\Vert \nabla_y Y_t^{y,(r)} \Vert_F^4]\right)^{1/2} \left( \mathbb{E} [\Vert \nabla f(Y_t^{y,(r)})\Vert^4]  \right)^{1/2} \nonumber
\\
& \le C(1 + \Vert y \Vert^m), 
\end{align}
\end{subequations}
for appropriate constants $C>0$ and $m \in \mathbb{N}$, independent of $r$. Consequently the family of random variables $(\nabla_y Y_t^{(r)}\nabla f(Y_t^{(r)}))_{y \in \mathcal{U}_x}$ is uniformly integrable. %in the sense of \cite[Section 4.5]{bogachev2007measure}.
Therefore, combining Vitali's theorem with the proof technique from \cite[Corollary 2.8.7(ii)]{bogachev2007measure}, we can exchange differentiation and integration:
Indeed, since $x \mapsto f(Y_t^{x,(r)})$ is continuously differentiable \citep[Theorem 3.4.1]{kunita2019stochastic}, the mean-value theorem allows us to choose (random) $\xi_{h, v} \in \mathbb{R}^d$ so that
\begin{equation}
\tfrac{1}{h}\left(f(Y_t^{x + hv,(r)}) - f(Y_t^{x,(r)})\right) = v \cdot \nabla_x Y_t^{\xi_{h,v},(r)} \nabla f(Y_t^{\xi_{h,v},(r)}),
\end{equation}
for all $h>0$ and $v \in \mathbb{R}^d$ such that $x + hv \in \mathcal{U}_x$.  According to Vitali's theorem, we can then pass this expression to the limit and exchange differentiation and integration. This proves (\ref{eq:nabla u}). We also observe that a similar use of Vitali's theorem shows that $\nabla_x u_t^{(r)} \in C(U;\mathbb{R}^d)$.

\emph{Letting $r \rightarrow \infty$.} To remove the cut-off, we show that for any open and bounded set $\mathcal{U} \subset \mathbb{R}^d$ and fixed $t \in [0,T]$, the sequence $(u_t^{(r)})_{r \in \mathbb{N}}$ converges uniformly to $u_t$ on $\mathcal{U}$, and likewise $(\nabla_x u^{(r)}_t)_{r \in \mathbb{N}}$ is a   Cauchy sequence in $C(\mathcal{U};\mathbb{R}^d)$. Together, these claims show that $u_t$ is differentiable on $\mathcal{U}$, and the derivative coincides with the limit of $(\nabla_x u^{(r)}_t)_{r \in \mathbb{N}}$. Furthermore, since the upper bound on (\ref{bound.the.first.u.derivavative:12}) does not depend on $r$, we obtain the bound \eqref{eq:der u bound} for the limit $\nabla u_t$.

To show that $u_t^{(r)} \rightarrow u_t$, uniformly on $\mathcal{U}$, notice that
\begin{equation}
\sup_{x \in \mathcal{U}} |u_t^{(r)}(x) - u_t(x) | = \sup_{x \in \mathcal{U}} \left| \mathbb{E} \left[ \left(f(Y_t^{x,(r)}) - f(Y^{x}_t) \right) \mathbb{I}\left(\sup_{0 \le s \le t} \Vert Y^x_t \Vert \ge r\right)\right] \right|, 
\end{equation}
based on the fact that $Y_t^{(r)}$ and ${Y_t}$ can be coupled to coincide almost surely on ${\{\sup_{0 \le s \le t}\Vert Y^x_s\Vert \le r \}}$. Using Markov's inequality, the fact that $f \in C_P^l(\mathbb{R}^d)$ and the bound \eqref{eq:Y bound}, we see that up to a multiplicative constant, this expression can be bounded by the supremum over $x \in U$ of the quantity
\begin{equation}
\frac{1}{r} \mathbb{E} \left[ \sup_{0 \le s \le t} \Vert Y^x_s \Vert \right] \le \frac{1}{r} \mathbb{E} \left[ \Vert x \Vert + \int_0^t \Vert \mu(Y_s^x) \Vert \mathrm{d}s + \sup_{0 \le s \le t}\left\Vert \int_0^s \sigma(Y_{s'}^x) \, \mathrm{d}W_{s'}\right \Vert \right]. 
\end{equation}
According to the Burkholder-Davis-Gundy (and using the polynomial growth assumption on $\mu$ and $\sigma$), this expression indeed converges to zero as $r \rightarrow \infty$, uniformly in $x \in \mathcal{U}$.

 To show that $(\nabla_x u^{(r)}_t)_{r \in \mathbb{N}}$ is a   Cauchy sequence in $C(\mathcal{U};\mathbb{R}^d)$, %we show that for any fixed $\widetilde{r} \in \mathbb{R}$
%\begin{equation}
%\sup_{x \in \mathcal{U}}\left \Vert \nabla_x u^{(r)}_t(x) - \nabla_x u^{(r + \widetilde{r})}_t(x) \right \Vert \xrightarrow{r \rightarrow \infty} 0.
%\end{equation}
we observe that for any $r,\widetilde{r}>0$
\begin{subequations}
\begin{align*}
& \left \Vert \nabla_x u^{(r)}_t(x) - \nabla_x u^{(r + \widetilde{r})}_t(x) \right \Vert = \left\Vert \mathbb{E} \left[\nabla_xY_t^{x,(r)} \nabla f(Y_t^{x,(r)}) - \nabla_xY_t^{x,(r+ \widetilde{r})} \nabla f(Y_t^{x,(r + \widetilde{r})})  \right] \right\Vert 
\\
& =\left\Vert \mathbb{E} \left[\left(\nabla_xY_t^{x,(r)} \nabla f(Y_t^{x,(r)}) - \nabla_xY_t^{x,(r+ \widetilde{r})} \nabla f(Y_t^{x,(r + \widetilde{r})})  \right) \mathbb{I}\left({\sup_{0 \le s \le t}\Vert Y_t^{x,(r)}\Vert  \ge r}\right)\right]\right\Vert,
\end{align*}
\end{subequations}
Proceeding as above, and on letting $r \rightarrow \infty$, we get that $(\nabla_x u^{(r)}_t)_{r \in \mathbb{N}}$ is a   Cauchy sequence in $C(\mathcal{U};\mathbb{R}^d)$.

\emph{Estimates on $\nabla^n Y^x_t$.} To establish bounds on the higher-order derivatives, we proceed as before and consider the tamed SDE \eqref{eq:tamed sde}. Exchanging derivatives and expectations as well as removal of the cut-off $r \in \mathbb{N}$ works in the same way as for the first-order case. We therefore omit these details and lay out the main steps to establish bounds on moments of $\nabla_x^n Y_t^x$.

Before we write the SDE that is satisfied by $\nabla_x^nY_t^{x,(r)}$ for $n \geq 2$, we need to introduce the following notation:
For $k \in \mathbb{N}$ and for the vectors $\zeta^1,...,\zeta^k \in \mathbb{R}^d$, we write
\begin{equation*}
    \nabla^k \mu_i(x)\left[ \zeta^1,...,\zeta^k \right] := \sum_{j_1,\dots,j_k=1}^d\partial_{j_1} \partial_{j_2} \dots \partial_{j_k} \mu_i(x) \zeta^1_{j_1}\zeta^2_{j_2} \dots \zeta^k_{j_k}.
\end{equation*}
Fix a set of indices $\left\{ i_1,..., i_n \right\} \subset \left\{ 1,...,d \right\}$. For any set $B= \left\{ j_1,..., j_m \right\} \subset \left\{ i_1,..., i_n \right\} $ we write $\nabla^{B}Y_t^{x,(r)}$ to denote the $d$-dimensional vector with $i$-coordinate given by 
\begin{equation*}
    \nabla^{B}Y_{t,i}^{x,(r)}=\partial_{j_1} \dots \partial_{j_m} Y_{t,i}^{x,(r)}.
\end{equation*}
Now, let $\mathcal{A}^{k}_{i_1,\dots , i_n}$ be the set of all possible partitions of set $\{ i_1, \dots , i_n \}$ that are composed of $k$ subsets of  $\{ i_1, \dots , i_n \}$. For any such partition $A=\left\{ A_1, \dots, A_k \right\} \in \mathcal{A}^{k}_{i_1,\dots , i_n}$ we write
\begin{equation*}
    \nabla^k \mu_i(x)\left[ Y_t^{x,(r)}, A \right]:= \nabla^k \mu_i(x)\left[ \nabla^{A_1}Y_t^{x,(r)},\dots, \nabla^{A_k}Y_t^{x,(r)} \right].
\end{equation*}

Then, using an induction argument for $n \in \mathbb{N}$, along with \cite[Theorem 3.4.1]{kunita2019stochastic} and straightforward calculations, one can conclude the following. For $n \geq 2$, the SDE that $\nabla_x^nY_t^{x,(r)}$ satisfies is such that for any $i,i_1,\dots,i_n \in \left\{ 1,
\dots,d \right\}$, the drift of $\partial_{i_1} \dots \partial_{i_n} Y_{t,i}^{x,(r)}$ is given by
\begin{equation}\label{any.derivative.drift.sde:1}
    \nabla^1 \mu_{i}(Y_t^{x,(r)}) \left[ \nabla^{i_1,...,i_n}Y_t^{x,(r)} \right] + \sum_{k=2}^n \sum_{A \subset \mathcal{A}^{k}_{i_1,\dots , i_n}}\!\!\!\nabla^k \mu_i(Y_t^{x,(r)})\left[ Y_t^{x,(r)} , A \right] \mathrm{d}t,
\end{equation}
and the volatility is given by 
\begin{equation}\label{any.derivative.volatility.sde:1}
   \sum_{j=1}^d  \nabla^1 \sigma_{i,j}(Y_t^{x,(r)}) \left[ \nabla^{i_1,...,i_n}Y_t^{x,(r)} \right] + \sum_{k=2}^n \sum_{A \subset \mathcal{A}^{k}_{i_1,\dots , i_n}}\!\!\nabla^k \sigma_{i,j}(Y_t^{x,(r)})\left[ Y_t^{x,(r)} , A \right]  \mathrm{d}W^j_t.
\end{equation}
with initial conditions 
$\nabla^n_x Y^{x,(r)}_0  = 0$. We observe that the only terms that involve $n$'th order derivatives of $Y_t^{x,(r)}$ appear on the first term of (\ref{any.derivative.drift.sde:1}) and (\ref{any.derivative.volatility.sde:1}), with coefficients $\nabla \mu_i$ and $\nabla \sigma_{i,j}$ respectively. All the other terms involve lower order derivatives.

In complete analogy with \eqref{eq:ItoDY}, we apply It{\^o}'s formula to obtain an evolution equation for $\mathbb{E} [\Vert \nabla^n_x Y^{x,(r)}_t \Vert^p]$. The only terms in the It{\^o} expansion that will involve the quantity $\Vert \nabla^n_x Y^{x,(r)}_t \Vert^p$ come from the first terms of %pertaining to the contributions in
(\ref{any.derivative.drift.sde:1}) and (\ref{any.derivative.volatility.sde:1}), and these can be bounded in the same way as for $n=1$ (see the inequality below \eqref{eq:DY3}), making essential use of \eqref{eq:drift contracting} and the boundedness of $\nabla \sigma$. All the other terms in the  It{\^o} expansion, correspond to products between $\Vert \nabla^n_x Y^{x,(r)}_t \Vert^q$ for various $q < p$ and lower order derivatives of $Y^{x,(r)}_t$. The expectations of these contributions %corresponding to the lower-order terms in \eqref{eq:DnY2}
can be bounded inductively: First we use H{\"o}lder's inequality to bound the expectation of such product by the product between the expectation of $\Vert \nabla^n_x Y^{x,(r)}_t \Vert^p$ (raised to the power $\frac{q}{p}<1$) and a moment of lower order derivatives. These moments have been shown to be bounded uniformly in $r$ and $t$ in the previous steps, where we consider lower order derivatives. Finally, any moment of $\nabla^k\sigma(Y_t^{x,(r)})$ or $\nabla^k\mu(Y_t^{x,(r)})$ for any $k=0,...,n$ can be bounded using boundedness of derivatives of $\sigma$ and the polynomial bounds on derivatives of $\mu$, in conjuction with the the bound on the the moments of $Y_t^{x}$ established in (\ref{eq:Y bound}). This shows the uniform bound (over $t,r$) of all moments of $n$'th order derivatives, given the bounds for order $k=0,...,n-1$.
\end{proof}

\section{Proof of Theorem \ref{theorem.weak.convergence:0}}\label{proof.weak.convergence:000}

We start with the following observation.

\begin{lemma}\label{a.small.lemma:1}
 Let $(u_n)_{n \in \mathbb{N}}$ be a sequence of non-negative real numbers. Assume that there exist $a>1$, $b>0$ such that for all $n \geq 1$
 \begin{equation}\label{a.small.lemma.assumption:1}
     u_n \leq a \ u_{n-1} + b.
 \end{equation}
 Then, for all $n \in \mathbb{N}$,
 \begin{equation}\label{a.small.lemma.result:1}
     u_n \leq a^{n} u_0 + b \frac{a^n-1}{a-1}.
 \end{equation}
\end{lemma}

\begin{proof}[Proof of Lemma \ref{a.small.lemma:1}.]
   The proof proceeds by induction. For $n=0$ the result is trivial since $a>1,b>0$. For $n=1$, the result follows from (\ref{a.small.lemma.assumption:1}). Assume that for some $n \geq 1$, (\ref{a.small.lemma.result:1}) holds. Using (\ref{a.small.lemma.assumption:1})for $n+1$, we directly calculate,
   \begin{align*}
       u_{n+1} &\leq a \ u_n +b \leq a \left( a^{n} u_0 + b \frac{a^n-1}{a-1} \right)+ b = a^{n+1} u_0 + a b \frac{a^n-1}{a-1}+b \\ & = a^{n+1} u_0 + b \left( \frac{a^{n+1}-a}{a-1} + 1 \right) =  a^{n+1} u_0 + b \frac{a^{n+1}-1}{a-1},  
   \end{align*}
   which proves (\ref{a.small.lemma.result:1}) for $n+1$ and completes the induction.
\end{proof}

In order to prove Theorem \ref{theorem.weak.convergence:0} we also need the following intermediate result.

\begin{lemma}\label{weak.conv.third.point.proof:1}
    Assume that Assumptions  \ref{ass.drift.general:1}, \ref{ass.volatility.general:1}, \ref{assumption.p.derivative:001} and \ref{regularity.coefficients:1} hold. Fix a time horizon $T > 0$. For all $x \in \mathbb{R}^d$ and $m \in \mathbb{N}$, there exists $C=C(T,m,x)$ such that for all $\Delta t \in (0,1)$ and $k \leq N=\left \lfloor{\frac{T}{\Delta t}}\right \rfloor$, 
   \begin{equation}\label{moments.control:1}
       \mathbb{E}_x\left[  \| X^k \|^m\right] \leq C.
   \end{equation}
\end{lemma}

\begin{proof}[Proof of Lemma \ref{weak.conv.third.point.proof:1}]
First of all, we observe that it suffices to prove (\ref{moments.control:1}) for all $m \in \mathbb{N}$ even. The result generalises to $m$ odd, using, for example, H\"older's inequality.

Recall that $X^1=x+\Delta t^{1/2} \sigma(x) \left(b \ast \nu \right)$, for $b \in \{ \pm 1  \}^d$ and $\nu \sim N(0,I_d)$. Therefore, for any $m \in \mathbb{N}$,
\begin{align*}
    &\| X^1 \|^{2m}= \left( \left(x_1+\Delta t^{1/2} (\sigma(x) \left(b \ast \nu \right))_1 \right)^2 +  ... + \left(x_d+\Delta t^{1/2} (\sigma(x) \left(b \ast \nu \right))_d \right)^2 \right)^m \\
    &=\left(  \| x \|^2 + \Delta t^{1/2} 2 \sum_{i,j=1}^d x_i\sigma_{i,j}(x)b_j\nu_j +\Delta t \| \sigma(x) (b \ast \nu) \|^2 \right)^m,
\end{align*}
    which is clearly a polynomial with respect to $\Delta t^{1/2}$. Since $\nu \sim N(0,I_d)$ has moments of all orders, by taking expectations of the above expression we see that $\mathbb{E}_x\left[ \| X^1\|^{2m} \right]$ is also a polynomial with respect to $\Delta t^{1/2}$. The constant term of the polynomial is equal to $\| x \|^{2m}$. Furthermore, using the test function $f(x)=\| x \|^{2m}$ in Proposition \ref{semigroup.expansion:1} and letting $\Delta t \rightarrow 0$ shows that the first order term of the polynomial (i.e. the coefficient of $\Delta t^{1/2}$) is zero. By Assumption \ref{ass.volatility.general:1}, there exists $c_1>0$ such that for all $i,j \in \{ 1,...,d \}$ we have that $ \sigma_{i,j}(x)  \leq c_1 (1+ \| x \|)$. This implies that all the other coefficients of the polynomial are of order $O(1+ \| x \|^{2m})$ with respect to $x$. Overall this means that there exists a constant $c=c(m)$ such that for all $x \in \mathbb{R}^d$,
\begin{equation*}
    \mathbb{E}_x\left[ \| X^1 \|^{2m} \right] \leq \|  x \|^{2m} + \Delta t \ c \|  x \|^{2m} + \Delta t \ c = \| x \|^{2m} \left( 1+ c \ \Delta t  \right) + c \ \Delta t.
\end{equation*}
Therefore, using the Markov property, for any $k=0,..., N-1$,
\begin{equation*}
    \mathbb{E}_x\left[ \| X^{k+1} \|^{2m} \right] \leq \mathbb{E}_x \left[ \|  X^{k} \|^{2m} \right] \left( 1+ c \Delta t  \right) + c \ \Delta t.
\end{equation*}
Using Lemma \ref{a.small.lemma:1}, we get that for all $k=0,...,N$ and $x \in \mathbb{R}^d$,
\begin{align*}
    \mathbb{E}_x\left[ \| X^{k} \|^{2m} \right] &\leq \| x \|^{2m} \left( 1+ c \ \Delta t  \right)^{k} + c \ \Delta t \frac{(1+ c \ \Delta t)^k-1}{1+ c \ \Delta t -1} \\
    &\leq \| x \|^{2m} \left( 1+ c \Delta t  \right)^{N}+ \left( 1+ c \Delta t  \right)^{N}.
\end{align*}
Recall that $N=\left \lfloor{\frac{T}{\Delta t}}\right \rfloor \leq \frac{T}{\Delta t}$, so $\Delta t \leq \frac{T}{N}$. This implies that
\begin{equation*}
    \mathbb{E}_x\left[ \| X^{k} \|^{2m} \right] \leq \| x \|^{2m} \left( 1+  \frac{cT}{N}  \right)^{N} + \left( 1+  \frac{cT}{N}  \right)^{N}  \leq \left( \| x \|^{2m} +1 \right) \exp\{  cT\},
\end{equation*}
which completes the proof.
\end{proof}

\begin{proof}[Proof of Theorem \ref{theorem.weak.convergence:0}]
From Lemma \ref{weak.conv.third.point.proof:1}, Assumption d) of the general convergence theorem of Milstein \& Tretyakov \cite[Theorem 2.2.1]{milstein.tretyakov:21} holds. From Proposition \ref{semigroup.expansion:1}.1, equation (2.2.7) of \cite[Theorem 2.2.1]{milstein.tretyakov:21} holds with $p=1$, while from Theorem \ref{Generalisation.theorem:1}, (2.2.7) holds for any test function $u$ of the form (\ref{def.u:01}). The rest of the proof of \cite[Theorem 2.2.1]{milstein.tretyakov:21} applies and the result follows.
\end{proof}

\section{Proof of Theorem \ref{thm:geometric_ergodicity}}\label{proof.geom.ergo:000}

We begin with the following proposition.

\begin{proposition} \label{prop:small_set}
Under Assumption \ref{ass:long_time} the Markov chain with transition $P$ associated with the skew-symmetric numerical scheme is $\lambda_\text{Leb}$-irreducible and aperiodic, and all compact sets are small. In addition the following holds for fixed $s>0$ and setting $V(x) := e^{s\|x\|_{\infty}}$: For any compact set $C \subset \mathbb{R}^d$ 
there exists a constant $K_{s,C} >0$ such that
\begin{equation} \label{eq:lyapunov_compact}    \int V(y) P(x,\mathrm{d}y) < K_{s,C},
\end{equation}
for all $x \in C$.
\end{proposition}
\begin{proof}[Proof of Proposition \ref{prop:small_set}]
    Note that by the skew-symmetric nature of the increments, and by Assumption \ref{ass:long_time}.4, the transition $P$ has a Lebesgue density that can be written $\gamma:=\prod_{i=1}^d \gamma_i$, where each 
    \begin{align*}
    &\gamma_i(x, x_i + z_i) \\  &=\left(g_i\left(x,\frac{z_i}{\sqrt{\Delta t}\sigma_{i,i}(x)}\right) + 1 - g_i\left(x,\frac{-z_i}{\sqrt{\Delta t}\sigma_{i,i}(x)}\right) \right) \frac{1}{\sqrt{\Delta t}\sigma_{i,i}(x)}\phi\left(\frac{z_i}{\sqrt{\Delta t}\sigma_{i,i}(x)} \right), 
    \end{align*}
    with $\phi$ being the standard Gaussian density function on $\mathbb{R}$. Under Assumption \ref{ass:long_time} each $g_i(x,\nu_i)$ is bounded away from 0 on compact sets and $\sigma_{ii}(x) \leq M$, which shows $\lambda_\text{Leb}$-irreducibility and aperiodicity.  It also implies that, by fixing a threshold value $A>0$, for any compact $C \subset \mathbb{R}^d$ there is an $\epsilon_{C} > 0$ such that for any $x \in C$
    \begin{equation*}
    \gamma_i(x ,x_i + z_i) \geq \epsilon_C \phi\left(\frac{z_i}{M}\right) \mathbb{I}_{|z_i| \leq A}
    \end{equation*}
    for all $i$ (noting also that $\phi$ is monotone on $[0,\infty)$). The standard Gaussian distribution on $\mathbb{R}^d$, constrained on $[-A,A]^d$, can therefore be taken as a minorising measure for $P$ on $C$ using the constant $\epsilon^d_C$. Since this argument holds for any compact $C$ then this shows that any compact set is small.  Finally using the skew-symmetric representation of the transition density it follows that for any $x \in C$,
    \begin{equation*}
    \int V(y)P(x,dy)  \leq \frac{2^d}{N^d \Delta t^{d/2}}\int V(y) \phi\left(\frac{y_1-x_1}{\sqrt{\Delta t}\sigma_{1,1}(x)}\right)...\phi\left(\frac{y_d-x_d}{\sqrt{\Delta t }\sigma_{d,d}(x)}\right)dy,
    \end{equation*}
    which is straightforwardly bounded above by some finite constant $K_{s,C}$ using properties of the Gaussian distribution. This completes the proof.
\end{proof}

\begin{lemma} \label{lem:1d_drift} 
Under Assumption \ref{ass:long_time}, setting $Y_i = x_i + b_i\sqrt{\Delta t}\sigma_{i,i}(x)\nu_i$ for each $i \in \{1,...,d\}$, and for any $s > 4d/\left(N\sqrt{\Delta t 2 \pi}\right)$,
\begin{equation} \label{eq:1d_drift}
    \limsup_{\|x\| \to \infty} \mathbb{E} [\exp(s|Y_i|-s|x_i|)]\mathbb{I}(|x_i|>\|x\|_\infty/2) < \frac{1}{d}.
\end{equation}
\end{lemma}

Intuitively Lemma \ref{lem:1d_drift} ensures that any coordinate far enough away from the centre of the space will tend to drift back towards it.  The constant $1/d$ ensures that the respective terms sum to less than one when summing over the coordinates.

\begin{proof}[Proof of Lemma \ref{lem:1d_drift}] 
Note that the stipulation $|x_i|>\|x\|_\infty/2$ and the equivalence of norms on $\mathbb{R}^d$ implies that we need only consider the situation in which $|x_i| \to\infty$, as in any other case the limit is 0.  We give details when $x_i \to \infty$ (and therefore consider $x_i>0$ here). The case $x_i \to - \infty$ is analogous.  Writing $Y_i = x_i + b_i \sqrt{\Delta t} \sigma_{i,i}(x) \nu_i$, we have that $\mathbb{E} [\exp(s|Y_i|-s|x_i|)\mathbb{I}(x_i>\|x\|_\infty/2)]$ is equal to
\begin{equation*}
\mathbb{E}_{\nu_i} \mathbb{E}_{b_i|\nu_i} [e^{s|x_i + b_i \sqrt{\Delta t} \sigma_{i,i}(x) \nu_i| - s|x_i|}\mathbb{I}(x_i>\|x\|_\infty/2)],
\end{equation*}
which upon further simplification can be written $A(x) + B(x)$, where
\begin{equation*}
\begin{aligned}
A(x) &:= \mathbb{E}_{\nu_i} [e^{s|x_i + \sqrt{\Delta t} \sigma_{i,i}(x) \nu_i| - s|x_i|} g_i(x, \nu_i) \mathbb{I}(x_i>\|x\|_\infty/2)]
\\
B(x) &:= \mathbb{E}_{\nu_i} [ e^{s|x_i - \sqrt{\Delta t} \sigma_{i,i}(x) \nu_i| - s|x_i|} (1 - g_i(x, \nu_i)) \mathbb{I}(x_i>\|x\|_\infty/2)].
\end{aligned}
\end{equation*}

We consider the behaviour of $A(x)$ and $B(x)$ in the limit superior.  In both cases the limit superior can be brought inside the expectation by the reverse Fatou lemma using the dominating function $e^{\sqrt{\Delta t}M s | \nu_i| }$, which is integrable since $\nu$ is Gaussian.  We have 
\begin{equation*}
\begin{aligned}
\limsup_{\|x\| \to \infty} A(x)  
&\leq  
\mathbb{E}_{\nu} [\limsup_{\|x\| \to \infty} e^{s|x_i + \sqrt{\Delta t}\sigma_{i,i}(x) \nu_i| - s|x_i|} g_i(x, \nu_i) \mathbb{I}(x_i>\|x\|_\infty/2)] 
\\
&=  
\mathbb{E}_{\nu} [e^{s\sqrt{\Delta t}\sigma_{i,i}(x) \nu_i}\limsup_{\|x\| \to \infty} g_i(x, \nu_i)\mathbb{I}(x_i>\|x\|_\infty/2)].
\end{aligned}
\end{equation*}
If $\nu_i > 0$ then $\nu_i\cdot\mu_i(x) \to -\infty$ under the restriction on $x_i$ by Assumption \ref{ass:long_time}.3, which implies that $g_i(x,\nu_i) \to 0$ by Assumption \ref{ass:long_time}.4. When $\nu_i < 0$ then by the same argument $g_i(x,\nu_i) \to 1$. Writing $\tilde{s} := s\sqrt{\Delta t} N$, it follows that
\begin{align*}
\limsup_{\|x\|\to \infty} A(x) &
\leq \frac{1}{\sqrt{2\pi}}\int_{-\infty}^0 e^{s\sqrt{\Delta t}\sigma_{i,i}(x) \nu_i - \nu_i^2/2}d\nu_i \leq 
\frac{1}{\sqrt{2\pi}}\int_{-\infty}^0 e^{\tilde{s}\nu_i - \nu_i^2/2}d\nu_i \\
&= \frac{1}{2}e^{\tilde{s}^2/2}\text{erfc}\left( \frac{\tilde{s}}{\sqrt{2}}\right) 
\leq \frac{2}{\tilde{s}\sqrt{2\pi}},
\end{align*}
where $\text{erfc}(t)$ denotes the complementary error function evaluated at $t \in \mathbb{R}$, and the last inequality is obtained using standard upper bounds from the literature (e.g. \cite{abramowitz1948handbook}).

By a similar argument, for $B(x)$ we have
\begin{equation*}
\begin{aligned}
\limsup_{\|x\|\to\infty}B(x)
%&\leq 
%\mathbb{E}_{\nu_i}[e^{-\tilde{s}\nu_i}(1-\limsup_{\|x\|\to\infty}g_i(x,\nu_i))]\mathbb{I}(x_i>\|x\|_\infty/2)
%\\
%&
\leq 
\frac{1}{\sqrt{2\pi}}\int_0^\infty e^{-\tilde{s}\nu_i - \nu_i^2/2}d\nu_i 
\leq
\frac{2}{\tilde{s}\sqrt{2\pi}}.
\end{aligned}
\end{equation*}
Combining and applying the same argument to the case $x_i \to -\infty$ gives
\begin{equation*}
\limsup_{\|x\| \to \infty}\mathbb{E} [\exp(s|Y_i|-s|x_i|)]\mathbb{I}(|x_i|>\|x\|_\infty/2) \leq \frac{4}{\tilde{s}\sqrt{2\pi}},
\end{equation*}
which will be $< 1/d$ since $s > 4d/\left(N\sqrt{\Delta t 2 \pi}\right)$. This completes the proof.
\end{proof}

\begin{proposition} \label{prop:drift_condition}
Let $s > 4d/\left(N\sqrt{\Delta t 2 \pi}\right)$ and $V(x) := \exp(s\| x \|_\infty)$. Under Assumption \ref{ass:long_time} the Markov chain produced by the $d$-dimensional skew-symmetric scheme with step-size $\Delta t$ satisfies 
\begin{equation*}
\limsup_{\| x\| \to \infty} \int \frac{V(y)}{V(x)} P(x,dy) < 1.
\end{equation*}
\end{proposition}

\begin{proof}[Proof of Proposition \ref{prop:drift_condition}]
Let $x \in \mathbb{R}^d$ and $Y_i = x_i + b_i\sqrt{\Delta t}\sigma_{i,i}(x)\nu_i$ for $i \in\{1,...,d\}$. Since $V(y) \le \sum_{i = 1}^d \exp(s|y_i|)$, we have 
\begin{equation*}
\mathbb{E} \left[ \frac{V(Y)}{V(x)}  \right] \le \sum_{i=1}^d \mathbb{E} \left[ \exp (s | Y_i | - s \| x\|_\infty ) \right].
\end{equation*}
Let $i \in \{ 1,...,d \}$. We will consider two cases: (1) $|x_i| \le \| x\|_\infty /2 $ and (2) $\| x\|_\infty /2 < |x_i|$.  When $|x_i| \le \| x\|_\infty /2$,
\begin{equation*}
|Y_i| = |x_i + b_i \sqrt{\Delta t}\sigma_{i,i}(x)  \nu_i| \le |x_i| +  | b_i \sqrt{\Delta t} \sigma_{i,i}(x) \nu_i| \leq  \sqrt{\Delta t} M | \nu_i| +  \| x\|_\infty / 2,
\end{equation*}
which implies
\begin{align*}
&\limsup_{\| x\| \to \infty} \mathbb{E} \left[ \exp (s | Y_i | - s \| x\|_\infty ) \right]\mathbb{I}(|x_i|\leq \|x\|_\infty/2) \\
&\leq \limsup_{\| x\| \to \infty} e^{-  s\| x\|_\infty / 2}\mathbb{E} \left[ \exp (s \sqrt{\Delta t}M | \nu_i|) \right] = 0.
\end{align*}
When $\| x\|_\infty /2 < |x_i| \leq \|x\|_{\infty}$ we have 
\begin{equation*}
\mathbb{E} \left[ \exp (s | Y_i | - s \| x\|_\infty ) \right]\mathbb{I}(|x_i|> \|x\|_\infty/2) \le  \mathbb{E} \left[ \exp (s | Y_i | - s | x_i| )\right]\mathbb{I}(|x_i| > \|x\|_\infty/2).
\end{equation*}
Since $s > 4d/(N\sqrt{\Delta t 2 \pi})$, by Lemma \ref{lem:1d_drift} we get
\begin{equation*}
\limsup_{\| x\| \to \infty} \mathbb{E} \left[ \exp (s | Y_i | - s \| x\|_\infty ) \right]\mathbb{I}(|x_i| > \|x\|_\infty/2) 
< \frac{1}{d},
\end{equation*}
from which we can conclude.
\end{proof}

\begin{proof}[Proof of Theorem \ref{thm:geometric_ergodicity}]
 We prove the theorem using the drift and minorisation approach popularised in \cite{meyn2012markov}, see e.g. Chapters 15-16 of that work. The main theorem we use is also stated as Theorem 2.25 in \cite{lelievre2016partial}, and a concise proof is provided in \cite{hairer2011yet}. In Proposition \ref{prop:small_set} we establish that the Markov chain is $\lambda_{\text{Leb}}$-irreducible and aperiodic, where $\lambda_\text{Leb}$ denotes the Lebesgue measure on $\mathbb{R}^d$, and that every compact set $C \subset \mathbb{R}^d$ is small in the 
sense that there exists a Borel probability measure $\varphi$ on $\mathbb{R}^d$ and a positive constant $\alpha_C$ such that $\inf_{x \in C}P(x,\cdot) \geq \alpha_C \varphi(\cdot)$. This establishes Assumption 2.23 in \cite{lelievre2016partial} for any compact $C$.  In Proposition \ref{prop:drift_condition} (which relies on Lemma \ref{lem:1d_drift}), we establish a geometric drift condition using the Lyapunov function $V(x) := e^{s\|x\|_\infty}$ in the limit as $\|x\| \to\infty$ for suitably large $s>0$.  Combining this with \eqref{eq:lyapunov_compact}, which is proven in Proposition \ref{prop:small_set}, shows that, by choosing a large enough compact set, and a large enough $s>0$, the inequality $\int V(y) P(x,\mathrm{d}y) \leq \lambda V(x) + K$ can be established, for some $\lambda<1$ and $K>0$, and all $x \in \mathbb{R}^d$. This means that Assumption 2.24 of \cite{lelievre2016partial} is satisfied, therefore Theorem 2.25 of the same work can be applied, giving the result.
\end{proof}

\section{Proof of Theorem \ref{thm.bias.equilibrium:1}}\label{proof.bias.equilibrium:00}

\begin{proof}
We verify the assumptions of \cite[Theorem 3.3]{lelievre2016partial} from which the result follows (see also \cite{leimkuhler2016computation} and \cite{abdulle2015long} for previous similar results). In particular, $C^{\infty}_P(\mathbb{R}^d)$ will play the role of $S$ in \cite[Theorem 3.3]{lelievre2016partial}, which is dense in $L^2(\pi)$ since all moments of $\pi$ are finite due to Assumption \ref{stoltz.ass.pi:1}.

First, by Proposition \ref{semigroup.expansion:1}, equation (3.15) of \cite{lelievre2016partial} holds with $p=1$.
Then, from the regularity properties of $\mu$ and $\sigma$ (Assumptions \ref{ass.drift.general:1} and \ref{ass.volatility.general:1}) it follows that if $L$ and $\mathcal{A}_2$ are as in (\ref{generator.diffusion:1}) and (\ref{second.derivative.numerical.scheme:1}) respectively, then for any $f \in C^{\infty}_P(\mathbb{R}^d)$, we have $Lf , \mathcal{A}_2f \in C^{\infty}_P(\mathbb{R}^d)$.  Furthermore, since $\pi$ is invariant for the diffusion (\ref{multi.dimensional.sde:0}), for any $f \in C^{\infty}_P(\mathbb{R}^d)$ it holds that $\int Lf(x)\pi(\mathrm{d}x)=0$. Therefore, $C^{\infty}_{P,0}(\mathbb{R}^d)= \left\{  f \in C^{\infty}_P(\mathbb{R}^d): \int f \,\mathrm{d} \pi=0 \right\}$ is invariant under $L$.
Moreover, if $L^{\ast}$ denotes the formal $L^2(\pi)$-adjoint of $L$, then $L^{\ast}=L$ as the diffusion is time-reversible. Hence, by \cite[Theorem 1]{pardoux.veretennikov:01} together with \cite[Theorem 9.19]{gilbarg1977elliptic}, we have that $L$ is invertible on $C^{\infty}_{P,0}(\mathbb{R}^d)$.

Since $\log \pi \in C^4_P(\mathbb{R}^d)$, we get that for all $k=0,1,...,4$, if $\pi^{(k)}$ is any partial derivative of $\pi$ of order $k$, then $\pi^{(k)}/\pi$ is polynomially bounded. Then, using integration by parts, we directly calculate that for $\mathcal{A}_2^{\ast}$, the formal $L^2(\pi)$-adjoint of $\mathcal{A}_2$, and for the function ${\bf 1}: \mathbb{R}^d \rightarrow \mathbb{R}$ equal to 1 everywhere, we have 
\begin{align*}
    &\mathcal{A}_2^{\ast}{\bf 1} (x)= -\sum_{i,k=1}^d \frac{\partial_i\left[ \partial_i^{\xi}\partial_k^{\xi}\partial_k^{\xi} q_i(x,0)  \sigma_{i,i}(x)  \pi(x) \right]}{\pi(x)} +
     \\
     & - \sum_{i \neq k}\frac{\partial_i\partial_k\left[ \mu_i(x) \mu_k(x) \pi(x) \right]}{\pi(x)} -\sum_{i \neq k} 4 \frac{\partial_i \partial_k \left[ \partial_k^{\xi}q_i(x,0) \partial_i^{\xi} q_k(x,0)  \sigma_{i,i}(x)   \sigma_{k,k}(x) \pi(x) \right]}{\pi(x)} \\
     &- \sum_{i,k=1}^d \frac{\partial_i \partial_k^2 \left[ \mu_i(x) \sigma_{k,k}(x)^2 \pi(x) \right]}{2 \pi(x)} -\sum_{i,k=1}^d \frac{\partial_i^2\partial_k^2\left[ \sigma_{i.i}(x)^2  \sigma_{k,k}(x)^2 \pi(x) \right]}{8\pi(x)},
\end{align*}

implying that $\mathcal{A}_2^{\ast}{\bf 1} \in C^{\infty}_P(\mathbb{R}^d)$ by Assumptions \ref{ass.drift.general:1}, \ref{ass.volatility.general:1}, \ref{regularity.coefficients:1} and \ref{stoltz.ass.pi:1}.
Observe also that 
\begin{equation*}
\int \mathcal{A}_2^{\ast}{\bf 1} \, \mathrm{d}\pi = \langle \mathcal{A}_2^{\ast}{\bf 1} , {\bf 1} \rangle_{\pi} = \langle {\bf 1} , \mathcal{A}_2{\bf 1} \rangle_{\pi} =0,
\end{equation*}
since $\mathcal{A}_2$ is a differential operator. From this, together with the expression for $\mathcal{A}_2^{\ast}{\bf 1} (x)$, it follows that 
$
\mathcal{A}_2^{\ast}{\bf 1} \in C^{\infty}_{P,0}(\mathbb{R}^d)
$.
Finally, from Theorem \ref{thm:geometric_ergodicity}, the skew-symmetric numerical scheme admits a unique invariant probability measure $\pi_{\Delta t}$, with all moments finite, meaning that for all $n \in \mathbb{N}$,
\begin{equation*}
\int \left( 1 + \| x \|^n \right) \pi_{\Delta t}(\mathrm{d}x) < \infty.
\end{equation*}
Each assumption of \cite[Theorem 3.3]{lelievre2016partial} is therefore satisfied, and the result follows.
\end{proof}

\clearpage % optional: force new page
\section*{Supplementary Material A: Comparison between the skew-symmetric scheme with other schemes}

In this section we will present some conceptual comparisons between the skew-symmetric numerical scheme and other state-of-the-art explicit schemes for stochastic differential equations with non-Lipschitz drift.

Assume for this section that $p$ is the form of Proposition 2.1, where $F$ and $f$ are the cumulative distribution function and probability density function respectively of the standard Gaussian distribution. Then in the one dimensional setting, the skew-symmetric scheme can be written as
$X^{n+1} = X^n + \sqrt{\Delta t}\cdot \sigma(X^n) \cdot \gamma_n$,
where $\gamma_n$ is a skew-Normal random variable. With this choice, the mean increment is explicit \citep{azzalini2013skew}, given by
\begin{equation*}
\mathbb{E}[X^{n+1} - X^n] = 
\Delta t \cdot \frac{\mu(X^n)}{\sqrt{1+\Delta t \cdot \frac{\pi}{2} \cdot \frac{\mu(X^n)^2}{\sigma(X^n)^2}}}.
\end{equation*}
This allows some comparison with the tamed and adaptive Euler schemes, although it should be noted that the random variable $X^{n+1}-X^n$ is also skewed in the direction of $\mu/\sigma$ for the skew-symmetric scheme.

\subsection*{The Langevin case}

When $\mu(x) = -U'(x)$ and $\sigma(x) = \sqrt{2}$, for large values of $|U'(x)|$, we have
\begin{equation*}
\mathbb{E}[X^{n+1} - X^n] \approx -\frac{2}{\sqrt{\pi}}\cdot \sqrt{\Delta t} \cdot \frac{U'(X^n)}{|U'(X^n)|}.
\end{equation*}
This is close to the tamed Euler scheme, in which the mean increment will be approximately 
$
-(\Delta t)^{1-\alpha} \cdot U'(X^n_T) / |U'(X^n_T)|
$
under the same setting. The standard variant of taming sets $\alpha = 1$, but \cite{sabanis2013note} argues that $\alpha = 1/2$ is also natural and shows its validity. The schemes therefore appear similar, however it is worth highlighting that in the skew-symmetric case as $U'(x) \to \infty$ then $X^{n+1} - X^n$ converges to a truncated Gaussian, meaning that $\mathbb{P}(X^{n+1} - X^n < 0) \to 1$,
whereas in the tamed Euler case, in the same limit, the increment will converge in distribution to
$
N(-(\Delta t)^{1-\alpha}, \Delta t),
$
which could still be positive for fixed $\Delta t$.

\subsection*{Linear volatility}

Assume instead that $\sigma(x) = |x|$ when $|x| > \epsilon$. Then for large $|\mu(X_n)|$,
\begin{equation*}
\mathbb{E}[X^{n+1} - X^n] \approx
\sqrt{\Delta t}\cdot \sqrt{\frac{2}{\pi}} \cdot \frac{\mu(X^n)}{|\mu(X^n)|} \cdot |X^n|.
\end{equation*}
This is close to the adaptive Euler scheme with a choice of step-size
$
h(X^n) = \Delta t \cdot |X^n| / |\mu(X^n)|.
$
Similar choices are recommended in \cite{fang2020adaptive}.  To illustrate benefits over the tamed Euler scheme in this setting, consider the stochastic differential equation
$
dY_t = aY_tdt + bY_tdW_t$ with $Y_0 = x_0 > 0,
$
which has explicit solution given by
$
Y_t = x_0 \exp\left\{ at - \frac{1}{2}b^2t + bW_t\right\}
$
(see \cite{skiadas2010exact}).  Since $\mathbb{E}[e^{bW_t}] = e^{\frac{1}{2}b^2t}$, it follows that
\begin{equation*}
\mathbb{E}[Y_{\Delta t} - x_0]  = \Delta t \cdot a x_0  + O(\Delta t^2).
\end{equation*}
A single step of the skew-symmetric scheme satisfies
\begin{equation*}    
\mathbb{E}[X^1 - x_0] = \Delta t \cdot \frac{a x_0}{\sqrt{1+ \Delta t \cdot \frac{\pi}{2} \cdot \frac{a^2}{b^2} }}.
\end{equation*}
The tamed Euler scheme (denoted by $Z^1,Z^2,\dots$) satisfies
\begin{equation*}
\mathbb{E}[Z^1 - x_0] = \Delta t \cdot \frac{a x_0}{1+\Delta t^\alpha |a x_0|}
\end{equation*}
for some taming parameter $\alpha >0$.  Comparing schemes we see that to leading order in $\Delta t$,
$
\mathbb{E}[Y_{\Delta t} - x_0] / \mathbb{E}[X^1 - x_0]
$
is constant with respect to $x_0$, whereas
$
\mathbb{E}[Y_{\Delta t} - x_0] / \mathbb{E}[Z^1 - x_0]
$
grows linearly.

\begin{remark} 
In \cite{livingstone2022barker} the authors compared algorithms that can be viewed as simulating overdamped Langevin diffusions using the tamed Euler scheme, a particular version of adaptive step-size Euler and the skew-symmetric schemes (called the Barker proposal in that work), and then applying a Metropolis--Hastings correction.  It was found that the skew-symmetric scheme facilitated extremely fast convergence of adaptive tuning parameters that are employed within Metropolis--Hastings, allowing the resulting Markov chain to mix much more quickly than in the other two cases, which is desirable in the context of that work. The numerical results and accompanying discussion can be found in Section 8.3 of the supplement to \cite{livingstone2022barker}.
\end{remark}

\section*{Supplementary Material B: Illustration of weak order \& equilibrium bias
}

We illustrate the behaviour of the skew-symmetric scheme when simulating an Ornstein--Uhlenbeck process $\mathrm{d}Y_t = \theta (\mu - Y_t) \mathrm{d}t + \sigma \mathrm{d}W_t$, as compared to the Euler--Maruyama scheme and the tamed Euler approach of \cite{hutzenthaler2012strong}. Performance is assessed by considering the absolute difference between the exact solution $\mathbb{E}_x[f(Y_T)]$ and a Monte Carlo average of numerical values of $f(X_T)$. Given that each scheme is weakly first order, the logarithm of the absolute error as compared to the truth should grow approximately linearly with $\log(\Delta t)$.  We chose to compare estimates at time $T = 5$ of the simulated path using test function $f(x) = x^2$. The diffusion parameters were set to $\mu = 0$, $\theta = 1$, and $\sigma = \sqrt{2}$, with initial value for all simulations set to be $x = 1$.  Results are shown on the left-hand side of Figure \ref{fig:OU-simulation}, which indicates that all schemes behave similarly for small $\Delta t$. 

To assess equilibrium bias the right-hand side of Figure \ref{fig:OU-simulation} shows the absolute error in computing the 2nd, 4th and 6th moments of a distribution on $\mathbb{R}$ with density $\pi(x) \propto e^{-x^4/4}$. The true moments were computed using a quadrature scheme at high precision to be 0.676, 1.000 and 2.028 respectively. Approximations were computed using the skew-symmetric scheme by simulating the diffusion $\mathrm{d}Y_t = -Y_t^3\mathrm{d}t + \sqrt{2}\mathrm{d}W_t$ initialised close to equilibrium for 10,000 iterations and computing ergodic averages. In each case the error is shown to grow linearly with step-size in accordance with theory. We note that the Euler--Maruyama scheme is transient in this setting.

\begin{figure}[h!]  
\centering
  \includegraphics[width=12.3cm]{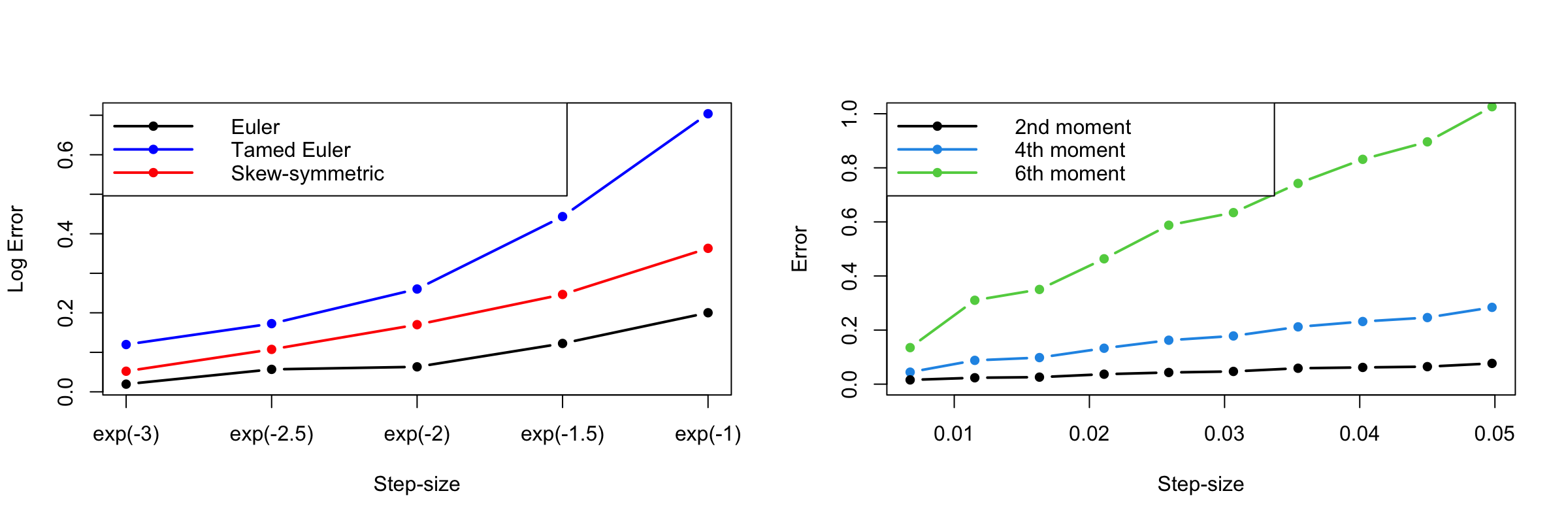}%was13cm
  \caption{Figure on the left-hand side shows logarithm of the absolute error when computing $\mathbb{E}_x[Y_T^2]$ for $T$ fixed for the Ornstein--Uhlenbeck process. Figure on the right-hand side shows absolute error when computing moments of the density $\pi(x) \propto e^{-x^4/4}$ by simulating the overdamped Langevin diffusion over long time scales.}
  \label{fig:OU-simulation}
\centering
\end{figure}

%USE THE BELOW OPTIONS IN CASE YOU NEED AUTHOR YEAR FORMAT.
%\bibliographystyle{abbrvnat}
%\bibliography{reference}

\end{document}